\documentclass[11pt]{article}
\pdfoutput=1
\usepackage{amsthm}
\usepackage[margin=1.206in]{geometry} 
\usepackage[latin1]{inputenc}
\usepackage{algorithm, algorithmic}
\usepackage{libertine}
\usepackage[libertine]{newtxmath}
\usepackage[ruled,vlined,algo2e]{algorithm2e}
\usepackage{amsmath}
\usepackage{amssymb}
\usepackage{graphicx,epstopdf} 
\usepackage[caption=false]{subfig}
\usepackage[dvipsnames]{xcolor}
\usepackage[citecolor=NavyBlue,colorlinks=true,linkcolor=NavyBlue,urlcolor=NavyBlue]{hyperref}
\usepackage{amscd}%
\usepackage{graphicx}
\usepackage{scrpage2}
\usepackage{multirow}
\usepackage{enumerate}
\usepackage{float}
\usepackage{bm}
\usepackage{dsfont}
\usepackage[capitalise]{cleveref}
\usepackage{braket,amsopn}
\crefformat{equation}{(\textup{#2#1#3})}
\crefrangeformat{equation}{\textup{(#3#1#4)--(#5#2#6)}}
\crefmultiformat{equation}{\textup{#2(#1)#3}}{ and \textup{#2(#1)#3}}
{, \textup{#2(#1)#3}}{, and \textup{#2(#1)#3}}
\crefrangemultiformat{equation}{\textup{#3(#1)#4--#5(#2)#6}}%
{ and \textup{#3(#1)#4--#5(#2)#6}}{, \textup{#3(#1)#4--#5(#2)#6}}{, and \textup{#3(#1)#4--#5(#2)#6}}

\Crefformat{equation}{Equation~#2\textup{(#1)}#3}
\Crefrangeformat{equation}{Equations~\textup{#3(#1)#4--#5(#2)#6}}
\Crefmultiformat{equation}{Equations~\textup{#2(#1)#3}}{ and \textup{#2(#1)#3}}
{, \textup{#2(#1)#3}}{, and \textup{#2(#1)#3}}
\Crefrangemultiformat{equation}{Equations~\textup{#3(#1)#4--#5(#2)#6}}%
{ and \textup{#3(#1)#4--#5(#2)#6}}{, \textup{#3(#1)#4--#5(#2)#6}}{, and \textup{#3(#1)#4--#5(#2)#6}}

\newtheorem{theorem}{Theorem}

\newtheorem{lemma}[theorem]{Lemma} 
\newtheorem{proposition}[theorem]{Proposition} 
\newtheorem{remark}[theorem]{Remark}
\newtheorem{definition}[theorem]{Definition}

\DeclareMathOperator{\Ran}{Ran}

\DeclareMathOperator{\trace}{tr}

\DeclareMathOperator*{\spann}{span}

\DeclareMathOperator*{\argmin}{arg\,min}
\DeclareMathOperator{\vecc}{vec}
\DeclareMathOperator{\rank}{rank}

\DeclareMathOperator{\R}{\mathbb{R}}
\DeclareMathOperator{\C}{\mathbb{C}}
\DeclareMathOperator{\N}{\mathbb{N}}

\DeclareMathOperator{\mat}{mat}
\DeclareMathOperator{\diag}{diag}

\DeclareMathOperator{\opt}{opt}

\usepackage[numbers]{natbib}
\usepackage{jabbrv}

\title{Harmonic Mean Iteratively Reweighted Least Squares for Low-Rank Matrix Recovery}

\author{Christian K\"ummerle\thanks{Department of Mathematics, Technische Universit\"at M\"unchen, 
Boltzmannstr.~3,
85748 Garching/Munich,
Germany  E-mail: \texttt{christian.kuemmerle@ma.tum.de}} , Juliane Sigl\thanks{Department of Mathematics, Technische Universit\"at M\"unchen, 
Boltzmannstr.~3,
85748 Garching/Munich,
Germany  E-mail: \texttt{juliane.sigl@ma.tum.de}}
}
\begin{document}
\title{Harmonic Mean Iteratively Reweighted Least Squares for Low-Rank Matrix Recovery}



\maketitle

\begin{abstract}
We propose a new iteratively reweighted least squares (IRLS) algorithm for the recovery of a matrix $X \in \mathbb{C}^{d_1\times d_2}$ of rank $r \ll\min(d_1,d_2)$ from incomplete linear observations, solving a sequence of low complexity linear problems.
The easily implementable algorithm, which we call harmonic mean iteratively reweighted least squares (\texttt{HM-IRLS}), optimizes a non-convex Schatten-$p$ quasi-norm penalization to promote low-rankness and carries three major strengths, in particular for the matrix completion setting.
First, we observe a remarkable {global convergence behavior} of the algorithm's iterates to the low-rank matrix for relevant, interesting cases, for which any other state-of-the-art optimization approach fails the recovery.
Secondly, \texttt{HM-IRLS} exhibits an empirical recovery probability close to $1$ even for a number of measurements very close to the theoretical lower bound $r (d_1 +d_2 -r)$, i.e., already for \emph{significantly fewer linear observations} than any other tractable approach in the literature.
Thirdly, \texttt{HM-IRLS} exhibits a \emph{locally superlinear rate of convergence (of order $2-p$)} if the linear observations fulfill a suitable null space property. While for the first two properties we have so far only strong empirical evidence, we prove the third property as our main theoretical result.
\end{abstract}

%
\thispagestyle{plain}
\markboth{
}{
}


\section{Introduction} \label{sec_introduction}
The problem of recovering a low-rank matrix from incomplete linear measurements or observations has gained considerable attention in the last few years due to the omnipresence of low-rank models in different areas of science and applied mathematics. Low-rank models arise in a variety of areas such as system identification \citep{LiuHV13,LiuV10}, signal processing \citep{AhmedR15}, quantum tomography \citep{GrossLFBE10,Gross11} and phase retrieval \citep{CandesSV13,Candes13,GrossKK15}. An instance of this problem of particular importance, e.g., in recommender systems \citep{SrebroRJ05,GoldbergNOT92,CR09}, is the \emph{matrix completion} problem, where the measurements correspond to entries of the matrix to be recovered. 

Although the low-rank matrix recovery problem is NP-hard in general, several tractable algorithms have been proposed that allow for provable recovery in many important cases. The \emph{nuclear norm minimization} (\texttt{NNM}) approach \citep{Fazel02,CR09}, which solves a surrogate semidefinite program, is particularly well-understood. 
For \texttt{NNM}, recovery guarantees have been shown for a number of measurements on the order of the information theoretical lower bound $r(d_1+d_2-r)$, if $r$ denotes the rank of a $d_1 \times d_2$-matrix \citep{Recht10,CR09}; i.e., for a number of measurements $m \geq \rho r(d_1+d_2-r)$ with some oversampling constant $\rho \geq 1$. Even though \texttt{NNM} is solvable in polynomial time, it can be computationally very demanding if the problem dimensions are large, which is the case in many potential applications. Another issue is that although the number of measurements necessary for successful recovery by nuclear norm minimization is of \emph{optimal order}, it is not \emph{optimal}. More precisely, it turns out that the oversampling factor $\rho$ of nuclear norm minimization \emph{has to be much larger than the oversampling factor of some other, non-convex algorithmic approaches} \citep{ZhengL15,TannerWei13}. 

These limitations of convex relaxation approaches have led to a rapidly growing line of research discussing the advantages of non-convex optimization for the low-rank matrix recovery problem \citep{JainMD10,TannerWei13,HaldarH09,Jain13,WenYZ12,TannerWei16,Vandereycken13,WeiCCL16MatrixRecovery,TuBSSR15}. For several of these non-convex algorithmic approaches, recovery guarantees comparable to those of \texttt{NNM} have been derived \citep{CandesLS15,TuBSSR15,ZhengL15,SunL15}. Their advantage is a higher empirical recovery rate and an often more efficient implementation. While there are some results about global convergence of first-order methods minimizing a non-convex objective \citep{Ge16,Srebro16} so that a success of the method might not depend on a particular initialization, the assumptions of these results are not always optimal, e.g., in the scaling of the numbers of measurements $m$ in the rank $r$ \citep[Theorem 5.3]{Ge16}. In general, the success of many non-convex optimization approaches relies on a distinct, possibly expensive initialization step.
\subsection{Contribution of this paper} 
In this spirit, we propose a new iteratively reweighted least squares (IRLS) algorithm for the low-rank matrix recovery problem\footnote{The algorithm and partial results were presented at the 12th International Conference on Sampling Theory and Applications in Tallinn, Estonia, July 3--7, 2017. The corresponding conference paper has been published in its proceedings \citep{HMIRLS_SampTA}.}  that strives to minimize a non-convex objective function based on the Schatten-$p$ quasi-norm
\begin{equation} \label{def_Schattenp_min}
\min_X \| X \|_{S_p}^p \text{ subject to } \Phi (X) = Y,
\end{equation}
for $0 < p < 1$, where $\Phi: \C^{d_1 \times d_2} \to \C^m$ is the linear measurement operator and $Y \in \C^m$ is the data vector that define the problem. The overall strategy of the proposed IRLS algorithm is to mimic this minimization by a sequence of weighted least squares problems. This strategy is shared by the related previous algorithms of Fornasier, Rauhut \& Ward \citep{Fornasier11} and Mohan \& Fazel \citep{Mohan10} which minimize \eqref{def_Schattenp_min} by defining iterates as 
\begin{equation} \label{def_IRLS_col}
X^{(n+1)} = \min_X \|{W_L^{(n)}}^{\frac{1}{2}}X\|^2_{F} \text{ subject to } \Phi (X) = Y,
\end{equation}
where $W_L^{(n)} \approx (X^{(n)} X^{(n)*})^{\frac{p-2}{2}}$ is a so-called \emph{weight matrix} which \emph{reweights} the quadratic penalty by operating on the column space of the matrix variable. Thus, we call this column-reweighting type of IRLS algorithms \texttt{IRLS-col}. Due to the inherent symmetry, it is evident to conceive, still in the spirit of \citep{Fornasier11,Mohan10}, the algorithm \texttt{IRLS-row}
\begin{equation} \label{def_IRLS_row}
X^{(n+1)} = \min_X \|{W_R^{(n)}}^{\frac{1}{2}}X^*\|^2_{F} \text{ subject to } \Phi (X) = Y
\end{equation}
with $W_R^{(n)} \approx (X^{(n)*} X^{(n)})^{\frac{p-2}{2}}$, which reweights the quadratic penalty by acting on the \emph{row space} of the matrix variable. We note that even for square dimensions $d_1=d_2$, \texttt{IRLS-col} and \texttt{IRLS-row} do not coincide.

In this paper, as an important innovation, we propose the use of a different type of weight matrices, so-called \emph{harmonic mean weight matrices}, which can be interpreted as the harmonic mean of the matrices $W_L^{(n)}$ and $W_R^{(n)}$ above. This motivates the name harmonic mean iteratively reweighted least squares (\texttt{HM-IRLS}) for the corresponding algorithm. The harmonic mean of the weight matrices of \texttt{IRLS-col} and of \texttt{IRLS-row} in \texttt{HM-IRLS} is able to use the information in both the column and the row space of the iterates, and it also gives rise to a \emph{qualitatively better} behavior than the use of more obvious symmetrizations as, e.g., the arithmetic mean of weight matrices would allow for, both in theory and in practice. 


We argue that the choice of harmonic mean weight matrices as in \texttt{HM-IRLS} 
leads to an efficient algorithm for the low-rank matrix recovery problem with fast convergence and superior performance in terms of sample complexity, also compared to algorithms based on strategies different from IRLS.

On the one hand, we show that the accumulation points of the iterates of \texttt{HM-IRLS} converge to stationary points of a smoothed Schatten-$p$ functional under the linear constraint, as it is known for, e.g., \texttt{IRLS-col}, c.f. \citep{Fornasier11,Mohan10}. On the other hand, we extend the theoretical guarantees which are based on a Schatten-$p$ \emph{null space property} (NSP) of the measurement operator \citep{Oymak11,Foucart13}, to \texttt{HM-IRLS}.

Our main theoretical result is that \texttt{HM-IRLS} exhibits a locally superlinear convergence rate of order $2-p$ in the neighborhood of a low-rank matrix for the non-convexity parameter $0 < p < 1$ connected to the Schatten-$p$ quasinorm, if the measurement operator fulfills the mentioned NSP of sufficient order. For $p \ll 1$, this means that the convergence rate is \emph{almost quadratic}. 


Although parts of our theoretical results, as in the case of the IRLS algorithms algorithms of Fornasier, Rauhut \& Ward \citep{Fornasier11} and Mohan \& Fazel \citep{Mohan10}, do not apply to the matrix completion setting, due to the popularity of the problem and for reasons of comparability with other algorithms, we conduct numerical experiments to explore the empirical performance of \texttt{HM-IRLS} also for this setting. 
Surprisingly enough we observe that the theoretical results comply with our numerical experiments also for matrix completion. In particular, the theoretically predicted local convergence rate of order $2-p$ can be observed very precisely for this important measurement model as well (see \Cref{ratesrho2,ratesrho12,ratesrho1}). 

This local superlinear convergence rate \emph{is unprecedented} for IRLS variants such as \texttt{IRLS-col} and as those that use the arithmetic mean of the one-sided weight matrices: this means that neither can a superlinear rate be verified numerically, nor is it possible to show such a rate by our proof techniques for any other IRLS variant.

To the best of our knowledge, \texttt{HM-IRLS} is the first algorithm for low-rank matrix recovery which achieves superlinear rate of convergence for low complexity measurements as well as for larger problems.

Additionally, we conduct extensive numerical experiments comparing the efficiency of \texttt{HM-IRLS} with previous IRLS algorithms as \texttt{IRLS-col}, Riemannian optimization techniques \citep{Vandereycken13}, alternating minimization approaches \citep{HaldarH09,TannerWei16}, algorithms based on iterative hard thresholding \citep{KyrillidisC14,BlanchardTW15}, and others \citep{Park16}, in terms of sample complexity, again for the important case of \emph{matrix completion}. 
	
The experiments lead to the following observation: \texttt{HM-IRLS} recovers low-rank matrices systematically with an optimal number of measurements that is very close to the theoretical lower bound on the number of measurements that is necessary for recovery with high empirical probability. 
We consider this result to be remarkable, as it means that for problems of moderate dimensionality (matrices of $\approx 10^7$ variables, e.g. $(d_1 \times d_2)$-matrices with $d_1 \approx d_2 \approx 5\cdot10^3$) \emph{the proposed algorithm needs fewer measurements for the recovery of a low rank matrix than all the state-of-the-art algorithms we included in our experiments} (see Figure \ref{avg_plot}).

An important practical observation of \texttt{HM-IRLS} is that its performance is very robust to the choice of the initialization and can be used as a stand-alone algorithm to recover low-rank matrices also
starting from a trivial initialization. This is suggested by our numerical experiments since even for random or adversary initializations, \texttt{HM-IRLS} converges to the low-rank matrix, even though it is based on an objective function which is highly non-convex. 
While a complete theoretical understanding of this behavior is not yet achieved, we regard the empirical evidence in a variety of interesting cases as strong. In this context, we consider a proof of the global convergence of \texttt{HM-IRLS} for non-convex penalizations under appropriate assumptions as an interesting open problem.

\subsection{Outline}
We proceed in the paper as follows. In the next section, we provide some background on Kronecker and Hadamard products of matrices as these concepts are used in the analysis of the algorithm to be discussed. Moreover, we explain different reformulations of the Schatten-$p$ quasi-norm in terms of weighted $\ell_2$-norms, which lead to the derivation of the harmonic mean iteratively reweighted least squares (\texttt{HM-IRLS}) algorithm in \Cref{sec_HM_IRLS}. We present our main theoretical results, the convergence guarantees and the locally superlinear convergence rate for the algorithm in \cref{sec_maintheoryresults}. Numerical experiments and comparisons to state-of-the-art methods for low-rank matrix recovery are carried out in \Cref{sec_num}.
In \Cref{sec_theoretical_analysis}, we interpret the algorithm's different steps as minimizations of an auxililary functional with respect to its arguments and show theoretical guarantees for \texttt{HM-IRLS} extending similar guarantees for \texttt{IRLS-col}. After this, we detail the proof of the locally superlinear convergence rate under appropriate assumptions on the null space of the measurement operator.

\section{Notation and background}
\subsection{General notation, Schatten-$p$ and weighted norms}
In this section, we explain some of the notation we use in the course of this paper.

The set of matrices $X\in \C^{d_1\times d_2}$ is denoted by $M_{d_1 \times d_2}$. Unless stated otherwise, vectors $x \in \mathbb{C}^d$ are considered as column vectors. We also use the vectorized form $X_{\vecc}= \left[
X_1^T,\dots,X_j^T,\dots,X_{d_2}^T\right]^T \in \C^{d_1 d_2}$ of a matrix $X \in M_{d_1\times d_2}$ with columns $X_j$, $j \in \{1,\ldots,d_2\}$. The reverse recast of a vector $x \in \mathbb{C}^{d_1 d_2}$ into a matrix of dimension $d_1 \times d_2$ is denoted by $x_{\mat(d_1,d_2)}= \left[X_1,\dots,X_j,\dots,X_{d_2} \right]$, where  $X_j=[x_{(d_1-1)\cdot j+1},\dots,x_{(d_1-1)\cdot j+d_1}]^T$, $j=1,\dots,d_2$ are column vectors, 
or $X_{\mat}$ if the dimensions are clear from the context. Obviously, it holds that $X=(X_{\vecc})_{\mat}$.

The identity matrix in dimension $d\times d$ is denoted by $\mathbf{I}_d$. With $\mathbf{0}_{d_1\times d_2}\in M_{d_1\times d_2}$ and $\mathbf{1}_{d_1\times d_2}\in M_{d_1\times d_2}$ we denote the matrices with only $0$- or $1$-entries respectively. 
 The set of Hermitian matrices is denoted by $H_{d \times d} := \{X \in M_{d \times d} \mid X = X^*\}$. We write $X^+ \in  M_{d_1 \times d_2}$ for the Moore-Penrose inverse of the matrix $X \in M_{d_1 \times d_2}$.


Let $\mathcal{U}_{d} = \{U \in \C^{d \times d} ; UU^* = \mathbf{I}_d\}$ denote the set of unitary matrices. Then the singular value decomposition of a matrix $X \in M_{d_1 \times d_2}$ can be written as $X= U\Sigma V^*$ with $U \in \mathcal{U}_{d_1}$, $V \in \mathcal{U}_{d_2}$ and $\Sigma \in M_{d_1 \times d_2}$, where $\Sigma$ is diagonal and contains the singular values of $X$ such that 
$\Sigma_{ii} = \sigma_i(X)  \geq 0$ for $i \in \{1,\dots, \min(d_1,d_2)\}$. We define the \emph{Schatten-$p$ (quasi-)norm} of $X\in M_{d_1 \times d_2}$ as  
\begin{equation}\label{def_schattenp} \|X\|_{S_p} := 
\begin{cases} \rank(X),  &\quad \text{ for } p=0,\\
\left[ \sum_{j=1}^{\min(d_1,d_2)} \sigma^p_j (X)\right]^{1/p}, &\quad \text{ for } 0 < p < \infty, \\
\sigma_{\max} (X), &\quad \text { for } p = \infty.	\end{cases}
\end{equation}
Note that for $p=1$, the Schatten-$p$ norm is also called \emph{nuclear norm}, written as $\|X\|_* := \|X\|_{S_1}$. 
The \emph{trace} $\trace[X]$ of a matrix $X \in M_{d_1 \times d_2}$ is defined by the sum of its diagonal elements, $\trace[X]=\sum_{j=1}^{\min(d_1,d_2)} X_{jj}$. It can be seen that the $p$-th power of the Schatten-$p$ norm coincides with $\|X\|^p_{S_p}=\trace\left[(X^*X)^{p/2}\right]$. The Schatten-$2$ norm is also called \emph{Frobenius norm} and has the property that it is induced by the Frobenius scalar product $\langle X,Y\rangle_F = \trace\left[X^* Y\right]$, i.e., $\|X\|_F = \|X\|_{S_2} = \sqrt{\langle X, X\rangle_F}$.
We define the \emph{weighted Frobenius scalar product} of two matrices $X,Y \in M_{d_1\times d_2}$ weighted by the the positive definite weight matrix $W \in H_{d_1 \times d_1}$ as $\langle X,Y \rangle_{F(W)} := \langle W X,Y \rangle_{F}= \langle X,W Y \rangle_{F}$. This scalar product induces the \emph{weighted Frobenius norm} $\|X\|_{F(W)} = \sqrt{\langle X,X\rangle_{F(W)}} = \sqrt{\trace[(WX)^*X]}$. 
It is clear that the Frobenius norm of a matrix $X$ coincides with the $\ell_2$-norm of its vectorization $X_{\vecc}$, i.e., $\|X\|_F=\|X_{\vecc}\|_{\ell_2}$.

Similar to weighted Frobenius norms, we define the \emph{weighted $\ell_2$-scalar product} of vectors $x,y \in \mathbb{C}^{d}$ weighted by the positive definite weight matrix $W \in H_{d \times d}$ as $\langle x,y \rangle_{\ell_2(W)} = x^*Wy = \overline{y^* W x}$ and its induced \emph{weighted $\ell_2$-norm} as $\|x\|_{\ell_2(W)} = \sqrt{x^*Wx}$. 
We use the notation $X \succ 0$ for a positive definite matrix $X \in H_{d \times d}$. Furthermore, we denote the range of a linear map $\Phi:M_{d_1\times d_2}\rightarrow \C^m$ by $\Ran(\Phi)=\big\{Y \in \C^m; \text{ there is }X \in M_{d_1 \times d_2} \text{ such that }Y =\Phi(X)\big\}$ and its null space by $\mathcal{N}(\Phi)=\big\{X \in M_{d_1 \times d_2} ; \Phi(X) = 0\big\}$. 

\subsection{Problem setting and characterization of \texorpdfstring{$S_p$}{Sp}- and reweighted Frobenius norm minimizers} \label{sec_weightedFrob}

Given a linear map $\Phi: M_{d_1\times d_2} \rightarrow \C^m$ such that $m \ll d_1 d_2$, we want to uniquely identify and reconstruct an unknown matrix $X_0$ from its linear image $Y:=\Phi(X_0) \in \C^m$. However, basic linear algebra tells us that this is not possible without further assumptions, since $\Phi$ is not injective if $m < d_1 d_2$. Indeed, there is a $(d_1 d_2 - m)$-dimensional affine space $\{X_0\} + \mathcal{N}(\Phi)$ fulfilling the linear constraint  
\begin{equation*}
\Phi(X)=Y.
\end{equation*}
Nevertheless, under the additional assumption that the matrix $X_0 \in M_{d_1 \times d_2}$ has rank $r < \min(d_1,d_2)$ and under appropriate assumptions on the map $\Phi$, the recovery of $X_0$ is possible by solving the affine rank minimization problem
\begin{equation}\label{model0}
\min \rank(X) \text{ subject to } \Phi(X)=Y.
\end{equation}
The unique solvability of \cref{model0} is given with high probability if, for example, $\Phi$ is a linear map whose matrix representation has i.i.d. Gaussian entries \citep{ENP12} and $m = \Omega(r(d_1+d_2))$. Unfortunately, solving \cref{model0} is intractable in general, but the works \citep{CR09,Recht10,Candes10} suggest solving the tractable convex optimization program
\begin{equation}\label{model1}
\min \|X\|_{S_{1}} \text{ subject to } \Phi(X)=Y,
\end{equation}
also called \emph{nuclear norm minimization (\texttt{NNM})}, as a proxy. 
%

%

As discussed in the introduction, there are empirical as well as theoretical results (e.g., in \citep{Daubechies10,Chartrand07}) coming from the related \emph{sparse vector recovery} problem that suggest alternative relaxation approaches. These results indicate that it might be even more advantageous to solve the non-convex problem 
\begin{equation}\label{modelp}
\min F^p(X):= \|X\|^p_{S_{p}} \text{ subject to } \Phi(X)=Y,
\end{equation}
for $0<p<1$, i.e., minimizing the $p$-th power of the Schatten-$p$ quasi-norms under the affine constraint. Heuristically, the choice of $p < 1$ relatively small can be motivated by the observation that by the definition \cref{def_schattenp} of the Schatten-$p$ quasi-norm
\[ 
\|X\|_{S_p}^p \xrightarrow{p \to 0} \rank(X)=:\|X\|_{S_0}.
\]
The above consideration suggests that the solution of \cref{modelp} might be closer to \cref{model0} than \cref{model1} for small $p$. On the other hand, again, it is in general computationally intractable to find a global minimum of the non-convex optimization problem \cref{modelp} if $p <1$. 
Therefore it is a natural and very relevant question to ask which optimization algorithm to use to find global minimizers of \cref{modelp}. 

In this paper, we discuss an algorithm striving to solve \cref{modelp} that is based on the following observations: Assume for the moment that we are given a square matrix $X \in M_{d_1 \times d_2}$ with $d_1 = d_2$ of full rank. Then, we can rewrite the $p$-th power of its Schatten-$p$ quasi-norm as a square of a weighted Frobenius norm, or, using Kronecker product notiation as explained in \cref{sec_kroneckerhadamard}, as a square of a weighted $\ell_2$-norm (if we use the vectorized notation $X_{\vecc}$): Iit turns out that
\begin{enumerate}
\item[(i)] $ \begin{aligned}[t] \|X\|^p_{{S_p}}&=\trace[(XX^*)^{\frac{p}{2}}]=\trace[(XX^*)^{\frac{p-2}{2}}(XX^*)]=\trace({W}_LXX^*)=\|{W}_L^{\frac{1}{2}}X\|^2_{F}\\&=\|X\|^2_{F(W_L)}=\|(\mathbf{I}_{d_2}\otimes{W}_L)^{\frac{1}{2}}X_{\vecc}\|^2_{\ell_2}=\|X_{\vecc}\|^2_{\ell_2(\mathbf{I}_{d_2}\otimes{W}_L)},\end{aligned}$ \\
where ${W}_L$ is the symmetric weight matrix $(XX^*)^{\frac{p-2}{2}}$ in $M_{d_1\times d_1}$ and $\mathbf{I}_{d_2}\otimes{W}_L$ is the block diagonal weight matrix in $M_{d_1 d_2 \times d_1 d_2}$ with $d_2$ instances of ${W}_L$ on the diagonal blocks, but also that
\item[(ii)]$ \begin{aligned}[t] \|X\|^p_{{S_p}}&=\trace[(X^*X)^{\frac{p}{2}}]=\trace[(X^*X)(X^*X)^{\frac{p-2}{2}}]=\trace(X^*X{W}_R)=\|X{W}_R^{\frac{1}{2}}\|^2_{F}\\&=\|X^*\|^2_{F(W_R)}=\|({W}_R\otimes\mathbf{I}_{d_1})^{\frac{1}{2}}X_{\vecc}\|^2_{\ell_2}=\|X_{\vecc}\|^2_{\ell_2(W_R\otimes\mathbf{I}_{d_1})},\end{aligned}$ \\
where $W_R$ is the symmetric weight matrix $(X^*X)^{\frac{p-2}{2}}$ in $M_{d_2\times d_2}$. It follows from the definition of the Kronecker product that the weight matrix ${W}_R \otimes \mathbf{I}_{d_1} \in M_{d_1 d_2 \times d_1 d_2}$ is a block matrix of diagonal blocks of the type $\diag\big((W_R)_{ij},\ldots,(W_R)_{ij} \big)\in M_{d_1\times d_1}$, $i,j \in [d_2]$. 
\vspace*{-6mm} 
\end{enumerate}
 \begin{figure}[H]
 \makebox [1\textwidth][c]{
    \subfloat[$\mathbf{I}_{d_2}\otimes{W}_L$]{
        \includegraphics[width=0.3\textwidth]{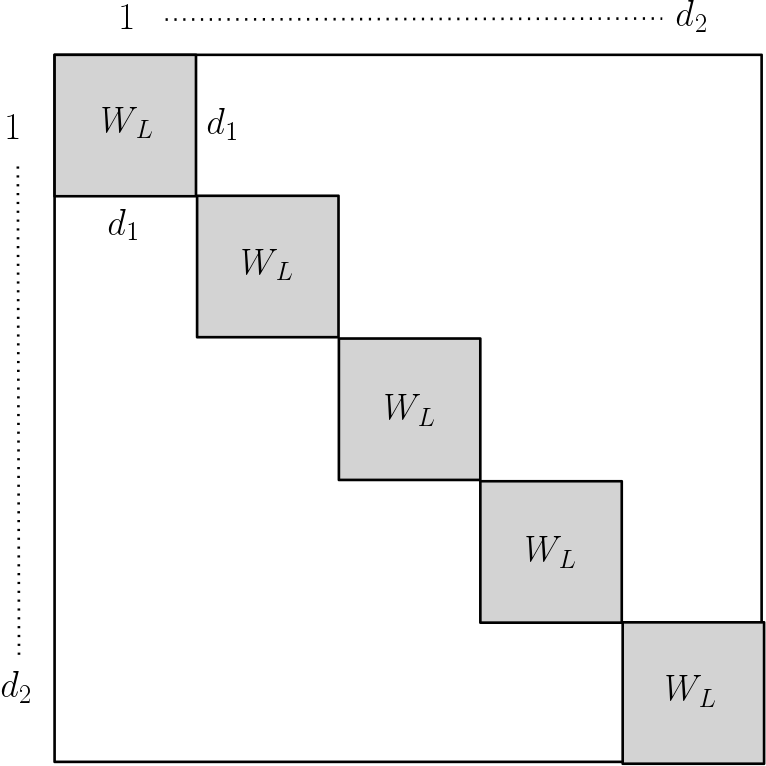}}
 \subfloat[${W}_R \otimes \mathbf{I}_{d_1}$]{
        \includegraphics[width=0.3\textwidth]{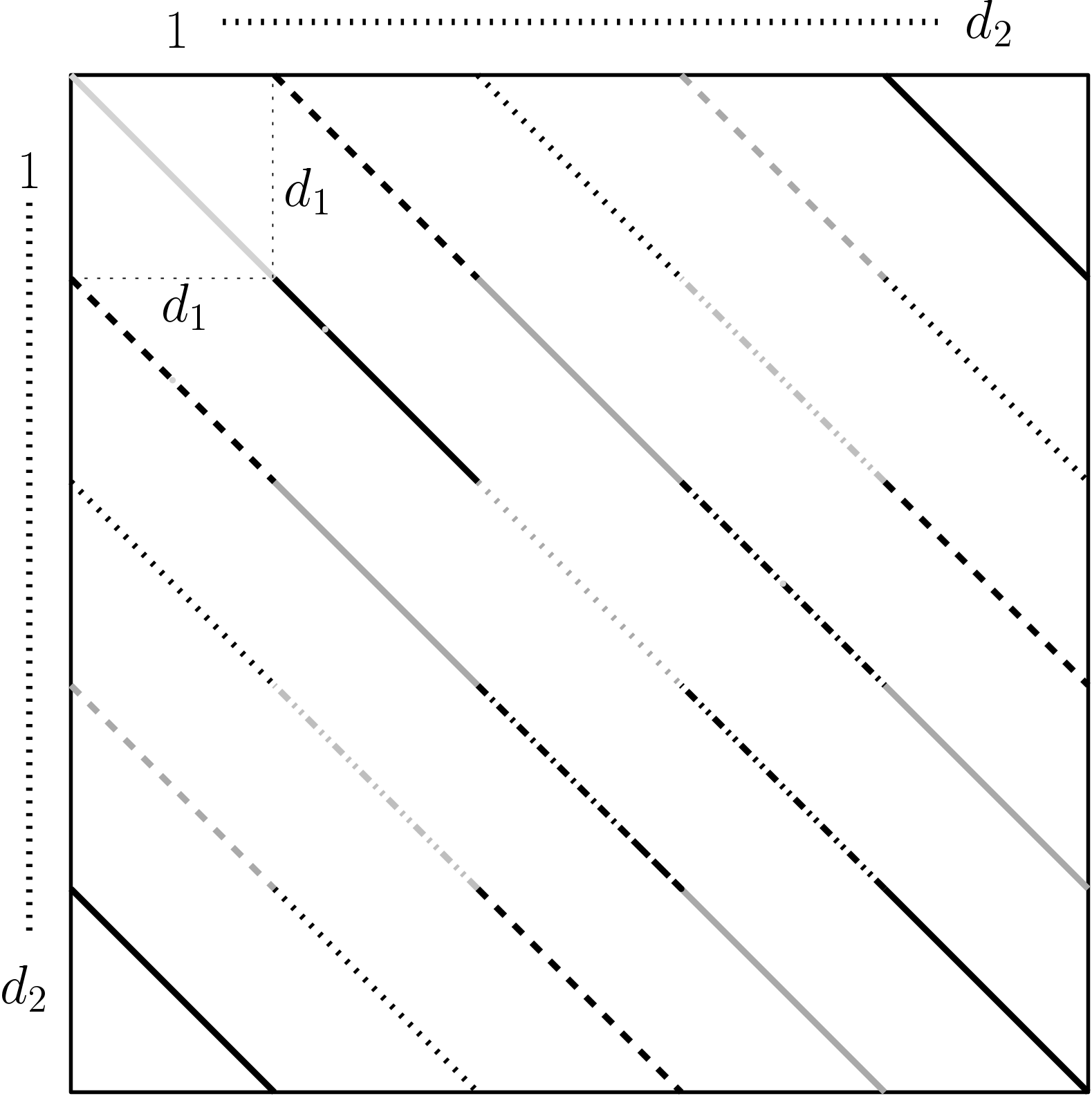}
    }}
\caption{Sparsity structure of the weight matrices $\in M_{d_1d_2 \times d_1 d_2}$} \label{fig_structure_weight}
\end{figure}
 The sparsity structures of $\mathbf{I}_{d_2}\otimes{W}_L$ and ${W}_R \otimes \mathbf{I}_{d_1}$ are illustrated in \cref{fig_structure_weight}. Note that a representation of $\|X\|_{S_p}^p$ by 
squares of Frobenius norms can be achieved by multiplying $X$ by $W_L^{\frac{1}{2}}$ from the left in (i), or by $W_R^{\frac{1}{2}}$ from the right in (ii).

The above calculations are not well-defined if $X$ is not of full rank or if $d_1 \neq d_2$, since in these cases at least one of the matrices $X X^* \in M_{d_1 \times d_1}$ or $X^*X \in M_{d_2 \times d_2}$ is singular, prohibiting the definition of the matrices $W_R = (X^*X)^{\frac{p-2}{2}}$ or $W_L = (XX^*)^{\frac{p-2}{2}}$ for $p < 2$. However, these issues can be overcome by introducing a smoothing parameter $\epsilon > 0$ and smoothed weight matrices ${W}_{L}(X,\epsilon) \in M_{d_1 \times d_1}$ and ${W}_{R}(X,\epsilon) \in M_{d_2 \times d_2}$ defined by
\begin{align}
{W}_{L}(X,\epsilon)&:=(XX^*+\epsilon^2\mathbf{I}_{d_1})^{\frac{p-2}{2}}, \label{W_L_eps}\\
{W}_{R}(X,\epsilon)&:=(X^*X+\epsilon^2\mathbf{I}_{d_2})^{\frac{p-2}{2}}. \label{W_R_eps}
\end{align}
\begin{remark}
 The weight matrices ${W}_{L}(X,\epsilon)$ and ${W}_{R}(X,\epsilon)$ are symmetric and positive definite. 
\end{remark}
The possibility to rewrite the $p$-th power of the Schatten-$p$ of a matrix as a weighted Frobenius norm gives rise to the general strategy of IRLS algorithms for low-rank matrix recovery: Weighted least squares problems of the type
\[
\min_{\substack{X \in M_{d_1 \times d_2} \\ \Phi(X)=Y}} \|X\|^2_{F(W_L)} \quad \text{ or } \min_{\substack{X \in M_{d_1 \times d_2} \\ \Phi(X)=Y}} \|X^*\|^2_{F(W_R)}
\]
are solved and weight matrices $W_L$ are updated alternatingly, leading to the algorithms column-reweighting \texttt{IRLS-col} and row-reweighting \texttt{IRLS-row}, respectively \citep{Mohan10,Fornasier11}.
%
%

\subsection{Averaging of weight matrices}\label{sec_avg_weightmatrices}	

While the algorithms \texttt{IRLS-col} and \texttt{IRLS-row} provide a tractable local minimization strategy of smoothed Schatten-$p$ functionals under the linear constraint, we argue that it is suboptimal to follow either one of the two approaches as they do not exploit the symmetry of the problem in an optimal way: They \emph{either} use low-rank information in the column space \emph{or} in the row space.

A first intuitive approach towards a symmetric exploitation of the low-rank structure is inspired by the following identity, by combing the calculations (i) and (ii) carried out in \cref{sec_weightedFrob}.
\begin{lemma}\label{arith}  Let $0 < p \leq 2$ and $X \in M_{d_1 \times d_2}$ with $d=d_1=d_2$ be a full rank matrix. Then
\begin{equation*}
\begin{aligned}
\|X\|^p_{S_p}&=\frac{1}{2}\left(\|{W}_L^{\frac{1}{2}}X\|^2_{F}+\|X{W}_R^{\frac{1}{2}}\|^2_{F}\right)
= \left\|\left(\frac{{W}_L \oplus {W}_R}{2}\right)^{\frac{1}{2}}X_{\vecc}\right\|^2_{\ell_2} = \|X_{\vecc}\|_{\ell_2(W_{(\text{arith})})}^2,
\end{aligned}
\end{equation*}
where $$\frac{1}{2}\left(\mathbf{I}_{d_2}\otimes {W}_L+{W}_R\otimes \mathbf{I}_{d_1}\right)=\frac{{W}_L \oplus {W}_R}{2} =: W_{(\text{arith})}$$ is the arithmetic mean matrix of the symmetric and positive definite weight matrices $\mathbf{I}_{d_2}\otimes {W}_L$ and ${W}_R\otimes \mathbf{I}_{d_1}$, $W_L := (XX^*)^{\frac{p-2}{2}}$, and $W_R := (X^*X)^{\frac{p-2}{2}}$. 
\end{lemma}
Unfortunately, the introduction of arithmetic mean weight matrices does not prove to be particularly advantageous compared to one-sided reweighting strategies. No convincing improvements can be noted neither in numerical experiments nor in the theoretical investigations for the convergence rate of IRLS for low-rank matrix recovery, cf. also \Cref{sec_num_ratecomparisons} and \Cref{remark_Mohanlaueftnicht}.

In contrast, we want to promote the usage of the \emph{harmonic mean of the weight matrices} $\mathbf{I}_{d_2}\otimes {W}_L$ and ${W}_R\otimes \mathbf{I}_{d_1}$, i.e., weight matrices of the type $2\left({W}_R^{-1}\otimes \mathbf{I}_{d_1}+\mathbf{I}_{d_2}\otimes {W}_L^{-1}\right)^{-1}=2\left({W}_L^{-1} \oplus {W}_R^{-1}\right)^{-1} =: W_{(\text{harm})}$. In the remaining parts of the paper, we explain why $W_{(\text{harm})}$ is able to significantly outperform other weighting variants both theoretically and practically. 

 The following lemma verifies that also the harmonic mean of the weight matrices $\mathbf{I}_{d_2}\otimes {W}_L$ and ${W}_R\otimes \mathbf{I}_{d_1}$ leads to a legitimate reformulation of the Schatten-$p$ quasi-norm power.
\begin{lemma}\label{harmonic} Let $0 < p \leq 2$ and $X \in \C^{d_1 \times d_2}$ with $d=d_1=d_2$ be a full rank matrix. Then
\begin{equation*}
\begin{aligned}
\|X\|^p_{S_p}&=2 \left\|\left({W}_L^{-1} \oplus {W}_R^{-1}\right)^{-\frac{1}{2}}X_{\vecc}\right\|^2_{\ell_2} = \|X_{\vecc}\|_{\ell_2(W_{(\text{harm})})}^2,
\end{aligned}
\end{equation*}
where $$2\left({W}_R^{-1}\otimes \mathbf{I}_{d_1}+\mathbf{I}_{d_2}\otimes {W}_L^{-1}\right)^{-1}=2\left({W}_L^{-1} \oplus {W}_R^{-1}\right)^{-1} =: W_{(\text{harm})}$$ is the harmonic mean matrix of the symmetric and positive definite weight matrices $\mathbf{I}_{d_2}\otimes {W}_L$ and ${W}_R\otimes \mathbf{I}_{d_2}$, $W_L := (XX^*)^{\frac{p-2}{2}}$ and $W_R := (X^*X)^{\frac{p-2}{2}}$. 
\end{lemma}
\begin{proof}
Let $X= U\Sigma V^*  = \sum_{i=1}^d \sigma_i u_iv_i^* \in M_{d\times d}$ be the singular value decomposition of $X$. Therefore for the vectorized version, $X_{\vecc}=(V\otimes U)\Sigma_{\vecc}$ holds true. By the definitions of $W_L$ and $W_R$, we can write $W_L^{-1} =\sum_{i=1}^d \sigma_i^{2-p} u_iu_i^*$ and $W_R^{-1} =\sum_{i=1}^d \sigma_i^{2-p} v_iv_i^*$. Using the Kronecker sum inversion formula of \Cref{lemma_jameson} in \cref{sec_kroneckerhadamard}, we obtain
\begin{equation*}
\begin{split}
\|X_{\vecc}\|_{\ell_2(W_{(\text{harm})})}^2  &= \|W_{(\text{harm})}^{\frac{1}{2}} X_{\vecc}\|_{\ell_2}^2 = 2 \left\|\left({W}_L^{-1} \oplus {W}_R^{-1}\right)^{-\frac{1}{2}}X_{\vecc}\right\|^2_{\ell_2} \\
&= 2\trace\left(\left(\left({W}_L^{-1} \oplus {W}_R^{-1}\right)^{-1}  X_{\vecc}\right)^*_{\mat}X\right)\\
&= \sum_{i=1}^{d} \sum_{j=1}^{d}  \sum_{k=1}^{d} \frac{2 \sigma_k}{\sigma_i^{2-p} +\sigma_j^{2-p}}v_j v_i^* v_k u_k^* u_i u_i^*  \sum_{l=1}^{d} \sigma_l u_lv_l^*\\
&=2\left(\sum_{i=1}^{d}\frac{\sigma_i^2}{2\sigma_i^{2-p}} \right)=\|X\|^p_{S_p},
\end{split}
\end{equation*}
which finishes the proof.
\end{proof}

\section{Harmonic mean iteratively reweighted least squares algorithm} \label{sec_HM_IRLS}
In this section, we use this idea to formulate a new iteratively reweighted least squares algorithm for low-rank matrix recovery. The so-called \emph{harmonic mean iteratively reweighted least squares algorithm} (\texttt{HM-IRLS}) solves a sequence of weighted least squares problems to recover a low-rank matrix $X_0 \in M_{d_1 \times d_2}$ from few linear measurements $\Phi(X_0) \in \C^m$. The weight matrices appearing in the least squares problems can be seen as the harmonic mean of the weight matrices 
in \eqref{W_L_eps} and \eqref{W_R_eps}, i.e., the ones used by \texttt{IRLS-col} and \texttt{IRLS-row}.

More precisely, for $0 < p \leq 1$ and $d=\min(d_1,d_2), D=\max(d_1,d_2)$, given a non-increasing sequence of non-negative real numbers $(\epsilon^{(n)})_{n=1}^\infty$ and the sequence of iterates $(X^{(n)})_{n=1}^\infty$ produced by the algorithm, we update our weight matrices such that
\begin{equation} \label{Wtilde_firstdef}
\widetilde{W}^{(n)} = 2 \left[U^{(n)}(\overline{\Sigma}_{d_1}^{(n)})^{2-p}U^{(n)*} \oplus V^{(n)}(\overline{\Sigma}_{d_2}^{(n)})^{2-p}V^{(n)*}\right]^{-1},
\end{equation}
with the diagonal matrices $\overline{\Sigma}_{d_t}^{(n)}\in M_{d_t \times d_t}$ for $d_t=\left\{d_1,d_2 \right\}$ such that 
\begin{equation}\label{sigmabar}
(\overline{\Sigma}^{(n)}_{d_t})_{ii} = \begin{cases}(\sigma_i(X^{(n)})^2 + \epsilon^{(n)2})^{\frac{1}{2}}
& \text{ if } i \leq d,\\
0 
& \text{ if }  d < i \leq D,
\end{cases}
\end{equation}
and the matrices $U^{(n)} \in \mathcal{U}_{d_1}$ and $V^{(n)} \in \mathcal{U}_{d_2}$, containing the left and right singular vectors of $X^{(n)}$ in its columns, respectively.

We note that this definition of $\widetilde{W}^{(n)}$ can be seen as a stabilized version of the harmonic mean weight matrix $W_{(\text{harm})}$ of \Cref{harmonic}. This stabilization is necessary as $\widetilde{W}^{(n)}$ becomes very ill-conditioned as soon as some of the singular values of $X^{(n)}$ approach zero and, related to that, $(X^{(n)}X^{(n)*})^{\frac{2-p}{2}} \oplus (X^{(n)*}X^{(n)})^{\frac{2-p}{2}}$ would even be singular as soon as $X^{(n)}$ is not of full rank.

Additionally, for the formualtion of the algorithm and any $n \in \N$, it is convenient to define the linear operator $(\widetilde{\mathcal{W}}^{(n)})^{-1}:M_{d_1 \times d_2} \rightarrow M_{d_1 \times d_2}$ as
\begin{equation} \label{left_mult_W}
(\widetilde{\mathcal{W}}^{(n)})^{-1}(X) :=  \frac{1}{2}\left[U^{(n)}(\overline{\Sigma}_{d_1}^{(n)})^{2-p}U^{(n)*} X + X V^{(n)}(\overline{\Sigma}_{d_2}^{(n)})^{2-p}V^{(n)*}\right],
\end{equation}
describing the operation of the inverse of $\widetilde{W}^{(n)}$ on $M_{d_1 \times d_2}$.

Finally, \texttt{HM-IRLS} can be formulated in pseudo code as follows.
\begin{algorithm}[H]
\caption{Harmonic Mean IRLS for low-rank matrix recovery (\texttt{HM-IRLS})} \label{algo1}
\DontPrintSemicolon 
\KwIn{A linear map $\Phi: M_{d_1\times d_2} \rightarrow \C^m$, image $Y=\Phi(X_0)$ of the ground truth matrix $X_0 \in M_{d_1 \times d_2}$, rank estimate $\widetilde{r}$, non-convexity parameter $0 < p \leq 1$.}
\KwOut{Sequence $(X^{(n)})_{n=1}^{n_0} \subset M_{d_1 \times d_2}$.}
Initialize $n=0$, $\epsilon^{(0)}=1$ and $\widetilde{W}^{(0)} = \mathbf{I}_{d_1d_2} \in M_{d_1 d_2 \times d_1 d_2}$. \;

\Repeat{stopping criterion is met.}{
\begin{align} \label{xmin}
&X^{(n+1)} =\argmin\limits_{\Phi(X)=Y} \|X_{\vecc}\|^2_{\ell_2(\widetilde{W}^{(n)})}\!\!= (\widetilde{\mathcal{W}}^{(n)})^{-1}\big(\Phi^*\big((\Phi \circ (\widetilde{\mathcal{W}}^{(n)})^{-1}\circ \Phi^*)^{-1}(Y)\big)\big),  \\
&\epsilon^{(n+1)}=\min\left(\epsilon^{(n)},\sigma_{\widetilde{r}+1}(X^{(n+1)})\right), \label{epsmin}\\
&\widetilde{W}^{(n+1)} = 2 \left[U^{(n+1)}(\overline{\Sigma}_{d_1}^{(n+1)})^{2-p}U^{(n+1)*} \oplus V^{(n+1)}(\overline{\Sigma}_{d_2}^{(n+1)})^{2-p}V^{(n+1)*}\right]^{-1}, \label{def_Wn_tilde} 
\end{align}
where $U^{(n+1)} \in \mathcal{U}_{d_1}$ and $V^{(n+1)} \in \mathcal{U}_{d_2}$ are matrices containing the left and right singular vectors of $X^{(n+1)}$ in its columns, and the $\overline{\Sigma}_{d_t}^{(n+1)}$ are defined for $t \in \{1,2\}$ according to \cref{sigmabar}. \;
\[
n = n+1, \quad\quad\quad \quad\quad\quad \quad\quad\quad \quad\quad\quad \quad\quad\quad \quad\quad\quad \quad\quad\quad \quad\quad\quad\quad
\]}
Set $n_0 = n$.
\end{algorithm}
From a practical point of view, it is beneficial that the explicit calculation of the very large weight matrices $\widetilde{W}^{(n+1)} \in H_{d_1 d_2 \times d_1 d_2}$ from \eqref{def_Wn_tilde} is not necessary in implementations of \Cref{algo1}. 
As suggested by formulas \eqref{left_mult_W} and \eqref{xmin}, it can be seen that just the operation of \emph{its inverse $(\widetilde{W}^{(n+1)})^{-1}$ resp. $(\widetilde{W}^{(n)})^{-1}$} is needed, which can be implemented by matrix-matrix multiplications on the space $M_{d_1 \times d_2}$: For matrices $X,\widetilde{X} \in M_{d_1 \times d_2}$, we have that $\widetilde{W}^{(n)} X_{\vecc} = \widetilde{X}_{\vecc}$ if and only if $X_{\vecc} = (\widetilde{W}^{(n)})^{-1}\widetilde{X}_{\vecc}$, which can be written in matrix variables as
\begin{equation*}
X =  \frac{1}{2}\left[U^{(n)}(\overline{\Sigma}_{d_1}^{(n)})^{2-p}U^{(n)*} \widetilde{X} + \widetilde{X} V^{(n)}(\overline{\Sigma}_{d_2}^{(n)})^{2-p}V^{(n)*}\right].
\end{equation*} 
The last equivalence is due to the definitions of $\widetilde{W}^{(n)}$ and the Kronecker sum, cf. \eqref{def_Wn_tilde} and \Cref{sec_kroneckerhadamard}.

Note that the smoothing parameters $\epsilon^{(n)}$ are chosen in dependence on a rank estimate $\tilde{r}$ here, which will be an important ingredient for the theoretical analysis of the algorithm. In practice, however, other choices of non-increasing sequences of non-negative real numbers $(\epsilon^{(n)})_{n=1}^\infty$ are possible and can as well lead to (a maybe even faster) convergence when tuned appropriately.

We refer to \Cref{section_computationalcomplexity} for a further discussion of implementation details.

\paragraph{Example}
With a simple example, we illustrate the versatility of \texttt{HM-IRLS}: Let $d_1 = d_2 =4$, and assume that we want to reconstruct the rank-$1$ matrix 
\[
X_0 = u v^* = \begin{pmatrix} 1 \\  10 \\ -2 \\ 0.1 \end{pmatrix} \begin{pmatrix} 1 & 2 & 3 & 4 \end{pmatrix} = 
\begin{pmatrix}           
 1   &        2    &    3     &     4 \\
           10   &     20   &      30  &  40 \\
           -2     &    -4      &    -6      &   -8 \\
          0.1   &    0.2    &   0.3   &     0.4 \end{pmatrix}
\]
from $m = d_f = r (d_1 +d_2 -r) = 7$ sampled entries $\Phi(X_0)$, where $\Phi$ is the linear map $\Phi: M_{4 \times 4}  \rightarrow \C^7$, $\Phi(X)= \begin{pmatrix}X_{2,1}, & X_{4,1}, & X_{3,2}, & X_{4,2}, & X_{4,3}, & X_{1,4}, & X_{2,4} \end{pmatrix}$. Since the linear map $\Phi$ samples some entries of matrices in $M_{4 \times 4}$ and does not see the others, this is an instance of the problem that is called \emph{matrix completion}.

In general, reconstructing a $(d_1 \times d_2)$ rank-$r$ matrix from $m = r (d_1 +d_2 -r)$ entries is a hard problem, as it is known that if $m < r (d_1 +d_2 -r)$, there is always more than one matrix $X$ such that $\Phi(X)=\Phi(X_0)$, and even for equality, the property that $\Phi$ is invertible on (most) rank-$r$ matrices might be hard to verify \citep{Kiraly15}. 

It can be argued that the specific matrix completion problem we consider is in some sense a hard one, since, e.g., the deterministic sufficient condition for unique completability of \citep[Theorem 2]{Pimentel15} is not fulfilled (less then $2$ observed entries in the third column), and since the classical coherence parameters $\mu(u) = d_1 \max\limits_{1\leq i \leq 4} \frac{\| u u^* e_i\|_2^2}{\|u\|_2^4} \approx 3.81$ and $\mu(v) = d_2 \max\limits_{1\leq i \leq 4} \frac{\| v v^* e_i\|_2^2}{\|v\|_2^4} \approx 2.13$ that are used to analyze the behavior of many matrix completion algorithms \citep{CR09,Jain13} are quite large, with $\mu(u)$ being quite close to the maximal value of $4$.

On the other hand, as the problem is small and $X_0$ has rank $r=1$, it is possible to impute the missing values of
\[
\begin{pmatrix}
 *   &        *    &       *  &   4 \\
           10   &    *   &      * &  40 \\
          *     &    -4      &    *      &   * \\
          0.1   &    0.2    &   0.3   &    *
\end{pmatrix}
\]
by solving very simple linear equations, since, for example, $X_{4,4}= u_4 v_4$, $X_{2,1}=u_2 v_1$, $X_{2,4}=u_2 v_4$, and $X_{4,1}=u_4 v_1$, and therefore $X_{4,4} = \frac{X_{4,1} X_{2,4}}{X_{2,1}} = 0.4$. This shows that the only rank-$1$ matrix compatible with $\Phi(X_0)$ is $X_0$.

It turns out that---without using the combinatorial simplicity of the problem---the classical \texttt{NNM} does not solve the problem, as the nuclear norm minimizer (solution of \eqref{model1} for $Y=\Phi(X_0)$) produced by the semidefinite program of the convex optimization package \texttt{CVX} \citep{cvx} converges to
\[
\overline{X}_{\text{nuclear}} \approx         
\begin{pmatrix}
1   &  0.023  &  0.041   &       4 \\
           10   &   0.232   &   0.411       &    40 \\
    -0.056   &        -4 &   -0.200   &   -0.226 \\
          0.1    &      0.2     &     0.3    & 0.400
          \end{pmatrix},
\]
a matrix with $ 45.74 \approx \|\overline{X}_{\text{nuclear}}\|_{S_1} < \|X_0\|_{S_1} = \sigma_1(X_0)\approx 56.13$ and a relative Frobenius error of $\frac{\|\overline{X}_{\text{nuclear}} - X_0\|_F}{\|X_0\|_F} = 0.661$. 

Interestingly, \texttt{HM-IRLS} is able to solve the problem, if $p$ is chosen small enough, with very high precision already after few iterations, for example, up to a relative error of $4.18\cdot 10^{-13}$ after $24$ iterations if $p=0.1$. This is in contrast to the behavior of \texttt{IRLS-col}, \texttt{IRLS-row} and also to \texttt{AM-IRLS}, the IRLS variant that uses weight matrices derived from the \emph{arithmetic mean} the weights of \texttt{IRLS-col} and \texttt{IRLS-row}, cf. \cref{arith}. The iterates $X^{(n)}$ for iteration $n=2000$ of these algorithms exhibit relative errors of $0.240$, $ 0.489$ and $0.401$, respectively, for the choice of $p=0.1$---furthermore, there is no choice of $p$ that would lead to a convergence to $X_0$.

To understand this very different behavior, we note that the $n$-th iterate of any of the four IRLS variants can be written, using \cref{sec_kroneckerhadamard}, in a concise way as 
\begin{equation} \label{eq_IRLS_variants_opt}
X^{(n+1)} = \argmin\limits_{\Phi(X)=Y}\;\; \langle X_{\vecc}, W^{(n)} X_{\vecc} \rangle,
\end{equation}
where
\begin{equation} \label{eq_IRLS_general_QF}
\langle X_{\vecc}, W^{(n)} X_{\vecc} \rangle = \langle X, U^{(n)}\big[ H^{(n)} \circ (U^{(n)*} X V^{(n)} ) \big] V^{(n)*} \rangle_F = \sum_{i,j=1}^4 H_{ij}^{(n)}  | \langle u_i^{(n)}, X v_j^{(n)}\rangle|^{2}
\end{equation}
with $X^{(n)} = U^{(n)} \Sigma^{(n)} V^{(n)*}= \sum_{i=1}^4 \sigma_i^{(n)} u_i^{(n)} v_i^{(n)}$ being the SVD of $X^{(n)}$, and
\[
H_{ij}^{(n)} = \begin{cases}
2\big[\big((\sigma_i^{(n)})^{2}+(\epsilon^{(n)})^{2})^{\frac{2-p}{2}}+\big((\sigma_j^{(n)})^{2}+ (\epsilon^{(n)})^{2})\big)^{\frac{2-p}{2}}\big]^{-1} & \text{ for \texttt{HM-IRLS}}, \\
\big((\sigma_i^{(n)})^{2}+(\epsilon^{(n)})^{2}\big)^{\frac{p-2}{2}} & \text{ for \texttt{IRLS-col}}, \\
\big((\sigma_j^{(n)})^{2}+(\epsilon^{(n)})^{2}\big)^{\frac{p-2}{2}} & \text{ for \texttt{IRLS-row}, and} \\
0.5\cdot \big[\big((\sigma_i^{(n)})^{2}+(\epsilon^{(n)})^{2}\big)^{\frac{p-2}{2}} +   \big((\sigma_i^{(n)})^{2}+(\epsilon^{(n)})^{2}\big)^{\frac{p-2}{2}} \big]& \text{ for \texttt{AM-IRLS}}, \\
\end{cases}
\]
for $i,j \in \{1,2,3,4\}$ and $\epsilon^{(n)} = \min(\sigma_2^{(n)},\epsilon^{(n-1)})$.

The values of the matrix $H^{(1)}$ of weight coefficients after the first iteration in the above example are visualized in \Cref{fig:IRLS_H_matrices}, for each of the four IRLS versions above.
\begin{figure} \subfloat[\texttt{HM-IRLS} \hspace{17.5mm}(b) \texttt{IRLS-col} \hspace{16.5mm} (c) \texttt{IRLS-row} \hspace{16.5mm} (d) \texttt{AM-IRLS} \hspace*{8mm}]{
\includegraphics[width=1.00\textwidth]{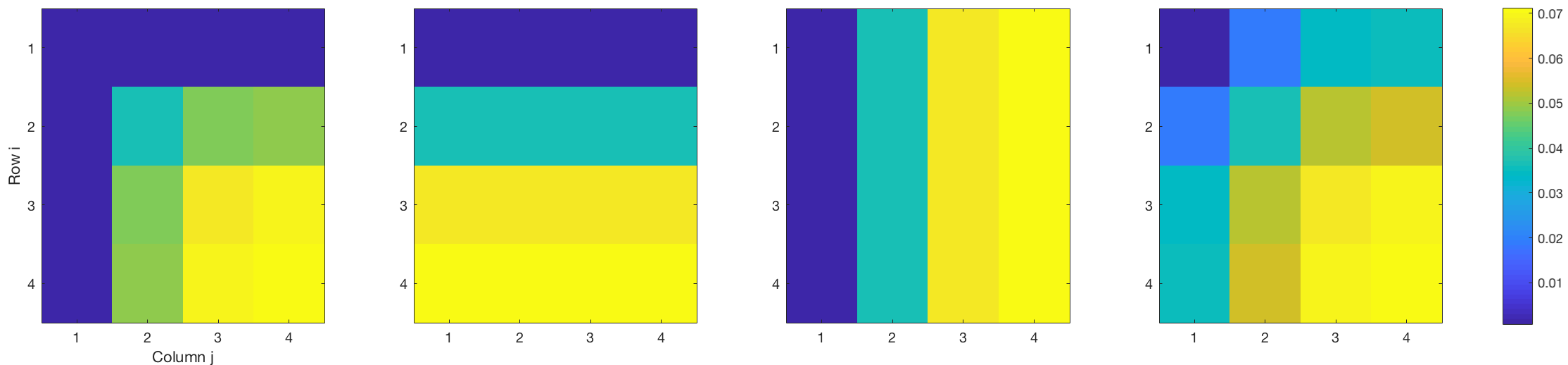}}  
\caption{Values of the matrix $H^{(1)}$ of "weight coefficients" corresponding to the orthonormal basis $(u_i^{(1)} v_j^{(1)*})_{i,j=1}^{4}$ after the first iteration in the example}\label{fig:IRLS_H_matrices}
\end{figure}

The intuition for the superior behavior of \texttt{HM-IRLS} is now the following: Since large entries of $H^{(n)}$ \emph{penalize} the corresponding parts of the space $M_{d_1 \times d_2} = \spann\{ u_i^{(n)}v_j^{(n)*}, i \in [d_1], j \in [d_2]\}$ in the minimization problem \cref{eq_IRLS_variants_opt}, large areas of \emph{blue} and \emph{dark blue} in \Cref{fig:IRLS_H_matrices} indicate a benign optimization landscape where the minimizer $X^{(n+1)}$ of \cref{eq_IRLS_variants_opt} \emph{is able to improve considerably} on the previous iterate $X^{(n)}$. 

In particular, it can be seen that in the case of \texttt{HM-IRLS}, the penalties on the \emph{whole direct sum of column and row space of the best rank-$r$ approximation of $X^{(n)}$} 
\begin{equation*}
T^{(n)}:=\big\{ \begin{pmatrix} u_1^{(n)},&\ldots,&u_r^{(n)}\end{pmatrix} Z_1^{*} +  Z_2 \begin{pmatrix} v_1^{(n)},&\ldots,&v_r^{(n)}\end{pmatrix}^{*} : Z_1 \in M_{d_1 \times r}, Z_2 \in M_{d_2 \times r} \big\},
\end{equation*}
are small compared to the other penalites, since the coefficients of $H^{(1)}$ corresponding to $T^{(1)}$ are exactly the ones in the first row and first column of the $(4 \times 4)$ matrices in \Cref{fig:IRLS_H_matrices}---a contrast that becomes more and more pronounced as $X^{(n)}$ approaches the rank-$r$ ground truth $X_0$ (with $r=1$ in the example).

On the other hand, \texttt{IRLS-col}, \texttt{IRLS-row} and \texttt{AM-IRLS} only have small coefficients on smaller parts of $T^{(n)}$, which, from a global perspective, explains why their usage might lead to non-global minima of the Schatten-$p$ objective.

We note that the space $T^{(n)}$ plays also an important role in Riemannian optimization approaches for matrix recovery problems \citep{Vandereycken13}, since it is also the tangent space of the smooth manifold of rank-$r$ matrices at the best rank-$r$ approximation of $X^{(n)}$.

\section{Convergence results} \label{sec_maintheoryresults}
In the following part, we state our main theoretical results about convergence properties of the algorithm \texttt{HM-IRLS}. Furthermore, their relation to existing results for \texttt{IRLS-col} and \texttt{IRLS-row} is discussed.

It cannot be expected that a low-rank matrix recovery algorithm like \texttt{HM-IRLS} succeeds to converge to a low-rank matrix without any assumptions on the measurement operator $\Phi$ that defines the recovery problem \eqref{model0}. For the purpose of the convergence analysis of \texttt{HM-IRLS}, we introduce the following strong \emph{Schatten-$p$ null space property} \citep{Fornasier11,Oymak11,Foucart13}.
\begin{definition}[Strong Schatten-$p$ null space property]\label{NSPs}
Let $0 < p \leq 1$. We say that a linear map $\Phi: M_{d_1\times d_2} \rightarrow \C^m$ fulfills the strong Schatten-$p$ null space property (Schatten-$p$ NSP) of order $r$ with constant $0 < \gamma_r \leq 1$ if
\begin{equation}\label{k-NSP}
\bigg(\sum_{i=1}^r \sigma_i^ 2(X)\bigg)^{p/2} < \frac{\gamma_r}{r^{1-\frac{p}{2}}} \bigg(\sum_{i=r+1}^{d} \sigma_i^p (X)\bigg)
\end{equation}
for all $X \in \mathcal{N}(\Phi)\setminus \{0\}$.
\end{definition}
Intuitively explained, if a map $\Phi$ fulfills the strong Schatten-$p$ null space property of order $r$, there are no rank-$r$ matrices in the null space and all the elements of the null space must not have a quickly decaying spectrum.

Null space properties have already been used to guarantee the success of nuclear norm minimization \eqref{model1}, or Schatten-1 minimization in our terminology, for solving the low-rank matrix recovery problem \citep{Recht11}. 

We note that the definitions of Schatten-$p$ null space properties are quite analogous to the $\ell_p$-null space property in classical compressed sensing \citep[Theorem 4.9]{Foucart13}, applied to the vector of singular values. In particular, \eqref{k-NSP} implies that 
\begin{equation} \label{eq_weakSchattenNSP}
\sum_{i=1}^r \sigma_i^p(X) < \sum_{i=r+1}^{d} \sigma_i^p(X)\quad\quad \text{ for all }X \in \mathcal{N}(\Phi)\setminus \{0\},
\end{equation}
since $\|X\|_{S_p} \leq r^{1/p-1/2} \|X\|_{S_2}$ for $X$ that is rank-$r$. This, in turn, ensures the existence of unique solutions to \eqref{modelp} if $Y=\Phi(X_0)$ are the measurements of a low-rank matrix $X_0$.
\begin{proposition}[\citep{Foucart16}] \label{thm_foucart}
Let $\Phi: M_{d_1\times d_2} \rightarrow \C^m$ be a linear map, let $0 < p \leq 1$ and $ r \in \N$. Then every matrix $X_0 \in M_{d_1\times d_2}$ such that $\rank(X_0) \leq r$ and $\Phi(X_0) = Y \in \C^m$ is the \emph{unique} solution of Schatten-$p$ minimization \eqref{modelp} if and only if $\Phi$ fulfills \eqref{eq_weakSchattenNSP}.
\end{proposition} 
\begin{remark}
The sufficiency of the Schatten-$p$ NSP \cref{eq_weakSchattenNSP} in \cref{thm_foucart} already been pointed out by Oymak et al. \citep{Oymak11}. The necessity as stated in the theorem, however, is due to a recent generalization of Mirsky's singular value inequalities to concave functions \citep{Audenaert14,Foucart16}.
\end{remark}
It can be seen that the (weak) Schatten-$p$ NSP of \eqref{eq_weakSchattenNSP} is a \emph{stronger} property for larger $p$ in the sense that if $0 < p' \leq p \leq 1$, the Schatten-$p$ property implies the Schatten-$p'$ property. 
Very related to that, it can be seen that for any $0 < p \leq 1$, the strong Schatten-$p$ null space property is implied by a sufficiently small \emph{rank restricted isometry constant $\delta_r$}, which is a classical tool in the analysis of low-rank matrix recovery algorithms \citep{Recht10,Candes10}.
\begin{definition}[Restricted isometry property (RIP)]
The \emph{restricted isometry constant $\delta_r > 0$ of order $r$} of the linear map $\Phi: M_{d_1\times d_2} \rightarrow \C^m$ is defined as the smallest number such that 
\[
(1-\delta_r) \|X\|_{F}^2 \leq \|\Phi(X)\|_{\ell_2}^2 \leq (1+\delta_r) \|X\|_{F}^2 
\]
for all matrices $X \in M_{d_1 \times d_2}$ of rank at most $r$.
\end{definition}
Indeed, it follows from the proof of \citep[Theorem 4.1]{Kutzarova15} that a restricted isometry constant of order $2r$ such that $\delta_{2r} < \frac{2}{\sqrt{2}+3} \approx 0.4531$ implies the strong Schatten-$p$ NSP of order $r$ with a constant $\gamma_r<1$ for any $0 < p \leq 1$. More precisely, it can be seen that $\delta_{2r} < \frac{2}{\sqrt{2}+3}$ implies that the strong Schatten-$p$ NSP \cref{k-NSP} of order $r$ holds with the constant $\gamma_r = \frac{(\sqrt{2}+1)^p}{2^p} \frac{\delta_{2r}^p}{(1-\delta_{2r})^p}$.

Linear maps that are instances drawn from certain random models 
%
are known to fulfill the restricted isometry property with high probability if the number of measurements is sufficiently large \citep{Davenport16}, and, a fortiori, the Schatten-$p$ null space property. In particular, this is true for \emph{(sub-)Gaussian} linear measurement maps $\Phi:M_{d_1\times d_2}\rightarrow \mathbb{C}^{m}$ whose matrix representation is such that
\begin{equation}\label{gauss}
\frac{1}{\sqrt{m}}\widetilde{\Phi} \in \mathbb{C}^{m\times d_1d_2}, \text{  where } \widetilde{\Phi} \text{  has i.i.d. standard (sub-)Gaussian entries,}
\end{equation} 
as it is summarized in the following lemma.
\begin{lemma} \label{remark_gaussian_NSP}
For any $0 < p \leq 1$, $0 < \gamma < 1$ and any (sub-)Gaussian random operator $\Phi: M_{d_1\times d_2} \rightarrow \C^m$ (e.g. as defined in \cref{gauss}), there exist constants $C_1 > 1$, $C_2>0$ such that if $m \geq C_1 r(d_1+d_2)$, the strong Schatten-$p$ null space property \eqref{k-NSP} of order $r$ with constant $\gamma_r< \gamma$ is fulfilled with probability at least $1-e^{-C_2m}$. 
\end{lemma}
  
\subsection{Local convergence for \texorpdfstring{$p < 1$}{p < 1}}
In this section, we provide a convergence analysis for \texttt{HM-IRLS} covering several aspects. We are able to show that the algorithm converges to stationary points of a smoothed Schatten-$p$ functional $g_{\epsilon}^p$ as in $\eqref{eq_def_gpe}$ without any additional assumptions on the measurement map $\Phi$. Such guarantees have already been obtained for IRLS algorithms with one-sided reweighting as \texttt{IRLS-col} and \texttt{IRLS-row}, in particular for $p=1$ by Fornasier, Rauhut \& Ward \citep{Fornasier11} and for $0<p \leq 1$ by Mohan \& Fazel \citep{Mohan10}. 

Beyond that, assuming the measurement operator fulfills an appropriate Schatten-$p$ null space property as defined in \cref{NSPs}, we show the a-posteriori exact recovery statement that \texttt{HM-IRLS} converges to the low-rank matrix $X_0$ if $\lim\limits_{n \rightarrow \infty}\epsilon_n= 0$, which only was shown for one-sided \texttt{IRLS} for the case $p=1$ by \cite{Fornasier11}. 

Moreover, we provide a local convergence guarantee stating that \texttt{HM-IRLS} recovers the low-rank matrix $X_0$ if we obtain an iterate $X^{(\overline{n})}$ that is close enough to $X_0$, which is novel for IRLS algorithms. 

Let $0 < p \leq 1$ and $\epsilon > 0$. To state the theorem, we introduce the  \emph{$\epsilon$-perturbed Schatten-$p$ functional $g^p_{\epsilon}: M_{d_1 \times d_2} \rightarrow \R_{\geq 0}$} such that
\begin{equation} \label{eq_def_gpe}
g^p_{\epsilon} (X)= \sum\limits^{d}_{i=1}  (\sigma_i(X)^2 + \epsilon^2)^{\frac{p}{2}}
\end{equation}
where $\sigma(X) \in \R^{d}$ denotes the vector of singular values of $X \in M_{d_1 \times d_2}$.

\begin{theorem}\label{conv}
Let ${\Phi}:M_{d_1 \times d_2}  \rightarrow  \mathbb{C}^{ m}$ be a linear operator and $Y \in \Ran(\Phi)$ a vector in its range. Let $(X^{(n)})_{n\geq 1}$ and $(\epsilon^{(n)})_{n \geq 1}$ be the sequences produced by \cref{algo1} for input parameters $\Phi,Y,r$ and $0 < p \leq 1$, let $\epsilon = \lim_{n\to \infty} \epsilon^{(n)}$. 
\begin{enumerate}
\item[(i)] If $\epsilon=0$ and if ${\Phi}$ fulfills the strong Schatten-$p$ NSP \cref{k-NSP} of order $r$ with constant $0 < \gamma_r < 1$, 
then the sequence $(X^{(n)})_{n\geq 1}$ converges to a matrix $\overline{X}\in M_{d_1 \times d_2}$ of rank at most $r$ that is the unique minimizer of the Schatten-$p$ minimization problem \cref{modelp}.
Moreover, there exists an absolute constant $\hat{C} > 0$ such that for any $X$ with ${\Phi}(X)={Y}$ and any $\widetilde{r} \leq r$, it holds that 
\begin{equation*}
\|X-\overline{X}\|^p_{F}\leq \frac{\hat{C}}{r^{1 - p/2}}\beta_{\widetilde{r}}(X)_{S_p},
\end{equation*}
where $\hat{C}=\frac{2^{p+1} \gamma_r^{1-p/2}}{1-\gamma_r}$ and $\beta_{\widetilde{r}}(X)_{S_p}$ is the best rank-$\widetilde{r}$ Schatten-$p$ approximation error of $X$, i.e.,
 \begin{equation}\label{beta}
\beta_{\widetilde{r}} (X)_{S_p} := \inf \big\{\|X - \widetilde{X} \|^p_{S_p},\, \widetilde{X}\in M_{d_1\times d_2} \text{ has rank }\widetilde{r}\big\}.
\end{equation} \label{conv_statement1}
\item[(ii)] If $\epsilon>0$, then each accumulation point $\overline{X}$ of $(X^{(n)})_{n\geq 1}$ is a stationary point of the $\epsilon$-perturbed Schatten-$p$ functional $g^p_{\epsilon}$ of \cref{eq_def_gpe} under the linear constraint $\Phi(X)=Y$.
If additionally $p = 1$, then $\overline{X}$ is the unique global minimizer of $g^p_{\epsilon}$. 
\item[(iii)]
Assume that there exists a matrix $X_0 \in M_{d_1 \times d_2}$ with $\Phi(X_0) = Y$ such that $\rank(X_{0})=r \leq \frac{\min(d_1,d_2)}{2}$, a constant $0 < \zeta < 1$ and an iteration $\overline{n} \in \N$ such that
 \begin{equation*}
\|X^{(\overline{n})}-X_0\|_{S_\infty} \leq \zeta \sigma_{\widetilde{r}}(X_{0}) 
 \end{equation*}
 and $\epsilon^{\overline{n}}=\sigma_{r+1}(X^{\overline{n}})$. 
 If $\Phi$ fulfills the strong Schatten-$p$ NSP of order $2r$ with $\gamma_{2r}< 1$ and if the condition number $\kappa=\frac{\sigma_1 (X_0)}{\sigma_{r} (X_0)}$ of $X_0$ and $\zeta$ are sufficiently small (see condition \cref{eq_mu_condition2} and formula \cref{eq_mu_definition}), then 
\[
X^{(n)}\to X_0 \quad \text{ for } n\to \infty.
\]
\end{enumerate}
\end{theorem}

It is important to note that by using \cref{remark_gaussian_NSP}, it follows that the assertions of \cref{conv}(i) and (iii) hold for (sub-)Gaussian operators \cref{gauss} with high probability in the regime of measurements of optimal sample complexity order. In particular, there exist constant oversampling factors $\rho_1,\rho_2 \geq 1$ such that the assertions of (i) and (iii) hold with high probability if $m > \rho_k r (d_1 + d_2)$, $k \in \{1,2\}$, respectively. 


%
\begin{remark} \label{remark_noNSP_MC}
However, if $m < d_1 d_2$, null space property-type assumptions as \cref{k-NSP} or \cref{eq_weakSchattenNSP} do not hold for the important case of matrix completion-type measurements \citep{CR09}, where $\Phi(X)$ is given as $m$ sample entries
\begin{equation}\label{def_MC_setting}
\Phi(X)_{\ell} = X_{i_\ell,j_\ell}, \quad\quad \ell = 1,\ldots,m,
\end{equation}
and $(i_\ell,j_\ell) \in [d_1] \times [d_2]$ for all $\ell \in [m]$, of the matrix $X \in M_{d_1 \times d_2}$, which also were considered in the example of \cref{sec_HM_IRLS}.

This means that parts (i) and (iii) of \cref{conv} do, unfortunately, not apply for matrix completion measurements, which define a very relevant class of low-rank matrix recovery problems. This problem is shared by any existing theory for IRLS algorithms for low-rank matrix recovery \citep{Fornasier11,Mohan10}. However, in \cref{sec_num}, we provide strong numerical evidence that \texttt{HM-IRLS} exhibits properties as predicted by (i) and (iii) of \cref{conv} even for the matrix completion setting.
We leave the extension of the theory of \texttt{HM-IRLS} to matrix completion measurements as an open problem to be tackled by techniques different from uniform null space properties \citep[Section V]{Davenport16}. \end{remark}
\subsection{Locally superlinear convergence rate for \texorpdfstring{$p < 1$}{p < 1}}

Next, we state the second main theoretical result of this paper, \cref{rate}. It shows that in a neighborhood of a low-rank matrix $X_0$ that is compatible with the measurement vector $Y$, the algorithm \texttt{HM-IRLS} converges to $X_0$ with a \emph{convergence rate} that is \emph{superlinear of the order $2-p$}, if the operator $\Phi$ fulfills an appropriate Schatten-$p$ null space property. 

\begin{theorem}[Locally Superlinear Convergence Rate] \label{rate}
Assume that the linear map $\Phi: M_{d_1 \times d_2} \to \C^m$ fulfills the strong Schatten-$p$ NSP of order $2r$ with constant $\gamma_{2r} < 1$ and that there exists a matrix $X_{0} \in M_{d_1 \times d_2}$ with $\rank(X_{0})=r \leq \frac{\min(d_1,d_2)}{2}$ such that $\Phi(X_0)={Y}$, let $\Phi,Y,r$ and $0 < p \leq 1$ be the input parameters of \cref{algo1}. Moreover, let $\kappa=\frac{\sigma_1 (X_0)}{\sigma_r(X_0)}$ be the condition number of $X_0$ and  $\eta^{(n)} := X^{(n)}-X_{0}$ be the error matrices  of the $n$-th output of \cref{algo1} for $n \in \N$. \\
Assume that there exists an iteration $\overline{n} \in \N$ and a constant $0 < \zeta < 1$ such that
 \begin{equation}\label{rho}
\|\eta^{(\overline{n})}\|_{S_\infty} \leq \zeta \sigma_r(X_{0}) 
 \end{equation}
and $\epsilon^{(\overline{n})}=\sigma_{r+1}(X^{(\overline{n})})$.
If additionally the condition number $\kappa$ and $\zeta$ are small enough, or more precisely, if
\begin{equation} \label{eq_mu_condition2}
\mu \|\eta^{(\overline{n})}\|_{S_\infty}^{p(1-p)} < 1
\end{equation}
with the constant
\begin{equation} \label{eq_mu_definition}
\mu := 2^{5p}  (1+\gamma_{2r})^{p} \Big(\frac{\gamma_{2r}(3+\gamma_{2r})(1+\gamma_{2r})}{(1-\gamma_{2r})}\Big)^{2-p}\Big(\frac{d-r}{r}\Big)^{2-\frac{p}{2}} r^p\frac{\sigma_r(X_0)^{p(p-1)}}{(1-\zeta)^{2p}} \kappa^p 
\end{equation}
then 
 \begin{equation*}
\|\eta^{(n+1)}\|_{S_\infty}\leq \mu^{1/p} \left(\|\eta^{(n)}\|_{S_\infty} \right)^{2-p} \quad \text{ and }  \quad \|\eta^{(n+1)}\|_{S_p}\leq \mu^{1/p} \left(\|\eta^{(n)}\|_{S_p} \right)^{2-p}
\end{equation*}
for all $n \geq \overline{n}$. 
\end{theorem}

We think that the result of \cref{rate} is remarkable, since there are only few low-rank recovery algorithms which exhibit either theoretically or practically verifiable superlinear convergence rates. In particular, although the algorithms of \citep{Mishra13} and \texttt{NewtonSLRA} of \citep{Schost16} do show superlinear convergence rates, the first are not competitive to \texttt{HM-IRLS} in terms of sample complexity and the second has neither applicable theoretical guarantees for most of the interesting problems nor the ability of solving medium size problems.
\begin{remark} 
We observe that while the statement describes the observed rates of convergence very accurately (cf. \cref{sec_num_ratecomparisons}), the assumption \cref{eq_mu_condition2} on the neighborhood that enables convergence of a rate $2-p$ is more pessimistic than our numerical experiments suggest. Our experiments confirm that the local convergence rate of order $2-p$ also holds for matrix completion measurements, where the assumption of a Schatten-$p$ null space property fails to hold, cf. \cref{sec_num}.
\end{remark}



\subsection{Discussion and comparison with existing IRLS algorithms}

Optimally, we would like to have a statement in \cref{conv} about the accumulation points $\overline{X}$ being \emph{global minimizers} of $g_{\epsilon}^p$, instead of mere stationary points \citep[Theorem 6.11]{Fornasier11}, \citep[Theorem 5.3]{Daubechies10}. A statement that strong is, unfortunately, difficult to achieve due to the non-convexity of the Schatten-$p$ quasinorm and of the $\epsilon$-perturbed version $g_{\epsilon}^p$. Nevertheless, our theorems can be seen as analogues of \citep[Theorem 7.7]{Daubechies10}, which discusses the convergence properties of an IRLS algorithm for sparse recovery based on $\ell_p$-minimization with $p < 1$.

As already mentioned in previous sections, Fornasier, Rauhut \& Ward \citep{Fornasier11} and Mohan \& Fazel \citep{Mohan10} proposed IRLS algorithms for low-rank matrix recovery and analysed their convergence properties.
The algorithm of \citep{Fornasier11} corresponds (almost) to \texttt{IRLS-col} with $p=1$ as explained in \Cref{sec_HM_IRLS}.
In this context, \cref{conv} recovers the results \citep[Theorem 6.11(i-ii)]{Fornasier11} for $p=1$ and generalize them, with weaker conclusions due to the non-convexity, to the cases $0 < p <1$. 
The algorithm \texttt{IRLS-$p$} of \citep{Mohan10} is similar to the former, but differs in the choice of the $\epsilon$-smoothing and also covers non-convex choices $0 < p <1$. 
However, we note that in the non-convex case, its convergence result \citep[Theorem 5.1]{Mohan10} corresponds to \cref{conv}(ii), but does not provide statements similar to (i) and (iii) of \cref{conv}. 

\Cref{rate} with its analysis of the convergence rate is new in the sense that to the best of our knowledge, there are no convergence rate proofs for IRLS algorithms for the low-rank matrix recovery problem in the literature. Indeed, we refer to \cref{remark_Mohanlaueftnicht} in \cref{sec_localconvrate} for an explanation why the variants of \citep{Fornasier11} and \citep{Mohan10} cannot exhibit superlinear convergence rates, unlike \texttt{HM-IRLS}.

We also note that there is a close connection between the statements of \cref{conv,rate} and results that were obtained by Daubechies, DeVore, Fornasier and G\"unt\"urk \citep[Theorems 7.7 and 7.9]{Daubechies10} for an IRLS algorithm dedicated to the sparse vector recovery problem. 

\section{Numerical experiments} \label{sec_num}

In this section, we demonstrate first that the superlinear convergence rate that was proven theoretically for \Cref{algo1} (\texttt{HM-IRLS}) in \Cref{rate} can indeed be accurately verified in numerical experiments, even beyond measurement operators fulfilling the strong null space property, 
%
and compare its performance to other variants of \texttt{IRLS}.

In \Cref{sec_comp_algs}, we then examine the recovery performance of \texttt{HM-IRLS} for the matrix completion setting with the performance of other state-of-the-art algorithms comparing the measurement complexities that are needed for successful recovery for many random instances. 

The numerical experiments are conducted on Linux and Mac systems with MATLAB R2017b. 
An implementation of the \texttt{HM-IRLS} algorithm and a minimal test example are available at \href{https://www-m15.ma.tum.de/Allgemeines/SoftwareSite}{https://www-m15.ma.tum.de/Allgemeines/SoftwareSite}.

\subsection{Experimental setup} \label{sec_measurement_setting}
In the experiments, we sample $(d_1 \times d_2)$ dimensional ground truth matrices $X_0$ of rank $r$ such that $X_0 = U \Sigma V^*$, where $U \in \R^{d_1 \times r}$ and $V \in \R^{d_2 \times r}$ are independent matrices with i.i.d. standard Gaussian entries and $\Sigma \in \R^{r \times r}$ is a diagonal matrix with i.i.d. standard Gaussian diagonal entries, independent from $U$ and $V$.

We recall that a rank-$r$ matrix $X \in M_{d_1 \times d_2}$ has $d_f=r(d_1+d_2-r)$ degrees of freedom, which is the theoretical lower bound on the number of measurements that are necessary for exact reconstruction \citep{Candes10}. The random measurement setting we use in the experiments can be described as follows: We take measurements of matrix completion type, sampling $m=\lfloor\rho d_f\rfloor$ entries of $X_0$ uniformly over its $d_1 d_2$ indices to obtain $Y=\Phi(X_0)$. Here, $\rho$ is such that $\frac{d_1d_2}{d_f}\geq\rho\geq1$ and parametrizes the difficulty of the reconstruction problem, from very hard problems for $\rho \approx 1$ to easier problems for larger $\rho$.

However, this uniform sampling of $\Phi$ could yield instances of measurement operators whose information content is not large enough to ensure well-posedness of the corresponding low-rank matrix recovery problem, even if $\rho > 1$. 
More precisely, it is impossible to recover a matrix exactly if the number of revealed entries in any row or column is smaller than its rank $r$, which is explained and shown in the context of the proof of \citep[Theorem 1]{Pimentel15}. 

Thus, in order to provide for a sensible measurement model for small $\rho$, we exclude operators $\Phi$ that sample fewer than $r$ entries in any row or column. 
Therefore, we adapt the uniform sampling model such that operators $\Phi$ are discarded and sampled again until the requirement of at least $r$ entries per column and row is met and recovery can be achieved from a theoretical point of view.

We note that the described phenomenon is very related to the fact that matrix completion recovery guarantees for the uniform sampling model require at least one additional $\log$ factor, i.e., they require at least $m \geq \log(\max(d_1,d_2)) d_f$ sampled entries \citep[Section V]{Davenport16}.

While we detail the experiments for the matrix completion measurement setting just described in the remaining section, we add that Gaussian measurement models also lead to very similar results in experiments.
\subsection{Convergence rate comparison with other IRLS algorithms} \label{sec_num_ratecomparisons}
In this subsection, we vary the Schatten-$p$ parameter between $0$ and $1$ and compare the corresponding convergence behavior of \texttt{HM-IRLS} with the IRLS variant \texttt{IRLS-col}, which performs the reweighting just in the column space, and with the arithmetic mean variant \texttt{AM-IRLS}. The latter two coincide with \cref{algo1} except that the weight matrices are chosen as described in \Cref{eq_IRLS_general_QF} in \cref{sec_HM_IRLS}.

We note that \texttt{IRLS-col} is very similar to the IRLS algorithms of \citep{Fornasier11} and \citep{Mohan10} and differs from them basically just in the choice of the $\epsilon$-smoothing. We present the experiments with \texttt{IRLS-col} to isolate the influence of the weight matrix type, but very similar results can be observed for the algorithms of \citep{Fornasier11} and \citep{Mohan10}.\footnote{Implementations of the mentioned authors' algorithms were downloaded from 
\url{https://faculty.washington.edu/mfazel/} and \url{https://github.com/rward314/IRLSM}, respectively.}

In the matrix completion setup of \cref{sec_measurement_setting}, we choose $d_1=d_2=40$, $r=10$ and distinguish easy, hard and very hard problems corresponding to oversampling factors $\rho$ of $2.0$, $1.2$ and $1.0$, respectively. The algorithms are provided with the ground truth rank $r$ and are stopped whenever the relative change of Frobenius norm $\|X^{(n)}-X^{(n-1)}\|_F/\|X^{(n-1)}\|_F$ drops below the threshold of $10^{-10}$ or a maximal iteration of iterations $n_{\max}$ is reached.

\begin{figure}[t!]
\includegraphics[width=1\textwidth]{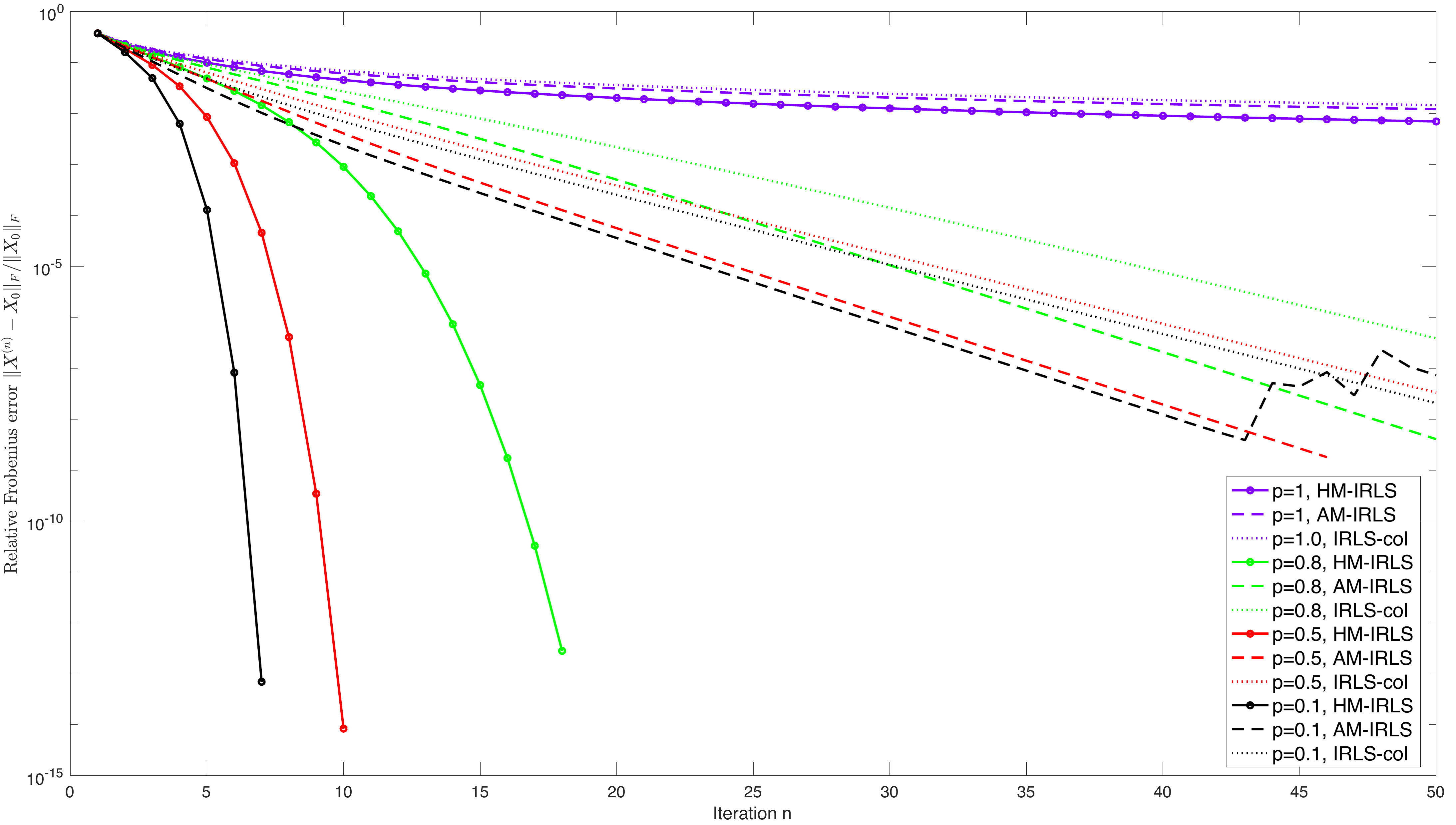}
\vspace*{-8mm}
\caption{Relative Frobenius errors as a function of the iteration $n$ for oversampling factor $\rho=2$ (easy problem).} \label{ratesrho2}
\vspace*{-2mm}
\end{figure}

\subsubsection{Convergence rates}
First, we study the behavior of the three IRLS algorithms for the easy setting of an oversampling factor of $\rho =  2$, which means that $\frac{2 r (d_1 +d_2 -r)}{d_1 d_2} = 0.875$ of the entries are sampled, and parameters $p \in \{0.1,0.5,0.8,1\}$. 

In \Cref{ratesrho2}, we observe that for $p=1$, \texttt{HM-IRLS}, \texttt{AM-IRLS} and \texttt{IRLS-col} have a quite similar behavior, as the relative Frobenius errors $\|X^{(n)}-X_0\|_F /\|X_0\|_F$ decrease only slowly, i.e., even a linear rate is hardly identifiable. For choices $p < 1$ that correspond to non-convex objectives, we observe a very fast, superlinear convergence of \texttt{HM-IRLS}, as the iterates $X^{(n)}$ converge up to a relative error of less than $10^{-12}$ within fewer than $20$ iterations for $p \in \{0.8,0.5,0.1\}$. Precise calculations verify that the rate of convergences are indeed of order $2-p$, the order predicted by \Cref{rate}. We note that this fast convergence rate kicks in not only locally, but starting from the very first iteration.

On the other hand, it is easy to see that \texttt{AM-IRLS} and \texttt{IRLS-col} converge \emph{linearly, but not superlinearly} to the ground truth $X_0$ for $p \in \{0.8,0.5,0.1\}$. The linear rate of \texttt{AM-IRLS} is slightly better than the one of \texttt{IRLS-col}, but the numerical stability of \texttt{AM-IRLS} deteriorates for $p=0.1$ close to the ground truth (after iteration 43). This is due to a bad conditioning of the quadratic problems as the $X^{(n)}$ are close to rank-$r$ matrices. In contrast, no numerical instability issues can be observed for \texttt{HM-IRLS}.

\begin{figure}[t!]
\includegraphics[width=1.00\textwidth]{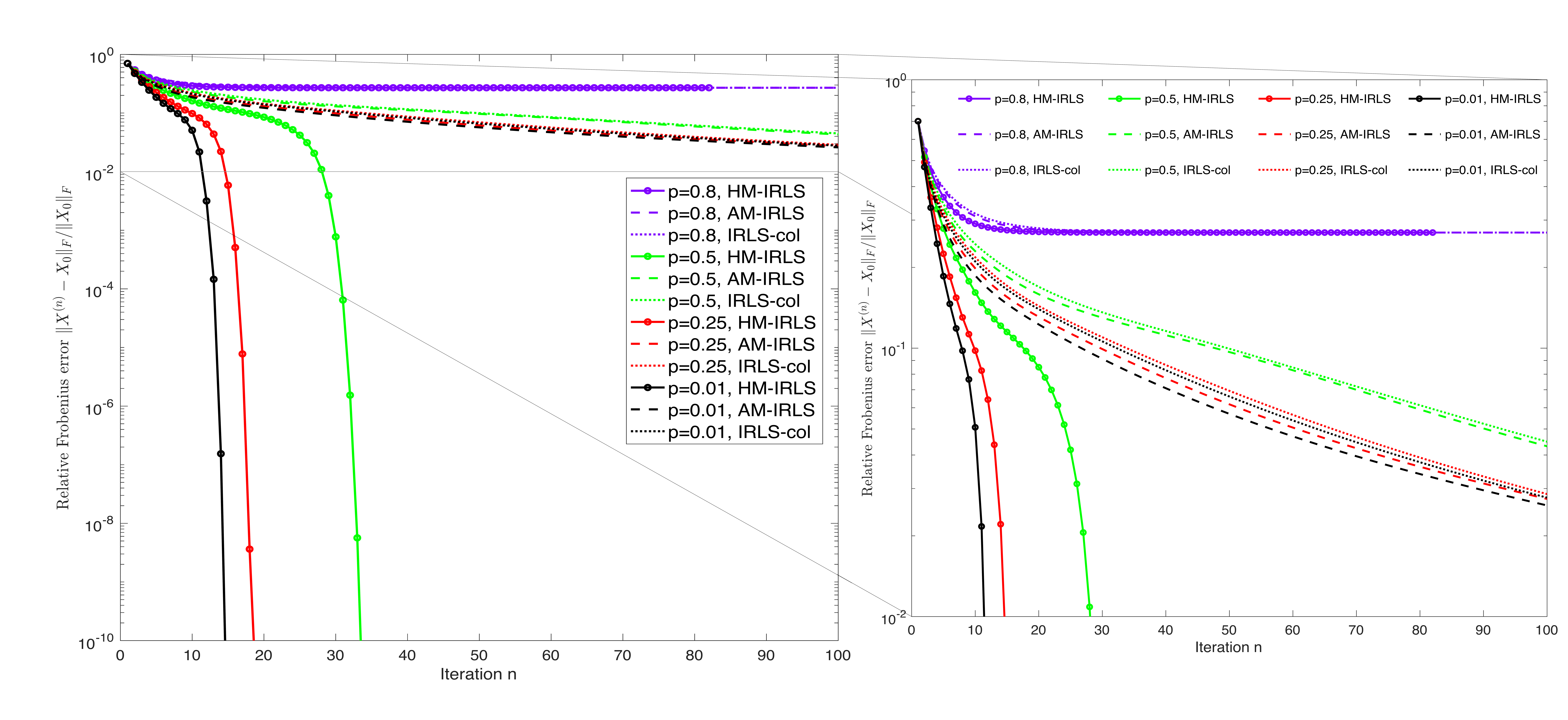}
\vspace*{-8mm}
 \caption{Relative Frobenius errors as a function of the iteration $n$ for oversampling factor $\rho=1.2$ (hard problem). Left column: $y$-range $[10^{-10};10^{0}]$. Right column: Enlarged section of left column corresponding to $y$-range of $[10^{-2};10^{0}]$.} \label{ratesrho12} 
\end{figure}

For the hard matrix completion problems with oversampling factor of $\rho = 1.2$, we observe that for $p=0.8$, the three algorithms typically do not converge to ground truth. This can be seen in the example that is shown in \Cref{ratesrho12}, where \texttt{HM-IRLS}, \texttt{AM-IRLS} and \texttt{IRLS-col} all exhibit a relative error of $0.27$ after $100$ iterations. We do not visualize the result for $p=1$, as the iterates of the three algorithms do not converge to the ground truth either, which is to be expected: In some sense, they implement nuclear norm minimization, which is typically not able to recover a low-rank matrix from measurements with an oversampling factor as small as $\rho = 1.2$ \citep{Donoho13}. The dramatically different behavior between \texttt{HM-IRLS} and the other approaches becomes very apparent for more non-convex choices of $p \in \{0.01,0.25,0.5\}$, where the former converges up to a relative Frobenius error of less than $10^{-10}$ within 15 to 35 iterations, while the others do not reach a relative error of $10^{-2}$ even after $100$ iterations. For \texttt{HM-IRLS}, the convergence of order $2-p$ can be very well locally observed also here, it just takes some iterations until the superlinear convergence begins, which is due to the increased difficulty of the recovery problem.

Finally, we see in the example shown in \Cref{ratesrho1} that even for the very hard problems where $\rho = 1$, which means that the number of sampled entries corresponds exactly to the degrees of freedom $r (d_1 +d_2 -r)$, \texttt{HM-IRLS} can be successful to recover the rank-$r$ matrix if the parameter $p$ is chosen small enough (here: $p \leq 0.25$). This is not the case for the algorithms \texttt{AM-IRLS} and \texttt{IRLS-col}. 

\begin{figure}[t!]
\includegraphics[width=1.00\textwidth]{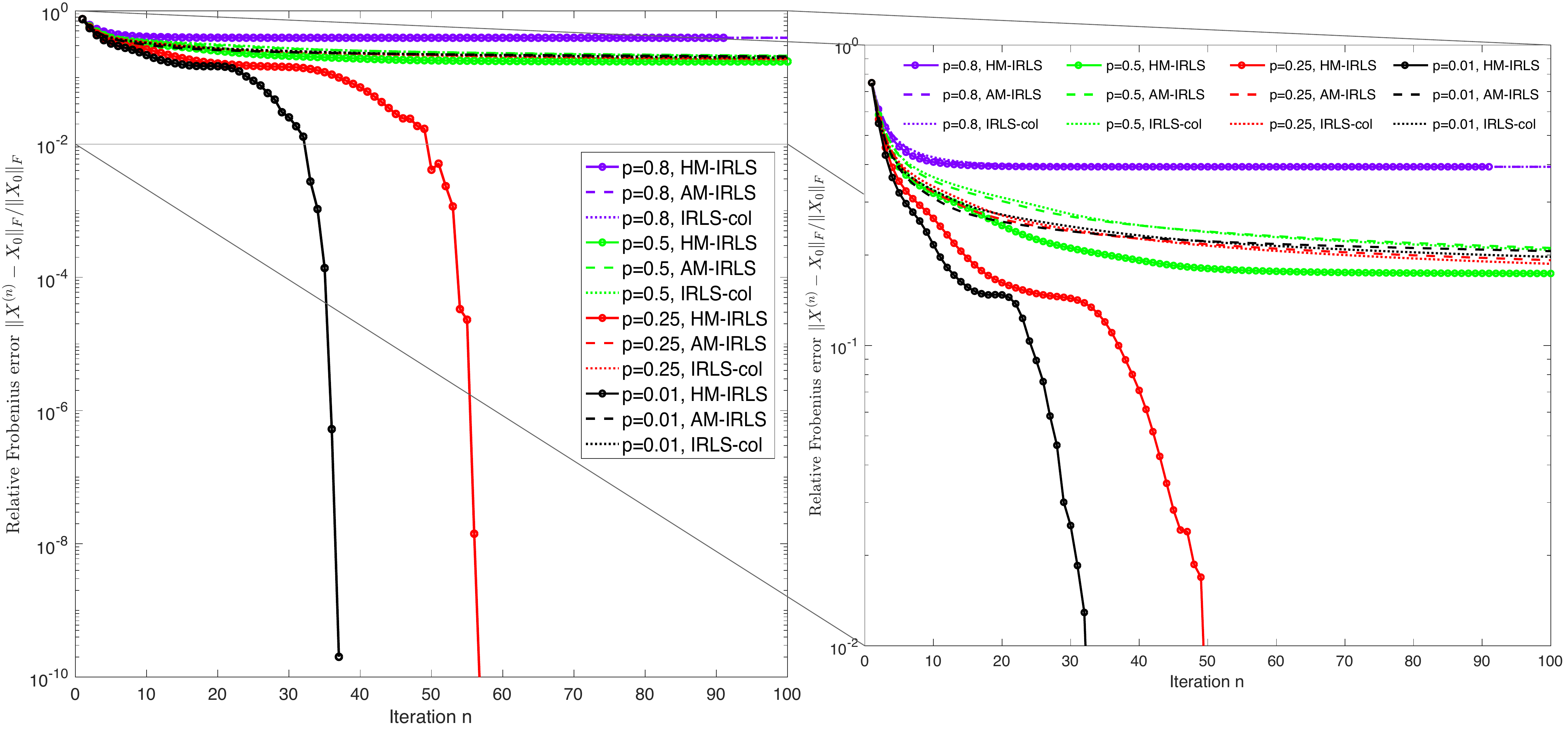}
\vspace*{-8mm}
 \caption{Relative Frobenius errors as a function of the iteration $n$ for oversampling factor $\rho=1.0$ (very hard problem). Left column: $y$-range $[10^{-10};10^{0}]$. Right column: Enlarged section of left column corresponding to $y$-range of $[10^{-2};10^{0}]$.}  \label{ratesrho1}
\end{figure}

\subsubsection{\texttt{HM-IRLS} as the best extension of IRLS for sparse recovery}
We summarize that among the three variants \texttt{HM-IRLS}, \texttt{AM-IRLS} and \texttt{IRLS-col}, only \texttt{HM-IRLS} is able to solve the low-rank matrix recovery problem for very low sample complexities corresponding to $\rho \approx 1$. Furthermore, it is the only IRLS algorithm for low-rank matrix recovery that exhibits a superlinear rate of convergence at all.

It is worthwhile to compare the properties of \texttt{HM-IRLS} with the behavior of the IRLS algorithm of \citep{Daubechies10} designed to solve the sparse vector recovery problem by mimicking $\ell_p$-minimization for $0 < p \leq 1$. While neither \texttt{IRLS-col} nor \texttt{AM-IRLS} are able to generalize the superlinear convergence behavior of \citep{Daubechies10} (which is illustrated in Figure 8.3 of the same paper) to the low-rank matrix recovery problem, \texttt{HM-IRLS} is, as can be seen in \Cref{ratesrho2,ratesrho12,ratesrho1}.

Taking the theoretical guarantees as well as the numerical evidence into account, we claim that \emph{\texttt{HM-IRLS} is the presently best extension of IRLS for vector recovery \citep{Daubechies10} to the low-rank matrix recovery setting}, providing a substantial improvement over the reweighting strategies of \citep{Fornasier11} and \citep{Mohan10}.


Moreover, we mention two observations which suggest that \texttt{HM-IRLS} has in some sense even more favorable properties than the algorithm of \citep{Daubechies10}: First, the discussion of \citep[Section 8]{Daubechies10} states that a superlinear convergence can only be observed locally after a considerable amount of iterations with just a linear error decay. In contrast to that, \texttt{HM-IRLS} exhibits a superlinear error decay quite early (i.e., for example as early as after two iterations), at least if the sample complexity is large enough, cf. \Cref{ratesrho2}.

Secondly, it can be observed that the convergence of the algorithm of \citep{Daubechies10} to a sparse vector often breaks down if $p$ is smaller than $0.5$ \citep[Section 8]{Daubechies10}. In contrast to that, we observe that \texttt{HM-IRLS} does not suffer from this loss of global convergence for $p \ll 0.5$. Thus, a choice of very small parameters $p \approx 0.1$ or smaller is suggested as such a choice is accompanied by a very fast convergence.


\subsection{Recovery performance compared to state-of-the-art algorithms} \label{sec_comp_algs}
After comparing the performance of \texttt{HM-IRLS} with other IRLS variants, we now conduct experiments to compare the empirical performance of \texttt{HM-IRLS} also to that of low-rank matrix recovery algorithms different from IRLS. 

To obtain a comprehensive picture, we consider not only the IRLS variants \texttt{AM-IRLS} and \texttt{IRLS-col}, but a variety of state-of-the-art methods in the experiments, 
as Riemannian optimization technique \texttt{Riemann\_Opt} 
\citep{Vandereycken13},
the alternating minimization approaches \texttt{AltMin} \citep{HaldarH09},
\texttt{ASD} \citep{TannerWei16} and \texttt{BFGD} \citep{Park16}, and finally the algorithms  
\texttt{Matrix ALPS II} \citep{KyrillidisC14} and \texttt{CGIHT\_Matrix} \citep{BlanchardTW15}, which are based on iterative hard thresholding. As the IRLS variants we consider, all these algorthms use knowledge about the true ground truth rank $r$.

In the experiments, we examine the empirical recovery probabilities of the different algorithms systematically for varying oversampling factors $\rho$, determining the difficulty of the low-rank recovery problem as the sample complexity fulfills $m = \lfloor\rho  d_f\rfloor$. We recall that a large parameter $\rho$ corresponds to an easy reconstruction problem, while a small $\rho$, e.g., $\rho \approx 1$, defines a very hard problem.

We choose $d_1=d_2=100$ and the $r=8$ as parameter of the experimental setting, conducting the experiments to recover rank-$8$ matrices $X_0 \in \R^{100 \times 100}$. We remain in the matrix completion measurement setting described in \Cref{sec_measurement_setting}, but sample now $150$ random instances of $X_0$ and $\Phi$ for different numbers of measurements varying between $m_{\min}=1500$ to $m_{\max}=4000$. This means that the oversampling factor $\rho$ increases from $\rho_{\min} = 0.975$ to $\rho_{\max} = 2.60$. For each algorithm, a successful recovery of $X_0$ is defined as a relative Frobenius error $\|X^{\text{out}} - X_0\|_F / \|X_0\|_F$ of the matrix $X^{\text{out}}$ returned by the algorithm of smaller than $10^{-3}$. The algorithms are run until stagnation of the iterates or until the maximal number of iterations $n_{\max}=3000$ is reached. The number $n_{\max}$ is chosen large enough to ensure that a recovery failure is not due to a lack of iterations.


In the experiments, except for \texttt{AltMin}, for which we used our own implementation, we used implementations provided by the authors of the corresponding papers for the respective algorithms, using default input parameters provided by the authors. The respective code sources can be found in the references.



\begin{figure}[t]
    \centering
    \makebox[\linewidth][c]{
          \includegraphics[width=1.0\textwidth]{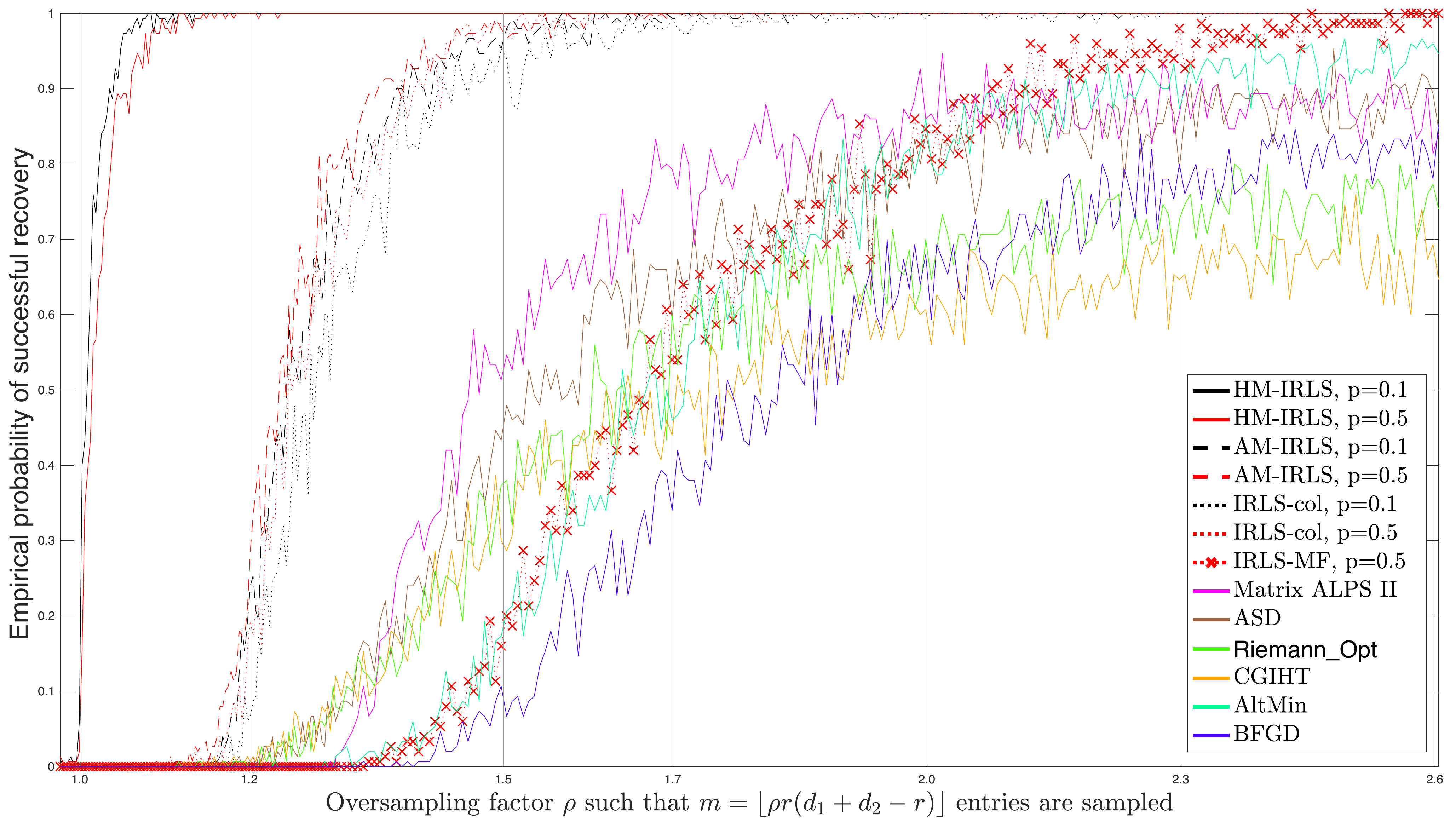}}
         \vspace*{-8mm}
       \caption{Comparison of empirical success rates of state-of-the-art algorithms, as a function of the oversampling factor $\rho$}
       \vspace*{-1mm}
        \label{avg_plot}     
    \end{figure}
    
\subsubsection{Beyond the state-of-the-art performance of \texttt{HM-IRLS}}
The results of the experiment can be seen in \Cref{avg_plot}. We observe that \texttt{HM-IRLS} exhibits a very high empirical recovery probability for $p=0.1$ and $p=0.5$ as soon as the sample complexity parameter $\rho$ is slightly larger than $1.0$, which means that $m = \lfloor \rho r(d_1+d_2-r) \rfloor$ measurements suffice to recover $(d_1 \times d_2)$-dimensional rank-$r$ matrices with $\rho$ close to $1$. This is very close to the information theoretical lower bound of $d_f = r(d_1+d_2-r)$. Very interestingly, it can be observed that the empirical recovery percentage reaches almost $100\%$ already for an oversampling factor of $\rho \approx 1.1$, and remains at exactly $100\%$ starting from $\rho \approx 1.2$.

Quite good success rates can also be observed for the algorithms \texttt{AM-IRLS} and \texttt{IRLS-col} for non-convex parameter choices $p \in \{0.1,0.5\}$, reaching an empirical success probability of almost $100\%$ at around $\rho = 1.5$. \texttt{AM-IRLS} performs only marginally better than the classical IRLS strategy \texttt{IRLS-col}, which are both outperformed considerably by \texttt{HM-IRLS}. It is important to note that in accordance to what was observed in \cref{sec_num_ratecomparisons}, in the successful instances, the error threshold that defines successful recovery is achieved already after a few dozen iterations for \texttt{HM-IRLS}, while typically only after several or many hundreds for \texttt{AM-IRLS} and \texttt{IRLS-col}. Furthermore, it is interesting to observe that the algorithm \texttt{IRLS-MF}, which corresponds to the variant studied and implemented by \citep{Mohan10} and differs from \texttt{IRLS-col} mainly only in the choice of the $\epsilon$-smoothing \cref{epsmin}, has a considerably worse performance than the other IRLS methods. This is plausible since the smoothing influences severely the optimization landscape of the objective to be minimized.

The strong performance of \texttt{HM-IRLS} is in stark contrast to the behavior of all the algorithms that are based on different approaches than IRLS and that we considered in our experiments. They basically never recover any rank-$r$ matrix if $\rho < 1.2$, and most of the algorithms need a sample complexity parameter of $\rho > 1.7$ to exceed a empirical recovery probability of a mere $50 \%$. A success rate of close to $80 \%$ is reached not before raising $\rho$ above $2.0$ in our experimental setting, and also only for a subset of the comparison algorithms, in particular for \texttt{Matrix ALPS II, ASD, AltMin}. The empirical probability of $100\%$ is only reached for some of the IRLS methods, and not for any competing method in our experimental setting, even for quite large oversampling factors such as $\rho = 2.5$. While we do not rule out that a possible parameter tuning could improve the performance of any of the algorithms slightly, we conclude that for hard matrix completion problems, the experimental evidence for the vast differences in the recovery performance of \texttt{HM-IRLS} compared to other methods is very apparent.

Thus, our observation is that the proposed \texttt{HM-IRLS} algorithm \emph{recovers low-rank matrices systematically with nearly the optimal number of measurements and needs fewer measurements than all the state-of-the-art algorithms we included in our experiments}, if the non-convexity parameter $p$ is chosen such that $p \ll 1$.

We also note that the very sharp phase transition between failure and success that can be observed in \Cref{avg_plot} for \texttt{HM-IRLS} indicates that the sample complexity parameter $\rho$ is indeed the major variable determining the success of \texttt{HM-IRLS}. In contrast, the wider phase transitions for the other algorithms suggest that they might depend more on other factors, as the realizations of the random sampling model and the interplay of measurement operator $\Phi$ and ground truth matrix $X_0$.


{Another conclusion that can be drawn from the empirical recovery probability of $1$ is that, despite the severe non-convexity of the underlying Schatten-$p$ quasinorm for, e.g.,  $p=0.1$, \texttt{HM-IRLS} with the initialization of $X^{(1)}$ as the Frobenius norm minimizer does not get stuck in stationary points if the oversampling factor is large enough. Further experiments conducted with random initializations as well as severely adversary initializations, e.g., with starting points chosen in the orthogonal complement of the spaces spanned by the singular vectors of the ground truth matrix $X_0$, lead to comparable results.
Therefore, we claim that \texttt{HM-IRLS} exhibits a global convergence behavior in interesting application cases and for oversampling factor ranges for which competing non-convex low-rank matrix recovery algorithms fail to succeed. We consider a theoretical investigation of such behavior as an interesting open problem to explore. 

\subsection{Computational complexity} \label{section_computationalcomplexity}
While the harmonic mean weight matrix $\widetilde{W}^{(n)}$, cf. \cref{def_Wn_tilde}, is an inverse of a $(d_1 d_2 \times d_1 d_2)$-matrix and therefore in general a dense $(d_1 d_2 \times d_1 d_2)$-matrix, it is important to note that it never has to be computed explicitly in an implementation of \texttt{HM-IRLS}; neither is it necessary to compute its inverse $(\widetilde{W}^{(n)})^{-1} = \frac{1}{2} \left(U^{(n)}(\overline{\Sigma}^{(n)})^{2-p}U^{(n)*} \oplus V^{(n)}(\overline{\Sigma}^{(n)})^{2-p}V^{(n)*}\right)$ explicitly. 

Indeed, as it can be seen in \cref{xmin} and by the definition of the Kronecker sum \cref{def_kronsum}, the harmonic mean weight matrix appears just as the linear operator $(\mathcal{W}^{(n)})^{-1}$ on the space of matrices $M_{d_1 \times d_2}$, whose action consists of a left- and right-sided matrix multiplication, cf. \cref{left_mult_W}. Therefore, the application of $(\mathcal{W}^{(n)})^{-1}$ is $O(d_1 d_2 (d_1+d_2))$ by the naive matrix multiplication algorithm, and can be easily parallelized.

While this useful observation is helpful for the implementation of \texttt{HM-IRLS}, it is not true for \texttt{AM-IRLS}, as the action of $(W_{(\text{arith})}^{(n)})^{-1}$, the inverse of the arithmetic mean weight matrix at iteration $n$, is not representable as a sum of left- and right-sided matrix multiplication. This means that even the execution of a fixed number of iterations of \texttt{HM-IRLS} is faster than computational advantage over \texttt{AM-IRLS}.

The cost to compute $\Phi \circ \widetilde{\mathcal{W}}^{(n)-1}\circ \Phi^* \in M_{m \times m}$ depends on the linear measurement operator $\Phi$. In the matrix completion setting \cref{def_MC_setting}, no additional arithmetic operations have to be performed, as $\Phi$ is a just a selection operator in this case, and for \texttt{HM-IRLS}, this means that $\Phi \circ \widetilde{\mathcal{W}}^{(n)-1}\circ \Phi^*$ is a sparse matrix.

Thus, the algorithm \texttt{HM-IRLS} consists of basically of two computational steps per iteration: The computation of the SVD of the $d_1\times d_2$-matrix $X^{(n)}$ and the solution of the linearly constrained least squares problem in \cref{xmin}. The first is of time complexity $O(d_1d_2 \min(d_1,d_2))$. The time complexity of the second depends on $\Phi$, but is dominated by the inversion of a symmetric, $m \times m$ sparse linear system in the matrix completion setting, if $m$ is the number of given entries. This has a worst case time complexity of $O(\max(d_1,d_2)^3 r^3)$ if $\rho$ is just a constant oversampling factor.

For the matrix completion case, this allows us to recover low-rank matrices up to, e.g., $d_1 =d_2 =3000$ on a single machine given very few entries with \texttt{HM-IRLS}.

\subsubsection*{Acceleration possibilities and extensions}

To tackle higher dimensionalities in reasonable runtimes, a key strategy could be to address the computational bottleneck of \texttt{HM-IRLS}, the solution of the $m \times m$ linear system in \cref{xmin}, by using iterative methods. For IRLS algorithms designed for the related sparse recovery problem, the usage of conjugate gradient (CG) methods is discussed in \citep{Fornasier16}. By coupling the accuracy of the CG solutions to the outer IRLS iteration and using appropriate preconditioning, the authors obtain a competitive solver for the sparse recovery problem, also providing a convergence analysis. Similar ideas could be used for an acceleration of \texttt{HM-IRLS}.


It is interesting to see if further computational improvements can be achieved by combining the ideas of \texttt{HM-IRLS} with the usage of truncated and randomized SVDs \citep{Halko11}, replacing the full SVDs of the $X^{(n)}$ that are needed to define the linear operator $(\mathcal{W}^{(n)})^{-1}$ in \cref{algo1}. 

\section{Theoretical analysis} \label{sec_theoretical_analysis}

For the theoretical analysis of \texttt{HM-IRLS}, we introduce the following auxiliary functional $\mathcal{J}_p$, leading to a variational interpretation of the algorithm. In the whole section, we denote $d=\min(d_1,d_2)$ and $D=\max(d_1,d_2)$.

\begin{definition}\label{J}
Let  $0< p \leq 1$. Given a full rank matrix $Z \in M_{d_1 \times d_2}$, let 
 \begin{equation*}
\widetilde{W}(Z):= 2\big[\mathbf{I}_{d_2} \otimes (ZZ^*)^{\frac{1}{2}}\big]\left[(ZZ^*)^{\frac{1}{2}} \oplus (Z^*Z)^{\frac{1}{2}}\right]^{-1}\big[(Z^*Z)^{\frac{1}{2}} \otimes \mathbf{I}_{d_1}\big] \in H_{d_1 d_2 \times d_1 d_2}
\end{equation*}
be the \emph{harmonic mean matrix $\widetilde{W}$ associated to $Z$}.

We define the \emph{auxiliary functional}  $\mathcal{J}_p: M_{d_1 \times d_2} \times \R_{\geq 0} \times M_{d_1 \times d_2}  \rightarrow \R_{\geq 0}$ as 
\begin{equation*}
\resizebox{.97 \textwidth}{!} 
{$\mathcal{J}_p(X,\epsilon,{Z}):= 
\begin{cases}
\frac{p}{2} \|X_{\vecc}\|^2_{\ell_2(\widetilde{W}(Z))}+ \frac{\epsilon^2 p}{2}\sum\limits^d_{i=1}\sigma_i(Z)+\frac{2-p}{2}\sum\limits^d_{i=1}\sigma_i(Z)^{\frac{p}{(p-2)}}& \text{ if }\rank(Z)=d, \\
 +\infty &  \text{ if }\rank(Z) < d.
\end{cases}$}
 \end{equation*}
\end{definition}

We note that the matrix $\widetilde{W}$ of \cref{J} is just the harmonic mean of the matrices $\widetilde{W}_1:=\mathbf{I}_{d_2} \otimes (ZZ^*)^{\frac{1}{2}}$ and $\widetilde{W}_2=(Z^*Z)^{\frac{1}{2}} \otimes \mathbf{I}_{d_1}$, as introduced in \cref{sec_avg_weightmatrices}, if $(ZZ^*)^{\frac{1}{2}}$ \emph{and} $(Z^*Z)^{\frac{1}{2}}$ are positive definite. Indeed, in this case, $(ZZ^*)^{\frac{1}{2}} \oplus (Z^*Z)^{\frac{1}{2}} = \widetilde{W}_1 + \widetilde{W}_2$ is invertible and as $(A^{-1}+B^{-1})^{-1}= A(A+B)^{-1}B$ for any positive definite matrices $A$ and $B$ of the same dimensions,
\begin{equation} \label{eq_Wtilde_W1_W2}
\widetilde{W}(Z) = 2 \widetilde{W}_1\big(\widetilde{W}_1 + \widetilde{W}_2\big)^{-1} \widetilde{W}_2 =2(\widetilde{W}_1^{-1}+\widetilde{W}_2^{-1})^{-1}.
\end{equation}
{
We use the more general definition $\widetilde{W}(Z)$ as it is well-defined for any full-rank $Z \in M_{d_1 \times d_2}$ and as it allows to handle the case of non-square matrices, i.e., the case $d_1 {\neq} d_2$, as in this case $(ZZ^*)^{\frac{1}{2}}$ or $(Z^*Z)^{\frac{1}{2}}$ 
has to be singular. Using the Moore-Penrose pseudo inverse $\widetilde{W}_1^{+}$ and $\widetilde{W}_2^{+}$ of the matrices $\widetilde{W}_1$ and $\widetilde{W}_2$, we can rewrite $\widetilde{W}(Z)$ from \cref{J} as 
\begin{equation*}
\widetilde{W}(Z) = 2 \widetilde{W}_1\big(\widetilde{W}_1 + \widetilde{W}_2\big)^{-1} \widetilde{W}_2 =2(\widetilde{W}_1^{+}+\widetilde{W}_2^{+})^{-1}.
\end{equation*}}

With the auxiliary functional $\mathcal{J}_p$ at hand, we can interpret \cref{algo1} as an alternating minimization of the functional $\mathcal{J}_p(X,\epsilon,Z)$ with respect to its arguments $X$, $\epsilon$ and $Z$.

In the following, we derive the formula \cref{def_Wn_tilde} for the weight matrix $\widetilde{W}^{(n+1)}$ as the evaluation $\widetilde{W}^{(n+1)}=\widetilde{W}\big(Z^{(n+1)}\big)$ of $\widetilde{W}$ from \cref{J} at the minimizer
\begin{equation}\label{minW}
Z^{(n+1)} = \argmin_{Z \in M_{d_1 \times d_2}}\mathcal{J}_p(X^{(n+1)},\epsilon^{(n+1)},Z),
\end{equation}
with the minimizer being unique. Similarly, the formula \cref{xmin} can be interpreted as
\begin{equation} \label{minX}
X^{(n+1)}  =\argmin\limits_{\substack{X\in M_{d_1 \times d_2}\\ \Phi(X)=Y}}\,\, \|X_{\vecc}\|^2_{\ell_2(\widetilde{W}(Z^{(n)}))}=\argmin\limits_{\substack{X\in M_{d_1 \times d_2}\\ \Phi(X)=Y}}\,\,\mathcal{J}_p(X,\epsilon^{(n)},Z^{(n)})
\end{equation}

These observations constitute the starting point of the convergence analysis of \cref{algo1}, which is detailed subsequently after the verification of the optimization steps.

\subsection{Optimization of \texorpdfstring{$\mathcal{J}_p$}{Jp} with respect to \texorpdfstring{${Z}$}{Z}  and  \texorpdfstring{${X}$}{X}}
We fix $X \in M_{d_1 \times d_2}$ with singular value decomposition $X=\sum^{d}_{i=1}\sigma_i u_iv_i^*$, where $u_i \in \C^{d_1}$ $v_i \in \C^{d_2}$ are the left and right singular vectors respectively and $\sigma_i = \sigma_i(X)$ denote its singular values for $i \in[d]$.

Our objective in the following is the justification of formula \cref{def_Wn_tilde}. To yield the building blocks of the weight matrix $\widetilde{W}^{(n+1)}$,  we consider the minimization problem 
\begin{equation} \label{FpWmin}
\argmin_{Z \in M_{d_1 \times d_2}}\mathcal{J}_p(X,\epsilon,Z)
\end{equation}
for $\epsilon > 0$.

\begin{lemma}\label{Wsubs}
The unique minimizer of \cref{FpWmin} is given by
 \begin{equation*} %
{Z}_{\opt}=\sum^d_{i=1} ({\sigma}_i(X)^2+\epsilon^2)^{\frac{p-2}{2}} u_iv_i^*.
\end{equation*}
Furthermore, the value of $\mathcal{J}_p$ at the minimizer $Z_{\opt}$ is
\begin{equation} \label{Jf}
\begin{split}
\mathcal{J}_p(X,\epsilon,Z_{\opt})= \sum_{i=1}^{d} ({\sigma}_i (X)^2 + \epsilon^2)^{\frac{p}{2}} =:   {g_{\epsilon}^p(X)}
\end{split} 
\end{equation}
for $p > 0$.
\end{lemma}
The proof of this results is detailed in the appendix.
\begin{remark}
We note that the value of $\mathcal{J}_p(X,\epsilon,Z_{\opt})$ can be interpreted as a smooth $\epsilon$-perturbation of a $p$-th power of a Schatten-$p$ quasi-norm of the matrix $X$. In fact, for $\epsilon = 0$ we have
\begin{equation*}
\mathcal{J}_p(X,0,Z_{\opt}) = \|X\|_{S_p}^p = g_{0}^p(X).
\end{equation*}
\end{remark}

Now, we show that our definition rule \cref{xmin} of $X^{(n+1)}$ in \cref{algo1} can be interpreted as a minimization of the auxiliary functional $\mathcal{J}_p$ with respect to the variable $X$. Additionally, this minimization step can be formulated as the solution of a weighted least squares problem with weight matrix $\widetilde{W}^{(n)}$. This is summarized in the following lemma.
\begin{lemma}\label{JminX_prop} Let  $0< p \leq 1$. Given a full-rank matrix $Z \in M_{d_1 \times d_2}$, let 
$\widetilde{W}(Z):= 2([(ZZ^*)^{\frac{1}{2}}]^+\oplus [(Z^*Z)^{\frac{1}{2}}]^+)^{-1} \in H_{d_1d_2 \times d_1d_2}$ be the matrix from \cref{J} and $\mathcal{W}^{-1}: M_{d_1 \times d_2} \rightarrow M_{d_1 \times d_2}$ the linear operator of its inverse 
\[
\mathcal{W}^{-1}(X):= \frac{1}{2}\left[[(ZZ^*)^{\frac{1}{2}}]^+X+X [(Z^*Z)^{\frac{1}{2}}]^+\right].
\]
Then the matrix
$$
X_{\opt} =  \big(\mathcal{W}^{-1}\circ \Phi^*\circ(\Phi \circ \mathcal{W}^{-1}\circ \Phi^*)^{-1}\big)\big(Y\big) \in M_{d_1 \times d_2}
$$ is the unique minimizer of the optimization problems
\begin{equation} \label{eq_opt_WLS}
\argmin\limits_{\Phi(X)=Y}\,\,\mathcal{J}_p(X,\epsilon,Z)=\argmin\limits_{\Phi(X)=Y}\,\, \|X_{\vecc}\|^2_{\ell_2(\widetilde{W})}.
\end{equation}
Moreover, a matrix $X_{\opt} \in M_{d_1\times d_2}$ is a minimizer of the minimization problem \cref{eq_opt_WLS} if and only if it fulfills the property
\begin{equation} \label{eq_proof_JminX_1}
\langle \widetilde{W}(Z)(X_{\opt})_{\vecc},H_{\vecc}\rangle_{\ell_2} = 0 \; \text{ for all } \; H \in \mathcal{N}(\Phi)\;\;\text{ and } \;\; \Phi(X_{\opt}) = Y. 
\end{equation}
\end{lemma}
In \cref{sec_proof_JminX_prop}, the interested reader can find a sketch of the proof of this lemma. 
\subsection{Basic properties of the algorithm and convergence results} \label{sec_preliminary_results}
In the following subsection, we will have a closer look at \cref{algo1} and point out some of its properties, in particular, the boundedness of the iterates $(X^{(n)})_{n\in \mathbb{N}}$ and the fact that two consecutive iterates are getting arbitrarily close as $n\rightarrow \infty$. 
These results will be useful to develop finally the proof of convergence and to determine the rate of convergence of \cref{algo1} under conditions determined along the way.

\begin{lemma}\label{prelim}Let $(X^{(n)},\epsilon^{(n)})_{n \in \N}$ be the sequence of iterates and smoothing parameters of \cref{algo1}. 
Let $X^{(n)}= \sum_{i=1}^{d}  \sigma_i^{(n)}u_i^{(n)} v_i^{(n)*}$ be the SVD of the $n$-th iterate $X^{(n)}$. 
Let $(Z^{(n)})_{n \in \mathbb{N}}$ be a corresponding sequence such that
\[
Z^{(n)} = \sum_{i=1}^{d} (\sigma_i^{(n)2}+\epsilon^{(n)2})^{\frac{p-2}{2}} u_i^{(n)} v_i^{(n)*}
\]
for $n \in \N$. Then the following properties hold:
\begin{enumerate} 
\item[(a)] 
\label{monotonicity}
$
\mathcal{J}_p(X^{(n)},\epsilon^{(n)},{Z}^{(n)})
\geq\mathcal{J}_p(X^{(n+1)},\epsilon^{(n+1)},{Z}^{(n+1)})
$ for all $n \geq 1$,
\item[(b)]
$\label{xbound}
\|X^{(n)}\|^p_{S_p}\leq\mathcal{J}_p(X^{(1)},\epsilon^{(0)},{Z}^{(0)}) =: \mathcal{J}_{p,0}$  for all $n \geq 1$, 
\item[(c)]
\label{cauchyx} The iterates $X^{(n)},X^{(n+1)}$ come arbitrarily close as $n \rightarrow \infty$, i.e., \\
$\lim\limits_{n \rightarrow \infty} \|(X^{(n)}-X^{(n+1)})_{\vecc}\|_{\ell_2}^2 = 0.$
\end{enumerate}
\end{lemma}


At this point we notice that, assuming $X^{(n)} \to \overline{X}$ and $\epsilon^{(n)}\rightarrow \overline{\epsilon}$ for $n \to \infty$ with the limit point $(\overline{X},\overline{\epsilon}
) \in M_{d_1\times d_2} \times \R_{\geq 0}$, it would follow that
\[
\mathcal{J}_p(X^{(n)},\epsilon^{(n)},Z^{(n)}) \to g_{\overline{\epsilon}}^p(\overline{X})
\]
for $n \to \infty$ by equation \cref{Jf}.

Now, let $\epsilon > 0$, a measurement vector $Y \in \C^m$ and the linear operator $\Phi$ be given and consider the optimization problem
\begin{equation}\label{gemin}
\min\limits_{\substack{X \in M_{d_1 \times d_2} \\ \Phi(X)=Y}} g_{\epsilon}^p(X)
\end{equation}
with $g_{\epsilon}^p(X) = \sum_{i=1}^{d} (\sigma_i (X)^2 + \epsilon^2)^{\frac{p}{2}}$ and $\sigma_i(X)$ being the $i$-th singular value of $X$, cf. \cref{Jf}.
If $g_{\epsilon}^p(X)$ is non-convex, which is the case for $p < 1$, one might practically only be able to find critical points of the problem. 
\begin{lemma}\label{charf}
Let $X \in M_{d_1 \times d_2}$ be a matrix with the SVD such that $X =$ \\$\sum_{i=1}^d \sigma_i u_i v_{i}^*$, let $\epsilon > 0$. If we define
\[
\widetilde{W}(X,\epsilon)= 2 \bigg[\Big( \sum_{i=1}^d (\sigma_i^2 + \epsilon^2)^{\frac{2-p}{2}} u_i u_i^* \Big) \oplus \Big(\sum_{i=1}^d (\sigma_i^2 + \epsilon^2)^{\frac{2-p}{2}} v_i v_i^*\Big)\bigg]^{-1} \in H_{d_1 d_2 \times d_1 d_2},
\]
then $\widetilde{W}(X^{(n)},\epsilon^{(n)}) = \widetilde{W}^{(n)}$, with $\widetilde{W}^{(n)}$ defined as in \cref{algo1}, cf. \cref{Wtilde_firstdef}.\\
Furthermore, $X$ is a critical point of the optimization problem \cref{gemin} if and only if 
\begin{equation} \label{eq_charf}
\langle \widetilde{W}(X,\epsilon)X_{\vecc},H_{\vecc}\rangle_{\ell_2} = 0 \; \text{ for all } \; H \in \mathcal{N}(\Phi)\;\;\text{ and } \;\; \Phi(X) = Y. 
\end{equation}
In the case that $g_\epsilon^p$ is convex, i.e., if $p = 1$, \cref{eq_charf} implies that $X$ is the unique minimizer of \cref{gemin}.
\end{lemma}

Now, we have some basic properties of the algorithm at hand that allow us, together with the strong nullspace property in \cref{NSPs} to carry out the proof of the convergence result in \cref{conv}. The proof is sketched in \cref{sec_proof_conv_app} 
 using the results above.
\subsection{Locally superlinear convergence} \label{sec_localconvrate}
In the proof of \cref{rate} we use the following bound on perturbations of the singular value decomposition, which is originally due to Wedin \citep{Wedin72}. It bounds the alignment of the subspaces spanned by the singular vectors of two matrices by their norm distance, given a gap between the first singular values of the one matrix and the last singular values of the other matrix that is sufficiently pronounced. 
%
\begin{lemma}[Wedin's bound \citep{Stewart06}] \label{Wedinbound} 
Let $X$ and $\bar{X}$ be two matrices of the same size and their singular value decompositions 
\begin{align*}
X=\begin{pmatrix}U_1&U_2\end{pmatrix}\begin{pmatrix}\Sigma_1 &0\\0& \Sigma_2 \end{pmatrix}\begin{pmatrix}V_1^*\\V_2^*\end{pmatrix}\quad\text{ and } \quad \bar{D}=\begin{pmatrix}\bar{U}_1&\bar{U}_2\end{pmatrix}\begin{pmatrix}\bar{\Sigma}_1 &0\\0& \bar{\Sigma}_2 \end{pmatrix}\begin{pmatrix}\bar{V}_1^*\\\bar{V}_2^*\end{pmatrix},
\end{align*}
where the submatrices have the sizes of corresponding dimensions. Suppose that $\delta,\alpha$ satisfying $0< \delta\leq \alpha$ are such that $\alpha\leq \sigma_{\min}(\Sigma_1)$ and $\sigma_{\max}(\bar{\Sigma}_2)<\alpha-\delta$. Then 
\begin{equation}\label{Wedin}
\|\bar{U}_2^*U_1\|_{S_\infty} \leq \sqrt{2} \frac{\|X-\bar{X}\|_{S_\infty}}{\delta} \text{  and  }\|\bar{V}_2^*V_1\|_{S_\infty} \leq \sqrt{2}\frac{\|X-\bar{X}\|_{S_\infty}}{\delta}.
\end{equation}
\end{lemma}
%
%

As a first step towards the proof of \cref{rate}, we show the following lemma.
\begin{lemma} \label{lemma_localrate_1}
Let $(X^{(n)})_n$ be the output sequence of \cref{algo1} for parameters $\Phi,Y,r$ and $0 < p \leq 1$, and $X_0 \in M_{d_1 \times d_2}$ be a matrix such that $\Phi (X_0)=Y$.
\begin{enumerate}[(i)]
\item Let $\eta_{2r}^{(n+1)}$ be the best rank-$2r$ approximation of $\eta^{(n+1)}=X^{(n+1)}-X_0$. Then
\begin{equation*}
\|\eta^{(n+1)}-\eta_{2r}^{(n+1)}\|_{S_p}^{2p} \leq 2^{2-p} \bigg(\sum_{i=r+1}^{d} \big(\sigma_i ^2(X^{(n)}) + \epsilon^{(n)2}\big)^{\frac{p}{2}} \bigg)^{2-p} \|\eta_{\vecc}^{(n+1)}\|_{\ell_2(\widetilde{W}^{(n)})}^{2p},
\end{equation*}
where $\widetilde{W}^{(n)}$ denotes the harmonic mean weight matrix from \cref{Wtilde_firstdef}.
\item Assume that the linear map $\Phi: M_{d_1 \times d_2} \to \mathbb{C}^m$ fulfills the strong Schatten-$p$ NSP of order $2r$ with constant $\gamma_{2r} < 1$. Then
\begin{equation}
\begin{split}
\|\eta^{(n+1)}\|_{S_2}^{2p} &\leq 2^p  \frac{\gamma_{2r}^{2-p}}{r^{2-p}} \bigg(\sum_{i=r+1}^{d} \big(\sigma_i ^2(X^{(n)}) + \epsilon^{(n)2}\big)^{\frac{p}{2}} \bigg)^{2-p} \|\eta_{\vecc}^{(n+1)}\|_{\ell_2(\widetilde{W}^{(n)})}^{2p} 
\end{split}
\end{equation}
\item Under the same assumption as for (ii), it holds that
\[
\|\eta^{(n+1)}\|_{S_p}^{2p} \leq (1+\gamma_{2r})^2 2^{2-p} \bigg(\sum_{i=r+1}^{d} \big(\sigma_i ^2(X^{(n)}) + \epsilon^{(n)2}\big)^{\frac{p}{2}} \bigg)^{2-p} \|\eta_{\vecc}^{(n+1)}\|_{\ell_2(\widetilde{W}^{(n)})}^{2p}.
\]
\end{enumerate}
\end{lemma}
\begin{proof}[Proof of \cref{lemma_localrate_1}]
(i) Let the $X^{(n)} = \widetilde{U}^{(n)} \Sigma^{(n)} \widetilde{V}^{(n)*}$ be the (full) singular value decomposition of $X^{(n)}$, i.e., $\widetilde{U}^{(n)} \in \mathcal{U}_{d_1}$ and $\widetilde{V}^{(n)} \in \mathcal{U}_{d_2}$ are unitary matrices and $ \Sigma^{(n)} =\diag(\sigma_1(X^{(n)}),\ldots,\sigma_r(X^{(n)})) \in M_{d_1 \times d_2}$. We define $U_T^{(n)} \in M_{d_1 \times r}$ as the matrix of the first $r$ columns of $\widetilde{U}^{(n)}$ and $U^{(n)}_{T_c} \in M_{d_1 \times (d_1 -r)}$ as the matrix of its last $d_1-r$ columns, so that $\widetilde{U}^{(n)} = \begin{pmatrix} U_T^{(n)} & U_{T_c}^{(n)} \end{pmatrix}$, and similarly $V_T^{(n)}$ and $V_{T_c}^{(n)}$. 

As $\mathbf{I}_{d_1} = U_T^{(n)} U_T^{(n)*} + U_{T_c}^{(n)} U_{T_c}^{(n)*}$ and $\mathbf{I}_{d_2} = V_T^{(n)} V_T^{(n)*} + {V_{T_c}^{(n)}V_{T_c}^{(n)*}}$, we note that
\[
U_{T_c}^{(n)} U_{T_c}^{(n)*}\eta^{(n+1)}V_{T_c}^{(n)} V_{T_c}^{(n)*} = \eta^{(n+1)}-U_T^{(n)} U_T^{(n)*}\eta^{(n+1)} + U_{T_c}^{(n)} U_{T_c}^{(n)*}\eta^{(n+1)}V_T^{(n)} V_T^{(n)*},
\]
while $U_T^{(n)} U_T^{(n)*}\eta^{(n+1)} + U_{T_c}^{(n)} U_{T_c}^{(n)*}\eta^{(n+1)}V_T^{(n)} V_T^{(n)*}$ has a rank of at most $2r$. This implies that 
\begin{equation} \label{eq_eta_best2r}
\|\eta^{(n+1)}-\eta_{2r}^{(n+1)}\|_{S_p} \leq \|U_{T_c}^{(n)} U_{T_c}^{(n)*}\eta^{(n+1)}V_{T_c}^{(n)} V_{T_c}^{(n)*}\|_{S_p} = \|U_{T_c}^{(n)*}\eta^{(n+1)}V_{T_c}^{(n)}\|_{S_p}.
\end{equation}
Using the definitions of $\widetilde{U}^{(n)}$ and $\widetilde{V}^{(n)}$, we write the harmonic mean weight matrices of the $n$-th iteration \cref{Wtilde_firstdef} as
\begin{equation} \label{Wtilde_seconddef}
\widetilde{W}^{(n)} = 2 (\widetilde{V}^{(n)} \otimes \widetilde{U}^{(n)})\big(\overline{\Sigma}_{d_1}^{(n)2 -p} \oplus \overline{\Sigma}_{d_2}^{(n)2 -p}\big)^{-1} (\widetilde{V}^{(n)} \otimes \widetilde{U}^{(n)})^*,
\end{equation}
where $\overline{\Sigma}_{d_1}^{(n)} \in M_{d_1 \times d_1}$ and $\overline{\Sigma}_{d_2}^{(n)} \in M_{d_2 \times d_2}$ are the diagonal matrices with the smoothed singular values of $X^{(n)}$ from \cref{sigmabar}, but filled up with zeros if necessary. Using the abbreviation 
\begin{equation} \label{eq_Omega_def}
\Omega := (\widetilde{V}^{(n)} \otimes \widetilde{U}^{(n)})^*\widetilde{W}^{(n)\frac{1}{2}}\eta^{(n+1)}_{\vecc} \in \C^{d_1 d_2},
\end{equation}
we rewrite
\begin{equation} \label{eq_eta_rewrite}
\begin{split}
\eta_{\vecc}^{(n+1)} &= \widetilde{W}^{(n)-\frac{1}{2}} \widetilde{W}^{(n)\frac{1}{2}}\eta_{\vecc}^{(n+1)} = 2^{-1/2}(\widetilde{V}^{(n)} \otimes \widetilde{U}^{(n)})\big(\overline{\Sigma}_{d_1}^{(n)2 -p} \oplus \overline{\Sigma}_{d_2}^{(n)2 -p}\big)^{1/2} \Omega \\
&= 2^{-1/2}(\widetilde{V}^{(n)} \otimes \widetilde{U}^{(n)}) \left[(\mathbf{I}_{d_2} \otimes \overline{\Sigma}_{d_1}^{(n)\frac{2-p}{2}}) D_L + (\overline{\Sigma}_{d_2}^{(n)\frac{2-p}{2}} \otimes \mathbf{I}_{d_1})D_R \right] \Omega
\end{split}
\end{equation}
with the diagonal matrices $D_L, D_R \in M_{d_1 d_2 \times d_1 d_2}$ such that 
\[
(D_L)_{i+(j-1)d_1,i+(j-1)d_1} =  \Big(1+\Big(\frac{\sigma_j^2(X^{(n)}) + \epsilon^{(n)2}}{\sigma_i^2 (X^{(n)})+\epsilon^{(n)2}}\Big)^{\frac{2-p}{2}}\Big)^{-1/2}
\]
and
\[
(D_R)_{i+(j-1)d_1,i+(j-1)d_1} = \Big(\Big(\frac{\sigma_i^2(X^{(n)}) + \epsilon^{(n)2}}{\sigma_j^2 (X^{(n)})+\epsilon^{(n)2}}\Big)^{\frac{2-p}{2}}+1\Big)^{-1/2}
\]
for $i \in [d_1]$ and $j \in [d_2]$. This can be seen from the definitions of the Kronecker product $\otimes$ and the Kronecker sum $\oplus$ (cf. \cref{sec_kroneckerhadamard}), as 
\begin{equation*} \begin{split}
&\Big(\big(\overline{\Sigma}_{d_1}^{(n)2 -p} \oplus \overline{\Sigma}_{d_2}^{(n)2 -p}\big)^{1/2}\Big)_{i+(j-1)d_1,i+(j-1)d_1} = (s_{i}+s_{j})^{1/2}\\
 &= s_i (s_i + s_j)^{-1/2} + s_j (s_i + s_j)^{-1/2} = s_i^{1/2} ( 1 + \frac{s_{j}}{s_{i}})^{-1/2} + s_{j}^{1/2} (\frac{s_{i}}{s_{j}}+1)^{-1/2} 
 \end{split}
\end{equation*}
if $s_\ell$ denotes the $\ell$-th diagonal entry of $\overline{\Sigma}_{d_2}^{(n)2 -p}$ and $\overline{\Sigma}_{d_1}^{(n)2 -p}$ for $\ell \in [\max(d_1,d_2)]$.

If we write $\overline{\Sigma}_{d_1,T_c}^{(n)\frac{2-p}{2}} \in M_{(d_1-r) \times (d_1 -r)}$ for the diagonal matrix containing the $d_1-r$ last diagonal elements of $\overline{\Sigma}_{d_1}^{(n)2 -p}$ and $\overline{\Sigma}_{d_2,T_c}^{(n)\frac{2-p}{2}} \in M_{(d_1-r) \times (d_1 -r)}$ for the diagonal matrix containing the $d_2-r$ last diagonal elements of $\overline{\Sigma}_{d_2}^{(n)2 -p}$, it follows from \cref{eq_eta_rewrite} that
\begin{equation*}
\resizebox{ \textwidth}{!}{ 
$\begin{split}
 \big\|U_{T_c}^{(n)*}&\eta^{(n+1)}V_{T_c}^{(n)}\big\|^{p}_{S_p}\!\! =\! 2^{-\frac{p}{2}}\Big\| U_{T_c}^{(n)*} \widetilde{U}^{(n)}\!\! \left[\overline{\Sigma}_{d_1}^{(n)\frac{2-p}{2}}\! (D_L \Omega)_{\mat}\! + \!(D_R \Omega)_{\mat} \overline{\Sigma}_{d_2}^{(n)\frac{2-p}{2}}\right]\! \widetilde{V}^{(n)*}V_{T_c}^{(n)}\Big\|_{S_p}^{p} \\
&\leq 2^{-\frac{p}{2}} \Big\| \overline{\Sigma}_{d_1,T_c}^{(n)\frac{2-p}{2}}\big[(D_L \Omega)_{\mat}\big]_{T_c,T_c}\Big\|_{S_p}^p + \Big\|\big[(D_R \Omega)_{\mat}\big]_{T_c,T_c} \overline{\Sigma}_{d_2,T_c}^{(n)\frac{2-p}{2}}\Big\|_{S_p}^p 
\end{split}$}
\end{equation*}
with the notation that $M_{T_c,T_c}$ denotes the submatrix of $M$ which contains the intersection of the last $d_1-r$ rows of $M$ with its last $d_2-r$ columns. 

Now, H\"older's inequality for Schatten-$p$ quasinorms (e.g., \citep[Theorem 11.2]{Gohberg2000}) can be used to see that
\begin{equation} \label{eq_T_1_def}
\Big\| \overline{\Sigma}_{d_1,T_c}^{(n)\frac{2-p}{2}}\big[(D_L \Omega)_{\mat}\big]_{T_c,T_c}\Big\|_{S_p}^p \leq 
\Big\|\overline{\Sigma}_{T_c}^{(n)\frac{2-p}{2}}\Big\|_{S_{\frac{2p}{2-p}}}^{p} \Big\|\big[(D_L \Omega)_{\mat}\big]_{T_c,T_c} \Big\|_{S_2}^p.
\end{equation}
Inserting the definition 
\[
\Big\|\overline{\Sigma}_{T_c}^{(n)\frac{2-p}{2}}\Big\|_{S_{\frac{2p}{2-p}}}^{p} \!\!\!\! =\bigg(\sum_{i=r+1}^{d} \big(\sigma_i ^2(X^{(n)}) + \epsilon^{(n)2}\big)^{\frac{2p(2-p)}{(2-p)4}} \bigg)^{\frac{2-p}{2}}\!\!\!\! =\!  \bigg(\sum_{i=r+1}^{d}\!\! \big(\sigma_i ^2(X^{(n)}) \!+ \epsilon^{(n)2}\big)^{\frac{p}{2}} \bigg)^{\frac{2-p}{2}} \\
\] 
allows us to rewrite the first factor, while the second factor can be bounded by
\begin{equation*}
\begin{split}
\big\|\big[(D_L \Omega)_{\mat}\big]_{T_c,T_c}\big\|_{S_2}^p &\leq \big\|(D_L \Omega)_{\mat}\big\|_{S_2}^p \leq \|\Omega_{\mat}\|_{S_2}^p = \|(\widetilde{V}^{(n)} \otimes \widetilde{U}^{(n)})^*\widetilde{W}^{(n)\frac{1}{2}}\eta^{(n+1)}_{\vecc}\|_{\ell_2}^p \\
&= \|\widetilde{W}^{(n)\frac{1}{2}}\eta^{(n+1)}_{\vecc}\|_{\ell_2}^p = \|\eta_{\vecc}^{(n+1)}\|_{\ell_2(\widetilde{W}^{(n)})}^p,
\end{split} 
\end{equation*}
as the matrix $D_L \in M_{d_1 d_2 \times d_1 d_2}$ from \cref{eq_eta_rewrite} fulfills $\|D_L\|_{S_\infty} \leq 1$ since its entries are bounded by $1$; we also recall the definition \cref{eq_Omega_def} of $\Omega$ and that $\widetilde{V}^{(n)}$ and $\widetilde{U}^{(n)}$ are unitary.

The term $\Big\|\big[(D_R \Omega)_{\mat}\big]_{T_c,T_c} \overline{\Sigma}_{d_2,T_c}^{(n)\frac{2-p}{2}}\Big\|_{S_p}^p$ in the bound of $\big\|U_{T_c}^{(n)*}\eta^{(n+1)}V_{T_c}^{(n)}\big\|^{p}_{S_p}$ can be estimated analogously. Combining this with \cref{eq_eta_best2r}, we obtain
\begin{equation*}
\|\eta^{(n+1)}-\eta_{2r}^{(n+1)}\|_{S_p}^{2p} \leq 2^{-p}\Big(2 \Big(\sum_{i=r+1}^{d} \big(\sigma_i ^2(X^{(n)}) + \epsilon^{(n)2}\big)^{\frac{p}{2}} \Big)^{\frac{2-p}{2}} \Big)^2 \|\eta_{\vecc}^{(n+1)}\|_{\ell_2(\widetilde{W}^{(n)})}^{2p}, 
\end{equation*}
concluding the proof of statement (i).
(ii) Using the strong Schatten-$p$ null space property \cref{k-NSP} of order $2r$ and that $\eta^{(n+1)} \in \mathcal{N}(\Phi)$, we estimate
\begin{equation*} 
\begin{split}
\|\eta^{(n+1)}\|^{2p}_{S_2}\! &= \!\big(\|\eta_{2r}^{(n+1)}\|_{S_2}^2\!\!+\!\|\eta^{(n+1)}\!\!-\eta_{2r}^{(n+1)}\|_{S_2}^2\big)^{p}\!\!  \leq  \!\Big(\frac{\gamma_{2r}^{2/p}\!+\gamma_{2r}^{2/p-1}}{(2r)^{2/p-1}}\|\eta^{(n+1)}\!\!-\eta_{2r}^{(n+1)}\|_{S_p}^2\Big)^{p} \\
& \leq \frac{\gamma_{2r}^{2-p}(\gamma_{2r}+1)^{p}}{(2r)^{2-p}} \|\eta^{(n+1)}-\eta_{2r}^{(n+1)}\|_{S_p}^{2p}\leq 2^{p} \frac{\gamma_{2r}^{2-p}}{2^{2-p} r^{2-p}}  \|\eta^{(n+1)}-\eta_{2r}^{(n+1)}\|_{S_p}^{2p},
\end{split}
\end{equation*}
where we use in the second inequality a version of Stechkin's lemma \citep[Lemma 3.1]{Kabanava15}, which leads to the estimate
\[
\|\eta^{(n+1)}-\eta_{2r}^{(n+1)}\|_{S_2}^2 \! \leq \! \frac{\|\eta^{(n+1)}_{2r}\|_{S_2}^{2-p}}{(2r)^{2-p}}\|\eta^{(n+1)}-\eta^{(n+1)}_{2r}\|_{S_p}^p \! \leq \! \frac{\gamma_{2r}^{2/p-1}}{(2r)^{2/p-1}}\|\eta^{(n+1)}-\eta^{(n+1)}_{2r}\|_{S_p}^2.
\]
Combining the estimate for $\|\eta^{(n+1)}\|^{2p}_{S_2}$ with statement (i), this results in 
\[
\|\eta^{(n+1)}\|^{2p}_{S_2}  \leq 2^p  \frac{\gamma_{2r}^{2-p}}{r^{2-p}} \bigg(\sum_{i=r+1}^{d} \big(\sigma_i ^2(X^{(n)}) + \epsilon^{(n)2}\big)^{\frac{p}{2}} \bigg)^{2-p} \|\eta_{\vecc}^{(n+1)}\|_{\ell_2(\widetilde{W}^{(n)})}^{2p},
\]
which shows statement (ii).

(iii) For the third statement, we use the strong Schatten-$p$ NSP \cref{k-NSP} to see that
\[
\|\eta^{(n+1)}\|_{S_p}^{p} = \|\eta_{2r}^{(n+1)}\|_{S_p}^{p} + \|\eta^{(n+1)}-\eta_{2r}^{(n+1)}\|_{S_p}^{p} \leq (1+\gamma_{2r}) \|\eta^{(n+1)}-\eta_{2r}^{(n+1)}\|_{S_p}^{p},
\]
and combine this with statement (i).
\end{proof}
\begin{lemma} \label{lemma_localrate_2}
Let $(X^{(n)})_n$ be the output sequence of \cref{algo1} with parameters $\Phi,Y,r$ and $0 < p \leq 1$, and $\widetilde{W}^{(n)}$ be the harmonic mean weight matrix matrix \cref{Wtilde_firstdef} for $n \in \N$. Let $X_0 \in M_{d_1 \times d_2}$ be a rank-$r$ matrix such that $\Phi (X_0) =Y$ with condition number  $\kappa:= \frac{\sigma_1 (X_0)}{\sigma_r (X_0)}$. \begin{enumerate}[(i)]
\item If \cref{rho} is fulfilled for iteration $n$, then $\eta^{(n+1)}=X^{(n)} - X_0$ fulfills
\[
\left\|\eta^{(n+1)}_{\vecc}\right\|^{2p}_{\ell_2(\widetilde{W}^{(n)})} \leq\frac{4^p r^{p/2} \sigma_r(X_0)^{p(p-1)}}{(1-\zeta)^{2p}} \kappa^p \frac{\|\eta^{(n)}\|_{S_\infty}^{2p-p^2}}{(\epsilon^{(n)})^{2p-p^2}} \|\eta^{(n+1)}\|_{S_2}^p.
\]
\item Under the same assumption as for (i), it holds that
\[
\left\|\eta^{(n+1)}_{\vecc}\right\|^{2p}_{\ell_2(\widetilde{W}^{(n)})} \leq \frac{7^p r^{p/2}\max(r,d-r)^{p/2} \sigma_r(X_0)^{p(p-1)}}{(1-\zeta)^{2p}} \kappa^p \frac{\|\eta^{(n)}\|_{S_\infty}^{2p-p^2}}{(\epsilon^{(n)})^{2p-p^2}} \|\eta^{(n+1)}\|_{S_\infty}^p
\]
\end{enumerate}
\end{lemma}
\begin{proof}[Proof of \cref{lemma_localrate_2}]
(i) Recall that $X^{(n+1)} = \argmin\limits_{\Phi(X)=Y} \|X_{\vecc}\|^2_{\ell_2(\widetilde{W}^{(n)})}$ is the minimizer of the weighted least squares problem with weight matrix $\widetilde{W}^{(n)}$. As $\eta^{(n+1)} = X^{(n+1)}-X_0$ is in the null space of the measurement map $\Phi$, it follows from \cref{JminX_prop} that
\[
0=\langle \widetilde{W}^{(n)}X^{(n+1)}_{\vecc},\eta^{(n+1)}_{\vecc}\rangle=\langle \widetilde{W}^{(n)}(\eta^{(n+1)}+X_{0})_{\vecc},\eta^{(n+1)}_{\vecc}\rangle,
\]
which is equivalent to 
\begin{equation*} 
\left\|\eta^{(n+1)}_{\vecc}\right\|^{2}_{\ell_2(\widetilde{W}^{(n)})}=\langle \widetilde{W}^{(n)}\eta^{(n+1)}_{\vecc},\eta^{(n+1)}_{\vecc}\rangle=-\langle \widetilde{W}^{(n)}(X_{0})_{\vecc},\eta^{(n+1)}_{\vecc}\rangle.
\end{equation*}
Using H\"older's inequality, we can therefore estimate
\begin{equation} \label{proof_superlinearcvg_2}
\begin{split}
\left\|\eta^{(n+1)}_{\vecc}\right\|^{2}_{\ell_2(\widetilde{W}^{(n)})}&=-\langle \widetilde{W}^{(n)}(X_{0})_{\vecc},\eta^{(n+1)}_{\vecc}\rangle_{\ell_2} = -\langle [\widetilde{W}^{(n)}(X_{0})_{\vecc}]_{\mat},\eta^{(n+1)}\rangle_{F} \\
&\leq \big\|\big[\widetilde{W}^{(n)}(X_{0})_{\vecc}\big]_{\mat}\big\|_{S_2}\|\eta^{(n+1)}\|_{S_2}.
\end{split}
\end{equation}
To bound the first factor, we first rewrite the action of $\widetilde{W}^{(n)}$ on $X_0$ in the matrix space as
\begin{equation*}
\begin{split}
\left[\widetilde{W}^{(n)}\!(X_{0})_{\vecc}\right]_{\mat}\!\!\!\! &=  2[(\widetilde{V}^{(n)}\! \otimes \widetilde{U}^{(n)})\big(\overline{\Sigma}_{d_1}^{(n)2 -p}\! \oplus \overline{\Sigma}_{d_2}^{(n)2 -p}\big)^{-1} (\widetilde{V}^{(n)}\! \otimes \widetilde{U}^{(n)})^* (X_{0})_{\vecc}]_{\mat}\! = \\
&=\widetilde{U}^{(n)}\big(H^{(n)} \circ (\widetilde{U}^{(n)*} X_0 \widetilde{V}^{(n)})\big)\widetilde{V}^{(n)*},
\end{split}
\end{equation*}
using \cref{Wtilde_seconddef} and \cref{lemma_localrate_1} about the action of inverses of Kronecker sums, with the notation that $H^{(n)} \in M_{d_1 \times d_2}$ such that
\begin{equation*}
H^{(n)}_{ij}=2 \left[\mathds{1}_{\{i \leq d\}}(\sigma_i^{2}(X^{(n)})+\epsilon^{(n)2})^{\frac{2-p}{2}}+\mathds{1}_{\{j \leq d\}}(\sigma_j^{2}(X^{(n)})+\epsilon^{(n)2})^{\frac{2-p}{2}}\right]^{-1}
\end{equation*}
for $i \in [d_1]$, $j \in [d_2]$, where $\mathds{1}_{\{i \leq d\}} = 1$ if $i \leq d$ and $\mathds{1}_{\{i \leq d\}} = 0$ otherwise. This enables us to estimate
\begin{equation*}
\resizebox{ \textwidth}{!} 
{$
\left\|\big[\widetilde{W}^{(n)}(X_{0})_{\vecc}\big]_{\mat}\right\|_{S_2}^2 \!\!\! =\! \left\|\widetilde{U}^{(n)}\big(H^{(n)}\!\! \circ (\widetilde{U}^{(n)*} X_0 \widetilde{V}^{(n)})\big)\widetilde{V}^{(n)*} \right\|_{S_2}^2 \!\!\! =\! \left\|H^{(n)}\!\!\circ (\widetilde{U}^{(n)^*} \!X_0 \widetilde{V}^{(n)})\right\|_{S_2}^2$}
\end{equation*}
\begin{equation} \label{eq_WtildeX0}
\begin{split}
\phantom{\quad\quad\quad\quad\quad\;}&= \left\|H^{(n)}\circ \begin{pmatrix} U^{(n)*}_T X_0 V^{(n)}_{T} & U^{(n)*}_T X_0 V^{(n)}_{T_c} \\ U^{(n)^*}_{T_c}X_0 V^{(n)}_{T}& U^{(n)*}_{T_c}X_0 V^{(n)}_{T_c}\end{pmatrix} \right\|_{S_2}^2 \\
&= \left\|H^{(n)}_{T,T}\circ  (U^{(n)*}_T X_0 V^{(n)}_{T})\right\|_{S_2}^2 +\left\|H^{(n)}_{T,T_c}\circ (U^{(n)*}_T X_0 V^{(n)}_{T_c})\right\|_{S_2}^2 \\  
&+ \left\|H^{(n)}_{T_c,T}\circ(U^{(n)*}_{T_c} X_0 V^{(n)}_{T})\right\|_{S_2}^2+\left\|H^{(n)}_{T_c,T_c}\circ(U^{(n)*}_{T_c} X_0 V^{(n)}_{T_c})\right\|_{S_2}^2,
\end{split}
\end{equation}
using the notation from the proof of \cref{lemma_localrate_1}. To bound the first summand, we calculate
\begin{equation*} 
\begin{split}
&\left\|H^{(n)}_{T,T}\!\circ\!  (U^{(n)*}_T\! X_0 V^{(n)}_{T})\right\|_{S_2} \!\! \leq \! \left\|H^{(n)}_{T,T}\!\circ\!  (U^{(n)*}_T\! X^{(n)} V^{(n)}_{T})\right\|_{S_2} \!\!\! +\! \left\|H^{(n)}_{T,T}\!\circ\!  (-U^{(n)*}_T \eta^{(n)} V^{(n)}_{T})\right\|_{S_2} \\
&\leq \left\|H^{(n)}_{T,T} \circ \Sigma_T^{(n)}\right\|_{S_2} +  \left\|H^{(n)}_{T,T}\circ  (U^{(n)*}_T \eta^{(n)} V^{(n)}_{T})\right\|_{S_2} \\
&\leq \bigg(\sum_{i=1}^r \frac{ \sigma_i^2 (X^{(n)})}{\big(\sigma_i^{2}(X^{(n)})+\epsilon^{(n)2}\big)^{2-p}}\bigg)^{1/2}  + \max_{i,j=1}^r |H_{i,j}^{(n)}| \|U^{(n)*}_T \eta^{(n)} V^{(n)}_{T}\|_{S_2} \\
&\leq \sqrt{r} \sigma_r^{p-1}(X^{(n)}) + (\sigma_r^{2}(X^{(n)})+\epsilon^{(n)2})^{\frac{p-2}{2}} \|U^{(n)*}_T \eta^{(n)} V^{(n)}_{T}\|_{S_2} \\
&\leq \sqrt{r} \sigma_r^{p-1}(X^{(n)}) +\sigma_r^{p-2}(X^{(n)})\sqrt{r} \| \eta^{(n)}\|_{S_\infty} = \sqrt{r} \sigma_r^{p-2}(X^{(n)}) \big[\sigma_r(X^{(n)})+ \|\eta^{(n)}\|_{S_\infty} \big],
\end{split}
\end{equation*}
denoting $\Sigma_T^{(n)} = \diag(\sigma_i (X^{(n)}))_{i=1}^r$ and that the matrices  $U_T^{(n)}$ and $V_T^{(n)}$ contain the first $r$ left resp. right singular vectors of $X^{(n)}$ in the second inequality, together with the estimates $\|X\|_{S_1} \leq \sqrt{r} \|X\|_{S_2} \leq r \|X\|_{S_\infty}$ for $(r \times r)$-matrices $X$.

With the notations $s_r^0 := \sigma_r (X_0)$ and $s_1^0 := \sigma_1 (X_0)$, we note that 
\[
\sigma_r (X^{(n)}) \geq s_r^0 (1-\zeta),
\]
as the assumption \cref{rho} implies that 
\begin{equation*}
s_r^0= \sigma_r(X_0) = \sigma_r(X^{(n)}-\eta^{(n)}) \leq \sigma_r(X^{(n)}) + \sigma_1(\eta^{(n)}) \leq \sigma_r(X^{(n)})+\zeta s_r^0,
\end{equation*}
using \citep[Proposition 9.6.8]{Bernstein09} in the first inequality.

Therefore, we can bound the first summand of \cref{eq_WtildeX0} such that
\begin{equation} \label{eq_H_TT}
\left\|H^{(n)}_{T,T}\circ  (U^{(n)*}_T X_0 V^{(n)}_{T})\right\|_{S_2} \leq \sqrt{r} (s_r^0 (1-\zeta))^{p-2}[s_r^0 (1-\zeta) + \zeta s_r^0] = \sqrt{r} (s_r^0)^{p-1} (1-\zeta)^{p-2}. 
\end{equation}
For the second summand in the estimate of $\left\|\big[\widetilde{W}^{(n)}(X_{0})_{\vecc}\big]_{\mat}\right\|_{S_2}^2$, similar arguments and again assumption \cref{rho} are used to compute \vspace{-2mm}
\begin{equation} \label{eq_H_TTc}
\begin{split}
&\left\|H^{(n)}_{T,T_c}\circ (U^{(n)*}_T X_0 V^{(n)}_{T_c})\right\|_{S_2} \leq \Big\|H^{(n)}_{T,T_c}\circ \overbrace{(U^{(n)*}_T X^{(n)} V_{T_c}^{(n)})}^{=0}\Big\|_{S_2} + \\
& \left\|H^{(n)}_{T,T_c}\circ (U^{(n)*}_T \eta^{(n)} V_{T_c}^{(n)})\right\|_{S_2} \leq \max_{\substack{i\in[r] \\ j\in \{r+1,\ldots,d_2\}}} |H_{i,j}^{(n)}| \|U^{(n)*}_T \eta^{(n)} V_{T_c}^{(n)}\|_{S_2} \\
& \leq  \frac{2 \|U^{(n)*}_T \eta^{(n)} V_{T_c}^{(n)}\|_F}{\big[(\sigma_r (X^{(n)})^2 + \epsilon^{(n)2})^{\frac{2-p}{2}}\big]} \leq 2 \sigma_r (X^{(n)})^{p-2} \|U^{(n)*}_T \eta^{(n)} V_{T_c}^{(n)}\|_{S_2} \\
&\leq 2\sqrt{r}(s_r^0 (1-\zeta))^{p-2} \|\eta^{(n)}\|_{S_\infty} \leq 2\zeta \sqrt{r} (s_r^0)^{p-1} (1-\zeta)^{p-2}.
\end{split}
\end{equation}
From exactly the same arguments it follows that also
\begin{equation} \label{eq_H_TcT}
\left\|H^{(n)}_{T_c,T}\circ(U^{(n)*}_{T_c} X_0 V^{(n)}_{T})\right\|_{S_2} \leq 2\zeta \sqrt{r} (s_r^0)^{p-1} (1-\zeta)^{p-2}.
\end{equation}
It remains to bound the last summand $\left\|H^{(n)}_{T_c,T_c}\circ(U^{(n)*}_{T_c} X_0 V^{(n)}_{T_c})\right\|_{S_2}^2$. We see that 
\begin{equation} \label{eq_H_TcTc}
\begin{split}
&\left\|H^{(n)}_{T_c,T_c}\circ(U^{(n)*}_{T_c} X_0 V^{(n)}_{T_c})\right\|_{S_2} \leq \max_{\substack{i\in\{r+1,\ldots,d_1\} \\ j\in \{r+1,\ldots,d_2\}}} \big|H_{i,j}^{(n)}\big| \|U^{(n)*}_{T_c} X_0 V^{(n)}_{T_c}\|_{S_2} \\
&\leq (\epsilon^{(n)})^{p-2} \|U^{(n)*}_{T_c} X_0 V^{(n)}_{T_c}\|_{S_2} \leq  (\epsilon^{(n)})^{p-2} \| U_{T_c}^{(n)*} U_T^0\|_{S_\infty} \|S^0\|_{S_2} \|V_T^{0*} V_{T_c}^{(n)}\|_{S_\infty} \\
&\leq (\epsilon^{(n)})^{p-2} \frac{\sqrt{2}\|\eta^{(n)}\|_{S_\infty}}{(1-\zeta) s_r^0} \sqrt{r} s_1^0 \frac{\sqrt{2}\|\eta^{(n)}\|_{S_\infty}}{(1-\zeta) s_r^0} = 2\sqrt{r} \|\eta^{(n)}\|_{S_\infty}^2(\epsilon^{(n)})^{p-2} (1-\zeta)^{-2} (s_r^0)^{-1} \frac{s_1^0}{s_r^0}   \\
 \end{split}
\end{equation}
where H\"older's inequality for Schatten norms was used in the third inequality. In the fourth inequality, Wedin's singular value perturbation bound of \cref{Wedinbound} is used with the choice $Z= X_0$, $\overline{Z}= X^{(n)}$, $\alpha= s_r^0$ and $\delta=(1-\zeta)s_r^0$, and finally $\epsilon^{(n)} \leq \zeta s_r^0$ in the last inequality, which is implied by the rule \cref{epsmin} for $\epsilon^{(n)}$ together with assumption \cref{rho}.

Summarizing the estimates \cref{eq_H_TT,eq_H_TTc,eq_H_TcT,eq_H_TcTc}, we conclude that
\begin{equation*}
\begin{split}
&\left\|\big[\widetilde{W}^{(n)}(X_{0})_{\vecc}\big]_{\mat}\right\|_{S_2}^2 \leq \frac{r(s_r^0)^{2p-2}}{(1-\zeta)^{4-2p}} \bigg[1+8\zeta^2 +  4 \frac{\|\eta^{(n)}\|_{S_\infty}^{4}}{(1-\zeta)^{2p}} (\epsilon^{(n)})^{2p-4}(s_r^0)^{-2p} \left(\frac{s_1^0}{s_r^0}\right)^2 \bigg] \\
&= \frac{r (s_r^0)^{2p-2}}{(1-\zeta)^4} \bigg[(1+8\zeta^2)(1-\zeta)^{2p} + 4 \frac{\|\eta^{(n)}\|_{S_\infty}^{4-2p}}{(\epsilon^{(n)})^{4-2p}}\frac{{\|\eta^{(n)}\|_{S_\infty}^{2p}}}{(s_r^0)^{2p}} \left(\frac{s_1^0}{s_r^0}\right)^2 \bigg] \\
&\leq \frac{r (s_r^0)^{2p-2}}{(1-\zeta)^4} \bigg[ 9+ 4  \frac{\|\eta^{(n)}\|_{S_\infty}^{4-2p}}{(\epsilon^{(n)})^{4-2p}} \zeta^{2p} \kappa^2\bigg]  \leq \frac{13 r (s_r^0)^{2p-2}}{(1-\zeta)^4} \bigg[ \frac{\|\eta^{(n)}\|_{S_\infty}^{4-2p}}{(\epsilon^{(n)})^{4-2p}} \kappa^2\bigg],
\end{split}
\end{equation*}
as $0< \zeta <1$, $\epsilon^{(n)} \leq \sigma_{r+1}(X^{(n)}) = \|X_{T_c}^{(n)}\|_{S_\infty} \leq \|\eta^{(n)}\|_{S_\infty}$ and using the assumption \cref{rho} in the second inequality. This concludes the proof of \cref{lemma_localrate_2}(i) together with inequality \cref{proof_superlinearcvg_2} as $13^{p/2} \leq 16^{p/2} = 4^p$.

(ii) For the second statement of \cref{lemma_localrate_2}, we proceed similarly as before, but note that by Hölder's inequality, also 
\begin{equation}
\left\|\eta^{(n+1)}_{\vecc}\right\|^{2}_{\ell_2(\widetilde{W}^{(n)})} \leq \big\|\big[\widetilde{W}^{(n)}(X_{0})_{\vecc}\big]_{\mat}\big\|_{S_1}\|\eta^{(n+1)}\|_{S_\infty},
\end{equation}
cf. \cref{proof_superlinearcvg_2}. Furthermore
\begin{equation}
\begin{split}
\big\|\big[\widetilde{W}^{(n)}(X_{0})_{\vecc}\big]_{\mat}\big\|_{S_1} &\leq \left\|H^{(n)}_{T,T}\circ  (U^{(n)*}_T X_0 V^{(n)}_{T})\right\|_{S_1} +\left\|H^{(n)}_{T,T_c}\circ (U^{(n)*}_T X_0 V^{(n)}_{T_c})\right\|_{S_1} \\
&+ \left\|H^{(n)}_{T_c,T}\circ(U^{(n)*}_{T_c} X_0 V^{(n)}_{T})\right\|_{S_1}+\left\|H^{(n)}_{T_c,T_c}\circ(U^{(n)*}_{T_c} X_0 V^{(n)}_{T_c})\right\|_{S_1}.
\end{split}
\end{equation}
The four Schatten-$1$ norms can then be estimated by $\max(r,(d-r))^{1/2}$ times the corresponding Schatten-$2$ norms. Using then again inequalities $\cref{eq_H_TT,eq_H_TTc,eq_H_TcT,eq_H_TcTc}$, we conclude the proof of (ii).
\end{proof}

\begin{proof}[Proof of \cref{rate}]
First we note that
\begin{equation} \label{eq_XnTcterm}
\left(\sum_{i=r+1}^{d} \big(\sigma_i ^2(X^{(n)}) + \epsilon^{(n)2}\big)^{\frac{p}{2}} \right)^{2-p} 
\leq  2^{p-\frac{p^2}{2}}(d-r)^{2-p} \sigma_{r+1}(X^{(n)})^{p(2-p)}
\end{equation}
as $\epsilon^{(n)}\leq \sigma_{r+1}(X^{(n+1)})$ due to the choice of $\epsilon^{(n)}$ in \cref{epsmin}. We proceed by induction over $n \geq \overline{n}$. \cref{lemma_localrate_1}(ii) and \cref{lemma_localrate_2}(ii) imply together with \cref{eq_XnTcterm} that for $n=\overline{n}$,
\begin{equation} \label{eq_etaSinfdecay}
\begin{split}
\|\eta^{(n+1)}\|_{S_\infty}^p \leq \frac{\|\eta^{(n+1)}\|_{S_2}^{2p}}{\|\eta^{(n+1)}\|_{S_\infty}^{p}} &\leq  2^p  \gamma_{2r}^{2-p} 2^{p-\frac{p^2}{2}}\Big(\frac{d-r}{r}\Big)^{2-p/2}\frac{7^p r^{p} (s_r^0)^{p(p-1)}}{(1-\zeta)^{2p}} \kappa^p \|\eta^{(n)}\|_{S_\infty}^{2p-p^2} \\
&\leq 2^{5p}  \gamma_{2r}^{2-p} \Big(\frac{d-r}{r}\Big)^{2-p/2}\frac{r^{p} (s_r^0)^{p(p-1)}}{(1-\zeta)^{2p}} \kappa^p \|\eta^{(n)}\|_{S_\infty}^{p(2-p)}
\end{split}
\end{equation}
as $\sigma_{r+1}(X^{(n)}) = \epsilon^{(n)}$ by assumption for $n = \overline{n}$.

Similarly, by \cref{lemma_localrate_1}(iii), \cref{lemma_localrate_2}(ii) and \cref{eq_XnTcterm}, the error in the Schatten-$p$ quasinorm fulfills
\begin{equation} \label{eq_etanp1Sp}
\begin{split}
\|\eta^{(n+1)}\|_{S_p}^{2p} &\leq (1+\gamma_{2r})^2 2^{2+2p} \big(d-r\big)^{2-p} \frac{r^{p/2}(s_r^0)^{p(p-1)}}{(1-\zeta)^{2p}}\kappa^p \|\eta^{(n)}\|_{S_\infty}^{p(2-p)}\|\eta^{(n+1)}\|_{S_2}^{p} \end{split}
\end{equation}
for $n = \overline{n}$. Using the strong Schatten-$p$ null space property of order $2r$ for the operator $\Phi$, we see from the arguments of the proof of \cref{lemma_localrate_1}(ii) that
\begin{equation*}
\|\eta^{(n)}\|_{S_\infty}^{p} \leq \|\eta^{(n)}\|_{S_2}^{p} \leq \frac{2^{p-1} \gamma_{2r}^{1-p/2}}{r^{1-p/2}}\|\eta^{(n)}\|_{S_p}^p
\end{equation*}
and also $\|\eta^{(n+1)}\|_{S_2}^{p} \leq \frac{2^{p-1} \gamma_{2r}^{1-p/2}}{r^{1-p/2}}\|\eta^{(n+1)}\|_{S_p}^p$. Inserting that in \cref{eq_etanp1Sp} and dividing by $\|\eta^{(n+1)}\|_{S_p}^p$, we obtain
\begin{equation*}
\begin{split}
\|\eta^{(n+1)}\|_{S_p}^{p} &\leq 2^{4p} (1+\gamma_{2r})^2 \gamma_{2r}^{2-p} \Big(\frac{d-r}{r}\Big)^{2-p} \frac{r^{p/2} (s_r^0)^{p(p-1)}}{(1-\zeta)^{2p}} \kappa^p \|\eta^{(n)}\|_{S_\infty}^{p(1-p)}\|\eta^{(n)}\|_{S_p}^p.
\end{split}
\end{equation*}
Under the assumption that \cref{eq_mu_condition2} holds, it follows from this and \cref{eq_etaSinfdecay} that
\begin{equation} \label{eq_etadecay}
\|\eta^{(n+1)}\|_{S_\infty}^{p} \leq \|\eta^{(n)}\|_{S_\infty}^p \text{ and } \|\eta^{(n+1)}\|_{S_p}^{p} \leq \|\eta^{(n)}\|_{S_p}^p
\end{equation}
for $n = \overline{n}$, which also entails the statement of \cref{rate} for this iteration.

Let now $n' > \overline{n}$ such that \cref{eq_etadecay} is true for all $n$ with $n' > n \geq \overline{n}$. If $\sigma_{r+1}(X^{(n')}) \leq \epsilon^{(n'-1)}$, then $\epsilon^{(n')} = \sigma_{r+1}(X^{(n')})$ and the arguments from above show \cref{eq_etadecay} also for $n = n'$.

Otherwise $\sigma_{r+1}(X^{(n')}) > \epsilon^{(n'-1)}$ and there exists $n' > n'' \geq \overline{n}$ such that $\epsilon^{(n')}=\epsilon^{(n'')}=\sigma_{r+1}(X^{(n'')})$. Then
\begin{equation*}
\|\eta^{(n'+1)}\|_{S_\infty}^p\!\!\! \leq \! 14^p  \frac{\gamma_{2r}^{2-p}}{r^{2-p}} \bigg[\sum_{i=r+1}^{d}\!\!\! \Big(\frac{\sigma_i ^2(X^{(n')})}{\epsilon^{(n'')2}} +1\Big)^{\frac{p}{2}} \bigg]^{\!2-p} \frac{ r^{p/2}\max(r,d-r)^{p/2}}{(s_r^0)^{p(1-p)}(1-\zeta)^{2p}} \kappa^p \|\eta^{(n')}\|_{S_\infty}^{p(2-p)}
\end{equation*}
and we compute
\begin{equation*}
\begin{split}
&\bigg[\sum_{i=r+1}^{d} \Big(\frac{\sigma_i ^2(X^{(n')})}{\epsilon^{(n'')2}} +1\Big)^{\frac{p}{2}} \bigg]^{2-p} \leq \bigg[\sum_{i=r+1}^{d} \frac{\sigma_i ^p(X^{(n')})}{\epsilon^{(n'')p}}+(d-r)\bigg]^{2-p} \\
&\leq \bigg[\frac{\|\eta^{(n')}\|_{S_p}^p}{\epsilon^{(n'')p}}+(d-r)\bigg]^{2-p} \leq \bigg[\frac{\|\eta^{(n'')}\|_{S_p}^p}{\epsilon^{(n'')p}}+(d-r)\bigg]^{2-p}  \\
&\leq \bigg[\frac{2(1+\gamma_{2r})\|X_{T_c}^{(n'')}\|_{S_p}^p}{(1-\gamma_{2r})\epsilon^{(n'')p}}+(d-r)\bigg]^{2-p} \leq \bigg(\frac{3+\gamma_{2r}}{1-\gamma_{2r}}\bigg)^{2-p}(d-r)^{2-p},
\end{split}
\end{equation*}
using that $X_0$ is a matrix of rank at most $r$ in the second inequality, the inductive hypothesis in the third and an analogue of \cref{eq_schatten_2_NSPestimate} for a Schatten-$p$ quasinorm on the left hand side (cf. \citep[Lemma 3.2]{Kabanava15} for the corresponding result for $p=1$) in the last inequality. The latter argument uses the assumption on the null space property. This shows that 
\begin{equation*}
\|\eta^{(n'+1)}\|_{S_\infty}^p \leq \mu \|\eta^{(n')}\|_{S_\infty}^{p(2-p)}
\end{equation*}
for 
\begin{equation*}
\widetilde{\mu} := 2^{4p}  \gamma_{2r}^{2-p} \Big(\frac{(3+\gamma_{2r})(d-r)}{(1-\gamma_{2r})r}\Big)^{2-p}\frac{r^{p/2} (s_r^0)^{p(p-1)}}{(1-\zeta)^{2p}} \kappa^p \max\left(2^p (d-r)^{\frac{p}{2}},(1+\gamma_{2r})^2 \right),
\end{equation*}
and $\|\eta^{(n'+1)}\|_{S_\infty}^p \leq \|\eta^{(n')}\|_{S_\infty}^{p}$ under the assumption \cref{eq_mu_condition2} of \cref{rate}, as $\widetilde{\mu} \leq \mu$ with $\mu$ as in \cref{eq_mu_definition}. Indeed $\widetilde{\mu} \leq \mu$ since 
\begin{equation*}
\max\left(2^p (d-r)^{\frac{p}{2}},(1+\gamma_{2r})^2 \right) \Big(\frac{d-r}{r}\Big)^{2-p} r^{p/2} \leq 2^p (1+\gamma_{2r})^2 \Big(\frac{d-r}{r}\Big)^{2-p/2} r^p.
\end{equation*}
The same argument shows that $\|\eta^{(n'+1)}\|_{S_p}^p \leq \|\eta^{(n')}\|_{S_p}^{p}$, which finishes the proof.
\end{proof}
%
%
%

 \begin{remark} \label{remark_Mohanlaueftnicht}
We note that the weight matrices of the previous IRLS approaches \texttt{IRLS-col} and \texttt{IRLS-row} \citep{Fornasier11,Mohan10} at iteration $n$ could be expressed in our notation as 
\begin{equation*}
  \mathbf{I}_{d_2} \otimes W_L^{(n)} := \mathbf{I}_{d_2}  \otimes U^{(n)}(\overline{\Sigma}_{d_1}^{(n)})^{p-2}U^{(n)*}
\end{equation*}
 and 
\begin{equation*}
W_R^{(n)} \otimes  \mathbf{I}_{d_1} := V^{(n)}(\overline{\Sigma}_{d_2}^{(n)})^{p-2}V^{(n)*} \otimes \mathbf{I}_{d_1},
\end{equation*}
 respectively, cf. \cref{sec_weightedFrob}, if 
$X^{(n)} = U^{(n)} \Sigma^{(n)} V^{(n)*} = U_T^{(n)}  \Sigma_T^{(n)} V_T^{(n)*}  + U_{T_c}^{(n)} \Sigma_{T_c}^{(n)}V_{T_c}^{(n)*} $
is the SVD of the iterate $X^{(n)}$ with $U_T^{(n)}$ and $V_T^{(n)}$ containing the $r$ first left- and right singular vectors.

Now let 
\begin{equation*}
T^{(n)}:=\big\{ U_T^{(n)} Z_1^{*} +  Z_2 V_T^{(n)*} : Z_1 \in M_{d_1 \times r}, Z_2 \in M_{d_2 \times r} \big\}
\end{equation*}
be the tangent space of the smooth manifold of rank-$r$ matrices at the best rank-$r$ approximation $U_T^{(n)}  \Sigma_T^{(n)} V_T^{(n)*} $ of $X^{(n)}$, or, put differently, the direct sum of the row and column spaces of $U_T^{(n)}  \Sigma_T^{(n)} V_T^{(n)*}$.
 
The fact that left- or right-sided weight matrices do not lead to algorithms with superlinear convergence rates for $p < 1$ can be explained by noting that there are always parts of the space $T^{(n)}$ that are equipped with too large weights if $X^{(n)} = U^{(n)}\Sigma^{(n)}V^{(n)*}$ is already approximately low-rank. In particular, proceeding as in \cref{eq_WtildeX0}, we obtain for $\mathbf{I}_{d_2} \otimes W_L^{(n)}$
\begin{equation*}
\begin{split}
\left\|\big[\mathbf{I}_{d_2} \otimes W_L^{(n)}(X_{0})_{\vecc}\big]_{\mat}\right\|_{S_2}^2 \!\! = &\Big\|\big(\overline{\Sigma}_T^{(n)}\big)^{p-2} U^{(n)*}_T\! X_0 V^{(n)}_{T}\Big\|_{S_2}^2\!\! +\left\|\big(\overline{\Sigma}_T^{(n)}\big)^{p-2}U^{(n)*}_T\! X_0 V^{(n)}_{T_c}\right\|_{S_2}^2 \\
+ \bigg\|\big(\overline{\Sigma}_{T_c}&^{(n)}\big)^{p-2}U^{(n)*}_{T_c} X_0 V^{(n)}_{T}\bigg\|_{S_2}^2\!\!+\left\|\big(\overline{\Sigma}_{T_c}^{(n)}\big)^{p-2}U^{(n)*}_{T_c} X_0 V^{(n)}_{T_c}\right\|_{S_2}^2
\end{split}
\end{equation*}
if $\overline{\Sigma}_T^{(n)}$ denotes the diagonal matrix with the first $r$ non-zero entries of $\overline{\Sigma}_{d_1}^{(n)}$ and $\overline{\Sigma}_{T_c}^{(n)}$ the one of the remaining entries. 

Here, the third of the four summands would become too large for $p<1$ to allow for a superlinear convergence when the last $d-r$ singular values of $X^{(n)}$ approach zero. An analogous argument can be used for the right-sided weight matrix $W_R^{(n)} \otimes  \mathbf{I}_{d_1}$ and, notably, also for arithmetic mean weight matrices $W_{\text{(arith)}}^{(n)} = \mathbf{I}_{d_2}\otimes {W}_L^{(n)}+{W}_R^{(n)}\otimes \mathbf{I}_{d_1}$, cf. \cref{sec_avg_weightmatrices}. 
\end{remark}

\section{Acknowledgments}
The two authors acknowledge the support and hospitality of the Hausdorff Research Institute for Mathematics (HIM) during the early stage of this work within the HIM Trimester Program "Mathematics of Signal Processing''.
C.K. is supported by the German Research Foundation (DFG) in the context of the Emmy Noether Junior Research Group ``Randomized Sensing and Quantization of Signals and Images'' (KR 4512/1-1) and the ERC Starting Grant ``High-Dimensional Sparse Optimal Control'' (HDSPCONTR - 306274). J.S. is supported by the DFG through the D-A-CH project no. I1669-N26 and through the international research training group IGDK 1754 ``Optimization and Numerical Analysis for Partial Differential Equations with Nonsmooth Structures''.
 The authors thank Ke Wei for providing code of his implementations. They also thank Massimo Fornasier for helpful discussions. 



\begin{appendix}

\section{Kronecker and Hadamard products} \label{sec_kroneckerhadamard}
For two matrices $A=(a_{ij})_{i\in[d_1],j\in[d_3]} \in \C^{d_1\times d_3}$ and $B \in \C^{d_2\times d_4}$, we call the matrix representation of their tensor product with respect to the standard bases the \emph{Kronecker product} $A \otimes B \in \C^{d_1\cdot d_2 \times d_3 \cdot d_4}$. By its definition, $A \otimes B$ is a block matrix of $d_2 \times d_4$ blocks whose block of index $(i,j)\in [d_1]\times[d_3]$ is the matrix $a_{ij} B\in \R^{d_2 \times d_4}$. This implies, e.g., for $A \in \C^{d_1 \times d_3}$ with $d_1=2$ and $d_3 =3$ that
\begin{equation*} A \otimes B =
\left[\begin{array}{ccc}
a_{11} & a_{12} & a_{13}\\
a_{21} &a_{22} & a_{23}\\
\end{array}
\right] \otimes B =  \left[\begin{array}{c|c|c}a_{11}B & a_{12}B & a_{13}B \\ \hline a_{21}B & a_{22}B & a_{23}B \end{array}\right].
\end{equation*}

The Kronecker product is useful for the elegant formulation of matrix equations involving left and right matrix multiplications with the variable $X$, as
\begin{equation*} 
AXB^*=Y\text{\hspace*{5mm} if and only if \hspace*{5mm}}(B\otimes A) X_{\vecc}=Y_{\vecc}.
\end{equation*}

We define the \emph{Hadamard product} $A\circ B \in \C^{d_1\times d_2}$ of two matrices $A \in \C^{d_1\times d_2}$ and $B \in \C^{d_1\times d_2}$ as their entry-wise product
\begin{equation*}
(A\circ B)_{i,j}=A_{i,j}B_{i,j}
\end{equation*}
with $i \in [d_1]$ and $j \in [d_2]$. The Hadamard product is also known as \emph{Schur product} in the literature. 

Furthermore, if $d_1=d_3$ and $d_2=d_4$, we define the \emph{Kronecker sum} $A\oplus B \in \C^{d_1 d_2 \times d_1 d_2}$ of two matrices $A \in \C^{d_1\times d_1}$ and $B \in \C^{d_2\times d_2}$ as the matrix 
\begin{equation} \label{def_kronsum}
A\oplus B=(\mathbf{I}_{d_2} \otimes A) +( B \otimes \mathbf{I}_{d_1}).
\end{equation}

Note that equations of the form $AX+XB^*=Y$ can be rewritten as 
\[ (A\oplus B) X_{\vecc}=Y_{\vecc}, \]
using again the vectorizations of $X$ and $Y$.
An explicit formula that expresses the inverse $(A \oplus B)^{-1}$ of the Kronecker sum $A \oplus B$ is provided by the following lemma.
\begin{lemma}[\citep{Jameson68}] \label{lemma_jameson}
Let $A \in H_{d_1 \times d_1}$ and $B \in H_{d_2 \times d_2}$, where one of the matrices is positive definite and the other positive semidefinite. If we denote the singular vectors of $A$ by $u_i \in \C^{d_1}$, $i \in [d_1]$, its singular values by $\sigma_i$, $i \in [d_1]$ and the singular vectors resp. values of $B$ by $v_j \in \C^{d_2}$ resp. $\mu_j$, $j \in [d_2]$, then
\begin{equation}\label{invformula}
(A\oplus B)^{-1}=\sum\limits^{d_1}_{i=1} \sum\limits^{d_2}_{j=1} \frac{v_jv_j^* \otimes u_iu_i^*}{\sigma_i+\mu_j}.
\end{equation}
Furthermore, the action of $(A\oplus B)^{-1}$ on the matrix space $M_{d_1 \times d_2}$ can be written as
\begin{equation}\label{hada}
\left[(A\oplus B)^{-1}Z_{\vecc}\right]_{\mat}= U\big(H\circ (U^*ZV)\big)V^*. 
\end{equation}
for $Z \in M_{d_1 \times d_2}$, $U = [u_1,\ldots,u_{d_1}]$, and $V = [v_1,\ldots,v_{d_2}]$ and the matrix $H\in M_{d_1 \times d_2}$ with the entries $H_{i,j}=(\sigma_i+\mu_j)^{-1}$, $i \in [d_1]$, $ j \in [d_2]$.
\end{lemma}

\section{Proofs of preliminary statements in \cref{sec_theoretical_analysis}}
\subsection{Proof of \cref{Wsubs}: Main part}

First, we define the function 
\begin{equation*}
\resizebox{ \textwidth}{!} 
{$\begin{aligned}
f^p_{X,\epsilon}(Z)&=\mathcal{J}_p(X,\epsilon,Z)=\begin{cases}\frac{p}{2} \|X_{\vecc}\|^2_{\ell_2(\widetilde{W}(Z))}+
\frac{\epsilon^2 p}{2}\sum\limits^d_{i=1}\sigma_i(Z)+\frac{2-p}{2}\sum\limits^d_{i=1}\sigma_i(Z)^{\frac{p}{(p-2)}} & \text{ if }\rank(Z)=d, \\
 +\infty &  \text{ if }\rank(Z) < d,
\end{cases}\\
\end{aligned}$}
\end{equation*}
 for $X\in M_{d_1\times d_2}$, $\epsilon>0$ fixed and with $Z \in M_{d_1\times d_2}$ as its only argument. 
 We note that the set of minimizers of $f^p_{X,\epsilon}(Z)$ does not contain an instance $Z$ with rank smaller than $d$ as the value of $f^p_{X,\epsilon}(Z)$ is infinite at such points and, therefore, it is sufficient to search for minimizers on the set $\Omega=\left\{Z \in M_{d_1\times d_2} | \rank(Z)=d\right\}$ of matrices with rank $d$. 
We observe that the set $\Omega$ is an open set and that we have that

\begin{enumerate}
\item[(a)] $f^p_{X,\epsilon}(Z)$ is lower semicontinuous, which means that any sequence $(Z^k)_{k \in \mathbb{N}}$ with $Z^k \stackrel{k\rightarrow \infty}{\longrightarrow} {Z}$ fulfills $\liminf\limits_{k\rightarrow \infty} f^p_{X,\epsilon}(Z^k)\geq f^p_{X,{\epsilon}}({Z})$,
 \item[(b)] $f^p_{X,\epsilon}(Z)\geq \alpha$ for all $Z \in M_{d_1\times d_2}$ for some constant $\alpha$,
\item[(c)] $f^p_{X,\epsilon}(Z)$ is coercive, i.e., for any sequence $(Z^k)_{k \in \mathbb{N}}$ with $\|Z^k\|_F\stackrel{k\rightarrow \infty}{\longrightarrow} \infty$, we have $f^p_{X,\epsilon}(Z^k)\stackrel{k\rightarrow \infty}{\longrightarrow}\infty$.
\end{enumerate}

Property (a) is true as $f^p_{X,\epsilon}(Z))|_{\Omega}$ is a concatenation of an indicator function of an open set, which is lower-semicontinuous and a sum of continuous functions on $\Omega$. Property (b) is obviously true for the choice $\alpha=0$. 

To justify point (c), we note that $f^p_{X,\epsilon}(Z)>\frac{\epsilon^2 p}{2}\sum\limits^d_{i=1}\sigma_i(Z)=\frac{\epsilon^2 p}{2}\|Z\|_{S_1}\geq \frac{\epsilon^2 p}{2}\|Z\|_{F}$ and therefore, coercivity is clear from its definition. As a consequence from (a) and (c), it is also true that the level sets $L_C=\left\{Z \in M_{d_1 \times d_2}| f^p_{X,\epsilon}(Z)\leq C\right\}$ are closed and bounded and therefore, compact.

Via the direct method of calculus of variations, we conclude from the properties (a) - (c) that $f^p_{X,\epsilon}(Z)$ has at least one global minimizer belonging to the set of critical points of $f^p_{X,\epsilon}(Z)$ \citep[Theorem 1]{Dacorogna89}.\\
To characterize the set of critical points of $f^p_{X,\epsilon}(Z)$, its derivative with respect to $Z$ is calculated explicitly and equated with zero in \cref{sec_proof_critical}. The solution of the resulting equation reveals that  
${Z}_{\opt}=\sum_{i=1}^d (\sigma_i^2(X)+\epsilon^2)^{\frac{p-2}{2}} u_iv_i^* =: \sum_{i=1}^d \widetilde{\sigma}_i u_iv_i^*$
is the only critical point and consequently the unique global minimizer of $f^p_{X,\epsilon}(Z)$.
We define 
 the matrices $W^L_{\opt}:=\sum^d_{i=1}\widetilde{\sigma}_i u_i u_i^*$ and $W^R_{\opt}:=\sum^{d}_{i=1} \widetilde{\sigma}_i v_i v_i^*$, and note that
$
\widetilde{W}(Z_{\opt})= 2\big((W^{R}_{\opt})^{-1} \oplus (W^{L}_{\opt})^{-1}\big)^{-1}
$
with \cref{J}. To verify the second part of the theorem, we simply plug the optimal solution $Z_{\opt}$ into the functional $\mathcal{J}_p$ and compute using \cref{invformula} that
\begin{equation*}
\resizebox{ \textwidth}{!} 
{$\begin{aligned} 
\mathcal{J}_p(X,\epsilon,Z_{\opt}) &= 
\frac{p}{2} \|X_{\vecc}\|^2_{\ell_2(\widetilde{W}(Z_{\opt}))}+
 \frac{\epsilon^2 p}{2}\sum\limits^d_{i=1}\widetilde{\sigma}_i+\frac{2-p}{2}\sum\limits^d_{i=1}\widetilde{\sigma}_i^{\frac{p}{p-2}}\\
 &=
\frac{p}{2} \sum^d_{i=1}\left[{\sigma}_i^2(X) (u_i^*\otimes v_i^*)2 \left(\sum\limits^{d_2}_{k=1} \sum\limits^{d_1}_{j=1} \frac{u_ku_k^*\otimes v_jv_j^*}{\widetilde{\sigma}^{-1}_k+\widetilde{\sigma}^{-1}_j}\right)(u_i\otimes v_i)\right]_{ii}+
 \frac{\epsilon^2 p}{2}\sum\limits^d_{i=1}\widetilde{\sigma}_i+\frac{2-p}{2}\sum\limits^d_{i=1}\widetilde{\sigma}_i^{\frac{p}{p-2}}\\
&=\frac{p}{2}\sum\limits^d_{i=1} ({\sigma}_i^2(X)+\epsilon^2) \widetilde{\sigma}_i+\frac{2-p}{2}\sum\limits^d_{i=1}\widetilde{\sigma}_i^{\frac{p}{p-2}} \\
&=\frac{p}{2}\sum\limits^d_{i=1} ({\sigma}_i^2(X)+\epsilon^2) ({\sigma}_i^2(X)+\epsilon^2)^{\frac{p-2}{2}}+\frac{2-p}{2}\sum\limits^d_{i=1}({\sigma}_i^2(X)+\epsilon^2)^{\frac{p}{2}} \\
&=\sum\limits^d_{i=1} ({\sigma}_i^2(X)+\epsilon^2)^{\frac{p}{2}}.
 \end{aligned}$}
\end{equation*}

\subsection{Proof of \cref{Wsubs}: Critical points of \texorpdfstring{$f^p_{X,\epsilon}$}{fpXepsilon}} \label{sec_proof_critical}
Let us without loss of generality consider the case $d=d_1=d_2$ and define 
$$
\Omega=\left\{Z\in M_{d\times d} \text{ s.t. } \rank(Z)=d  \right\}.
$$ 
As already mentioned in \cref{eq_Wtilde_W1_W2}, the harmonic mean matrix $\widetilde{W}(Z)$ can be then rewritten as 
\begin{equation*}
\widetilde{W}(Z) = 2 \widetilde{W}_1\big(\widetilde{W}_1 + \widetilde{W}_2\big)^{-1} \widetilde{W}_2 =2(\widetilde{W}_1^{-1}+\widetilde{W}_2^{-1})^{-1}
\end{equation*} 
for $Z \in \Omega$ with the definitions $\widetilde{W}_1:=\mathbf{I}_{d} \otimes (ZZ^*)^{\frac{1}{2}}$ and $\widetilde{W}_2=(Z^*Z)^{\frac{1}{2}} \otimes \mathbf{I}_{d}$. For $Z \in \Omega$, we reformulate the auxiliary functional such that
\begin{equation*}
\resizebox{0.9 \textwidth}{!}{$\begin{aligned}
f^p_{X,\epsilon}(Z)&=\mathcal{J}^p(X,\epsilon,Z)=\frac{p}{2} \|X_{\vecc}\|^2_{\ell_2(\widetilde{W}(Z))}+
\frac{\epsilon^2 p}{2}\sum\limits^d_{i=1}\sigma_i(Z)+\frac{2-p}{2}\sum\limits^d_{i=1}\sigma_i(Z)^{\frac{p}{(p-2)}}\\
&= \frac{p}{2} \|X_{\vecc}\|^2_{\ell_2(\widetilde{W}(Z))}+
\frac{\epsilon^2 p}{2}\| (Z^*Z)^{1/2}\|^2_{F}+\frac{2-p}{2}\|(Z^*Z)^{\frac{p}{2(p-2)}}\|^2_{F} .
\end{aligned}$}
\end{equation*} To identify the set of critical points of $f^p_{X,\epsilon}(Z)$ located in $\Omega$, we compute its derivative with respect to $Z$ using the derivative rules (7), (12), (13), (15), (16), (18), (20) in Chapter 8.2 and Theorem 3 in Chapter 8.4 of \citep{Magnus99} in the following. Using the notation of \citep{Magnus99}, we calculate
\begin{equation*}
\resizebox{ \textwidth}{!} 
{$\begin{aligned}
{\partial f^p_{X,\epsilon}(Z)}=-\frac{p}{2}\trace\left(X_{\vecc}^*\widetilde{W}{\partial\widetilde{W}^{-1}}\widetilde{W}X_{\vecc}\right)
+\frac{p\epsilon^2}{4}\left(\trace\left(Z(Z^*Z)^{-\frac{1}{2}}\partial Z^* \right)+\trace((Z^*Z)^{-\frac{1}{2}}Z^* \partial Z)\right) \\
-\frac{p}{4}\left(\trace\left(Z(Z^*Z)^{\frac{4-p}{2(p-2)}}\partial Z^* \right)+\trace((Z^*Z)^{\frac{4-p}{2(p-2)}}Z^* \partial Z)\right)  
\end{aligned}$}
\end{equation*}
 where
\begin{equation}\label{winddev}
\resizebox{ \textwidth}{!} 
{$\begin{aligned}
{\partial\widetilde{W}^{-1}}=\frac{1}{2}{\partial \left[(ZZ^*)^{-\frac{1}{2}}\oplus (Z^*Z)^{-\frac{1}{2}}\right]}=&
-\frac{1}{4}\left[\left((Z^*Z)^{-\frac{3}{2}}Z^* \partial Z+\partial Z^*Z(Z^*Z)^{-\frac{3}{2}} \right)\otimes \mathbf{I}_{d_1}\right]\\&-\frac{1}{4}\left[\mathbf{I}_{d_2}\otimes\left( \partial Z(ZZ^*)^{-\frac{3}{2}}Z^*+(ZZ^*)^{-\frac{3}{2}}Z\partial Z^* \right)\right].
\end{aligned}$}
\end{equation} We can reformulate the first term as follows using the cyclicity of the trace,
\begin{equation*}
\resizebox{ 0.9\textwidth}{!} {$\begin{aligned}
-\frac{p}{2}\trace\left(X_{\vecc}^*\widetilde{W}{\partial\widetilde{W}^{-1}}\widetilde{W}X_{\vecc}\right)&= \frac{p}{8}\left[ \trace\left((\widetilde{W}X_{\vecc})_{\mat}^*(\widetilde{W}X_{\vecc})_{\mat}(Z^*Z)^{-\frac{3}{2}}Z^*\partial Z \right)\right.\\
&+\left.\trace\left(Z(Z^*Z)^{-\frac{3}{2}}(\widetilde{W}X_{\vecc})_{\mat}^*(\widetilde{W}X_{\vecc})_{\mat}\partial Z^* \right)\right.\\
&\left.+\trace\left(Z^*(ZZ^*)^{-\frac{3}{2}}(\widetilde{W}X_{\vecc})_{\mat}(\widetilde{W}X_{\vecc})_{\mat}^*\partial Z \right)\right.\\
&\left.+\trace\left((\widetilde{W}X_{\vecc})_{\mat}(\widetilde{W}X_{\vecc})_{\mat}^*(ZZ^*)^{-\frac{3}{2}}Z\partial Z^* \right)\right].
\end{aligned}$}
\end{equation*}

To determine the critical points of $f^p_{X,\epsilon}(Z)$, we summarize the calculations above, rearrange the terms and equate the derivative with zero, such that
 \begin{equation*}
\resizebox{ \textwidth}{!} 
{$\begin{aligned}
{\partial f^p_{X,\epsilon}(Z)}=
\frac{p}{8}&\trace\left(\left[ (\widetilde{W}X_{\vecc})_{\mat}^*(\widetilde{W}X_{\vecc})_{\mat}(Z^*Z)^{-\frac{3}{2}}Z^*+Z^*(ZZ^*)^{-\frac{3}{2}}(\widetilde{W}X_{\vecc})_{\mat}(\widetilde{W}X_{\vecc})_{\mat}^*\right.\right.\\
&\left.\left.+2\epsilon^2 (Z^*Z)^{-\frac{1}{2}}Z^*- 2(Z^*Z)^{\frac{4-p}{2(p-2)}}Z^*\right]\partial Z \right)\\
+\frac{p}{8}&\trace\left(\left[ Z(Z^*Z)^{-\frac{3}{2}}(\widetilde{W}X_{\vecc})_{\mat}^*(\widetilde{W}X_{\vecc})_{\mat}+(\widetilde{W}X_{\vecc})_{\mat}(\widetilde{W}X_{\vecc})_{\mat}^*(ZZ^*)^{-\frac{3}{2}}Z\right.\right.\\
&\left.\left.+2\epsilon^2 Z(Z^*Z)^{-\frac{1}{2}}- 2Z(Z^*Z)^{\frac{4-p}{2(p-2)}}\right]\partial Z^* \right)\\
:= \frac{p}{8}&\trace\left(A\partial Z\right)+\frac{p}{8}\trace\left(A^*\partial Z^*\right)
= \frac{p}{8}\trace\left((A\oplus A)\partial Z\right)=0,
\end{aligned}$}
\end{equation*}
where 
\begin{equation} \label{eq_def_A}
\resizebox{ 0.9\textwidth}{!} 
{$\begin{split}
A&=\left[ (\widetilde{W}X_{\vecc})_{\mat}^*(\widetilde{W}X_{\vecc})_{\mat}(Z^*Z)^{-\frac{3}{2}}Z^*+Z^*(ZZ^*)^{-\frac{3}{2}}(\widetilde{W}X_{\vecc})_{\mat}(\widetilde{W}X_{\vecc})_{\mat}^*\right.\\
&\left.+2\epsilon^2 (Z^*Z)^{-\frac{1}{2}}Z^*- 2(Z^*Z)^{\frac{4-p}{2(p-2)}}Z^*\right].
\end{split}$}
\end{equation}
and hence an easy calculation as in \cite{Duchi07} gives
\begin{equation*}
\frac{\partial f^p_{X,\epsilon}(Z)}{\partial Z}=\frac{\frac{p}{8}\trace\left((A\oplus A)\partial Z\right)}{\partial Z}=\frac{p}{8}(A\oplus A)=0.
\end{equation*}

Now we have to find $Z$ such that $A\oplus A=0$. 
This implies that all eigenvalues of $A\oplus A=A\otimes \mathbf{I}_{d}+ \mathbf{I}_{d}\otimes A$ are equal to zero.
The eigenvalues of the Kronecker sum of two matrices $A_1$ and $A_2$ with eigenvalues $\lambda_s$ and $\mu_t$ with $s,t\in [d]$ are the sum of the eigenvalues $\lambda_s+\mu_t$. As in our case $A=A_1=A_2$ this means that all eigenvalues of $A$ itself have to be zero. This is only possible if $A$ is the zero matrix. 

Let $Z=U\Sigma V^*\in M_{d\times d}$ with ${U}, V\in \mathcal{U}_{d}$ and ${\Sigma}\in M_{d \times d}$, where $\Sigma = \diag(\sigma)$  is a diagonal matrix with \emph{ascending} entries.
We define the matrix $H=H_{i,j}=\frac{2}{\sigma^{-1}_i+\sigma^{-1}_j}$ for $i=1,\dots,d,j=1,\dots,d$ corresponding to the result of reshaping the diagonal of $2(\Sigma\oplus\Sigma)$ into a $d\times d$-matrix. Using \cref{hada}, we can express $(\widetilde{W}X_{\vecc})_{\mat}=U\big(H\circ(U^*XV)\big)V^*$ and denote $B:=H\circ(U^*XV)$. 

Plugging the decomposition $Z=U\Sigma V^*$ into \cref{eq_def_A}, we can therefore calculate 
\begin{equation}
\resizebox{ \textwidth}{!} 
{$\begin{aligned}
\text A=0 \Leftrightarrow&\;(UBV^*)^*(UBV^*)(V\Sigma^2V^*)^{-3/2}(U\Sigma V^*)^*+(U\Sigma V^*)^*(U\Sigma^2U^*)^*)^{-3/2}(UBV^*)(UBV^*)^*\\
&+2\epsilon^2(V\Sigma^2V^*)^{-1/2} (U\Sigma V^*)^*-2 (V\Sigma^2V^*)^{\frac{4-p}{2(p-2)}}(U\Sigma V^*)^*=0\\
\Leftrightarrow&\;VB^*B\Sigma^{-2} U^*+V\Sigma^{-2} BB^* U^*+2\epsilon^2 V\mathbf{I}_dU^*-2 V\Sigma^{\frac{2}{p-2}}U^*=0\\
\Leftrightarrow&\;B^*B\Sigma^{-2} +\Sigma^{-2} BB^*+2\epsilon^2 \mathbf{I}_d-2 \Sigma^{\frac{2}{p-2}}=0.\\
\end{aligned}$}\label{noorth}
\end{equation}

We now note that $2\epsilon^2 \mathbf{I}_d-2 \Sigma^{\frac{2}{p-2}}$ is diagonal and therefore, $B^*B\Sigma^{-2}+\Sigma^{-2} BB^*$ is diagonal as well. 
Moreover, observe that $B^*B+\Sigma^{-2} BB^*\Sigma^{2}$ is again a diagonal matrix and has a symmetric first summand $B^*B$. As the sum or difference of symmetric matrices is again symmetric also the second summand $\Sigma^{-2} BB^*\Sigma^{2}$ has to be symmetric, i.e., $\Sigma^{-2} BB^*\Sigma^{2}=(\Sigma^{-2} BB^*\Sigma^{2})^*=
\Sigma^{2} BB^*\Sigma^{-2}$. We conclude that it has to hold that $ BB^*\Sigma^4= \Sigma^4BB^*$ and hence $\Sigma^4$ and $BB^*$ commute. 

This is only possible if either $\Sigma$ is a multiple of the identity or if  $BB^*$ is diagonal. Assuming the first case, \cref{noorth} would imply that also $BB^*$ and $B^*B$ have to be a multiple of the identity. Therefore, this first case, where $\Sigma$ is a multiple of the identity is a special case of the second possible scenario, where $BB^*$ is diagonal. Hence, it suffices to further consider the more general second case. (Considerations for $B^*B$ can be carried out analogously.)

Diagonality of $BB^*$ only occurs if $B$ is either orthonormal or diagonal. Assuming orthonormality would lead to contradictions with the equations in \cref{noorth}. Hence $B=H\circ(U^* X V)$ can only be diagonal.

Let now be $X = \bar{U}\bar{S}\bar{V}^*$ the singular value decomposition of $X$. As $H$ has no zero entries due to the full rank of $W$, this implies the diagonality of $U^*\bar{U}\bar{S}\bar{V}^*V$.
Consequently, $ U$ and $V$ can only be chosen such that $P=[U^*\bar{U}]_{d\times d}$ and $P^*=[\bar{V}^*V]_{d \times d}$ for a permutation matrix $P \in \mathcal{U}_d$. The reshuffled indexing corresponding to $P$ is denoted by $p(i) \in [d]$ for $i\in [d]$. 
Having in mind that $H_{ii}=\sigma_i$ for $i \in [d]$, we obtain
\begin{equation*}
\resizebox{ \textwidth}{!} 
{$\begin{split}
&(H\circ(P\bar{S}P^*))^*(H\circ(P\bar{S}P^*))\Sigma^{-2} +\Sigma^{-2}(H\circ(P\bar{S}P^*))(H\circ(P\bar{S}P^*))^*+2\epsilon^2 \mathbf{I}_d-2 \Sigma^{\frac{2}{p-2}}=0\\
\Leftrightarrow& \quad 2 \bar{s}_{p(i)}^2+2\epsilon^2=2\sigma_i^{\frac{2}{p-2}} \text{ for all } i \in [d] \\
\Leftrightarrow& \quad \sigma_i= (\bar{s}_{p(i)}^2+\epsilon^2)^{\frac{p-2}{2}} \text{ for all } i \in [d].
\end{split}$}
\end{equation*}
As the diagonal of $\Sigma$ was assumed to have ascending entries and the diagonal of $\bar{S}$ has descending entries, the permutation matrix $P$ has to be equal to the identity matrix. From $P=\mathbf{I}_{d}$, it follows that $U=\bar{U}$ and $V=\bar{V}$ and hence $\Sigma=(\bar{S}^2+\epsilon^2\mathbf{I}_{d})^{\frac{p-2}{2}}$.

We summarize our calculations by stating that
$$Z_{\opt}=\bar{U} \Sigma \bar{V}^*=\bar{U} (\bar{S}^2+\epsilon^2\mathbf{I}_{d})^{\frac{p-2}{2}}\bar{V}^*$$
is the only critical point of $f^p_{X,\epsilon}$ on the domain $\Omega$.

The results extend for the case $d_1\neq d_2$, where the definition of $\widetilde{W}(Z)$ is adapted by introducing the Moore-Penrose pseudo inverse of $(ZZ^*)^{1/2}$
\begin{equation*}
\widetilde{W}(Z) = 2 \widetilde{W}_1\big(\widetilde{W}_1 + \widetilde{W}_2\big)^{-1} \widetilde{W}_2 =2(\widetilde{W}_1^{+}+\widetilde{W}_2^{-1})^{-1}.
\end{equation*}
The corresponding derivative rule as pointed out in \citep[Chapter 8.4, Theorem 5]{Magnus99} can be used for the calculation in \cref{winddev}.

\subsection{Proof of \cref{JminX_prop}} \label{sec_proof_JminX_prop}
The equality of the optimization problems \eqref{eq_opt_WLS} can easily be seen by the fact that only the first summand of $\mathcal{J}_p(X,\epsilon,Z)$ depends on $X$. Now, it is important to show first that
$\widetilde{W}(Z)= 2([(Z^*Z)^{\frac{1}{2}}]^+ \oplus [(ZZ^*)^{\frac{1}{2}}]^+)^{-1}$ is positive definite as minimizing $\mathcal{J}_p(X,\epsilon,Z)$ then reduces to minimizing a quadratic form. Let $Z=\sum^d_{i=1}\sigma_iu_iv_i^*$, where $u_i, v_i$ for $i \in [d]$ are the left and right singular vectors, respectively, and $\sigma_i$ for $i \in [d]$ are the singular values of $Z$. 
Since $Z^*Z=\sum^d_{i=1}\sigma_i^2v_iv_i^*\succeq 0$, also the generalized inverse root fulfills $[(ZZ^*)^{\frac{1}{2}}]^+ \succeq 0$ and for $ZZ^*=\sum^d_{i=1}\sigma_i^2u_iu_i^*\succeq 0$, it follows that $[(ZZ^*)^{\frac{1}{2}}]^{+} \succeq 0$. We stress that at least one of the matrices $(ZZ^*)^{\frac{1}{2}}$ and $(Z^*Z)^{\frac{1}{2}}$ is positive definite 
and hence also $\widetilde{W}(Z) \succ 0$. 
With the fact that $\widetilde{W}(Z) \succ 0$, the statement can be proven analogously to the results in \citep[Lemma 5.1]{Fornasier11}.
\subsection{Proof of \cref{prelim}}\label{sec_proof_prelim}
(a) With the minimization property that defines $X^{(n+1)}$ in \cref{minX}, the inequality $\epsilon^{(n+1)}\leq \epsilon^{(n)}$,  and the minimization property that defines ${Z}^{(n+1)}$ in \cref{minW} and \cref{Wsubs} the monotonicity follows from 
$$ \begin{aligned}
\mathcal{J}_p(X^{(n)},\epsilon^{(n)},{Z}^{(n)})&\geq \mathcal{J}_p(X^{(n+1)},\epsilon^{(n)},{Z}^{(n)}) \geq \mathcal{J}_p(X^{(n+1)},\epsilon^{(n+1)},{Z}^{(n)}) \\
&\geq \mathcal{J}_p(X^{(n+1)},\epsilon^{(n+1)},{Z}^{(n+1)})
\end{aligned}
$$
(b) Using \cref{Wsubs} and the monotonicity property of (a) for all $n \in \mathbb{N}$, we see that
$$
\|X^{(n)}\|_{S_{p}}^p \leq g_{\epsilon^{(n)}}^p(X^{(n)})= \mathcal{J}_p(X^{(n)},\epsilon^{(n)},{Z}^{(n)}) \leq \mathcal{J}_p(X^{(1)},\epsilon^{(0)},{Z}^{(0)}),
$$
(c) The proof follows analogously to \citep[Proposition 6.1]{Fornasier11} 
where only the technical calculation to bound $\sigma_1^p((\widetilde{W}^{(n)})^{-1})$ requires to take into account that the spectrum of a Kronecker sum $A \oplus B$ consists of the pairwise sum of the spectra of $A$ and $B$ \citep[Proposition 7.2.3]{Bernstein09}.

\subsection{Proof of \cref{charf}}\label{sec_proof_charf}
The first statement $\widetilde{W}(X^{(n)},\epsilon^{(n)}) = \widetilde{W}^{(n)}$ is clear from the definition of $\widetilde{W}(X,\epsilon)$ and \cref{Wtilde_firstdef}. To show the necessity of \cref{eq_charf}, let $X \in M_{d_1 \times d_2}$ be a critical point of \cref{gemin}. Without loss of generality, let us assume that $d_1 \leq d_2$. In this case, a short calculation shows that $g^p_\epsilon(X)=\trace\big[ (X X^* + \epsilon^{2} \mathbf{I}_{d_1})^{p/2}\big]$. 
It follows from the matrix derivative rules of \citep[Chapter 8.2, (7),(15),(18) and (20)]{Magnus99} that
\[
\nabla g_{\epsilon}^p (X) = p(XX^*+\epsilon^2\mathbf{I}_{d_1})^{\frac{p-2}{2}}X  = p \sum_{i=1}^d (\sigma_i^2 + \epsilon^{2})^{\frac{p-2}{2}}\sigma_i u_i v_i^*,
\]
using the singular value decomposition $X= \sum_{i=1}^d \sigma_i u_i v_{i}^*$ in the last equality. 
Using the Kronecker sum inversion formula \cref{invformula}, we see that
 $\nabla g_{\epsilon}^p (X) = p\big[\widetilde{W}(X,\epsilon)X_{\vecc}\big]_{\mat}.$ The proof can be continued analogously to \citep[Lemma 5.2]{Daubechies10}.

\section{Proof of \cref{conv}} \label{sec_proof_conv_app}
 For statement (i) of the convergence result of \cref{algo1}, we use the following \emph{reverse triangle inequalities} implied by the strong Schatten-$p$ NSP:  Let $X,X' \in M_{d_1 \times d_2}$ such that $\Phi(X-X')=0$. Then
\begin{equation} \label{eq_schatten_2_NSPestimate}
\|X' - X\|^p_{F} \leq \frac{2^p \gamma_r^{1-p/2}}{r^{1 - p/2}} \frac{1}{1-\gamma_r} \left(\|X'\|^p_{S_p}-\|X\|^p_{S_p}+2\beta_{r}(X)_{S_p}\right),
\end{equation}
where $\beta_{r}(X)_{S_p}$ is defined in \cref{beta}.
This inequality can be proven using an adaptation of the proof of the corresponding result for $\ell_p$-minimization in \citep[Theorem 13]{Gao15} and the generalization of Mirksy's singular value inequality to concave functions \citep{Audenaert14,Foucart16}. Furthermore, the proof of the similar statement in \citep[Theorem 12]{Kabanava15} can be adapted to show \cref{eq_schatten_2_NSPestimate}.

The further part of the proof of (i) as well as (ii) follow analogously to \citep[Theorem 6.11]{Fornasier11} and \citep[Theorem 5.3]{Daubechies10} using the preliminary results deduced in \cref{sec_theoretical_analysis}.
Statement (iii) is a direct consequence of \cref{rate}, which is proven in \cref{sec_localconvrate}.
\end{appendix}

\bibliography{Literature}
\end{document}